\documentclass[11pt,reqno,twoside]{article}
\usepackage[margin=1in]{geometry}
\usepackage{amsmath,amssymb,amsfonts,amsthm,mathtools,mathrsfs,amscd}
\usepackage{graphicx}
\usepackage[all]{xy}
\usepackage{booktabs,blkarray,array}
\usepackage{microtype}
\usepackage{xcolor}
\usepackage{tikz}
\usepackage{tikz-cd}
\usetikzlibrary{arrows.meta,calc,decorations.markings,positioning}
\usepackage{hyperref}
\hypersetup{
  hidelinks,
  pdftitle={Affine G-multibiquandle modules for colored links I: representability of coloring functors and color-profile polynomials},
  pdfauthor={Sang Youl Lee},
  pdfkeywords={colored link, multibiquandle, associated biquandle, group ring, universal coloring, representing module, color-profile polynomial}
}
\numberwithin{equation}{section}
\tikzset{
  oriented/.style={postaction={decorate},decoration={markings,mark=at position .58 with {\arrow{Stealth}}}},
  link/.style={line width=.8pt},
  overpass/.style={white,line width=4.2pt,preaction={draw=white,line width=4.2pt}}
}

\newcommand{\Hom}{\operatorname{Hom}}
\newcommand{\coker}{\operatorname{coker}}
\newcommand{\im}{\operatorname{Im}}

\newcommand{\Spec}{\operatorname{Spec}}
\newcommand{\Fitt}{\operatorname{Fitt}}

\newcommand{\id}{\operatorname{id}}
\newcommand{\Ext}{\operatorname{Ext}}
\theoremstyle{plain}
\newtheorem{theorem}{Theorem}[section]
\newtheorem{lemma}[theorem]{Lemma}
\newtheorem{proposition}[theorem]{Proposition}
\newtheorem{corollary}[theorem]{Corollary}

\theoremstyle{definition}
\newtheorem{definition}[theorem]{Definition}
\newtheorem{example}[theorem]{Example}
\theoremstyle{remark}
\newtheorem{remark}[theorem]{Remark}
\newtheorem*{remark*}{Remark}
\usepackage{enumitem}

\theoremstyle{plain}
\newtheorem{introtheorem}{Theorem}

\begin{document}

\title{Affine $G$-multibiquandle modules for colored links I:
representability of coloring functors and color-profile polynomials}

\author{Sang Youl Lee\\[2pt]
\small\itshape Department of Mathematics and Institute of Mathematical Sciences\\
\small\itshape Pusan National University, Busan 46241, Republic of Korea\\
\small\ttfamily sangyoul@pusan.ac.kr}

\renewcommand\leftmark{\footnotesize S. Y. Lee}
\renewcommand\rightmark{\footnotesize Affine $G$-multibiquandle modules for colored links I}
\date{}
\maketitle

\begin{abstract}
Let $R$ be a commutative ring with identity, let $G$ be an abelian group,
and fix $u\in R[G]^\times$. For each component-labeled oriented link
diagram $(D,\sigma)$, we construct a finitely presented $R[G]$-module
$M_u(D,\sigma;R[G])$ with a universal affine coloring $U_{D,\sigma}$.
Every affine coloring with values in an $R[G]$-module factors uniquely
through $U_{D,\sigma}$ by a module homomorphism. Thus the module
represents the affine coloring functor. Each specified colored
Reidemeister move induces a unique module isomorphism compatible with
coloring transport, and the module isomorphism class is a colored-link
invariant. The representing modules and universal colorings commute
with arbitrary scalar extension, without a flatness assumption.
For classical links, simultaneous rescaling of the component labels
removes the common deformation unit.

For any multibiquandle, associated-biquandle colorings decompose into
fibers indexed by component labelings. For finite targets and palettes,
we define a color-profile polynomial recording the joint distribution of fiber sizes and
numbers of palette labels used, and prove its invariance
under oriented link isotopy. In the affine case, the representing
modules determine these fibers, and their specialized dimensions
determine the coloring counts over finite fields.
The color-profile polynomial is not determined, in general, by the
full unweighted distribution of coloring-fiber sizes. It therefore
retains information beyond the associated-biquandle coloring count.
\end{abstract}




\section{Introduction}\label{intro}

Quandle and biquandle colorings encode the local relations of a link
diagram in algebraic operations whose identities ensure Reidemeister
invariance \cite{Jo,Ma,FJK,EN}. Affine operations connect these
diagrammatic relations with linear algebra and module presentations.
For example, Traldi's multivariate Alexander colorings admit a
module-homomorphism description \cite{TraldiColor}. Group-indexed
operations and multiquandles extend the algebraic setting to links
with component labels \cite{IIJO,LS,KLS1,KLS2,Tu}, so that the
operation used at a crossing can depend on the label of its
over-component.

Within this setting, component labels and semiarc colors play
different roles. A component label selects an operation, whereas
a semiarc color is a value satisfying the resulting crossing
relations. Associated di(bi)quandles for two labels \cite{KLS2}
combine these two kinds of data in one ordinary algebraic coloring.
Representation-fiber enhancements for $G$-families \cite[Definition~6.2]{IN}
record the sizes of fibers over group representations.
On the common domain of finite $G$-families of biquandles with
$G$ abelian, the unweighted profile considered below agrees with
the classical-link restriction of that enhancement.
The present formulation allows arbitrary palettes and also records
the number of palette labels used.

These constructions suggest two questions. First, how can all affine
colorings for one fixed component labeling be represented by a module,
with the universal coloring, Reidemeister transport, and changes of
coefficients treated compatibly? Second, when the labelings are
collected into a finite profile, can recording the number of palette
labels used distinguish links with the same full unweighted
distribution of fiber sizes?
The first question concerns the coloring functor for every module
target and every homomorphism of targets. It differs from fixing
a $G$-Alexander biquandle and forming a cocycle enhancement
\cite{GAB}. The second concerns the information retained when
the labeling data are summarized numerically.

This paper answers these questions by constructing universal affine
colorings and studying the finite color profiles obtained from their
represented fibers. Let $R$ be a commutative ring with identity,
let $G$ be an abelian group, put $A=R[G]$, and fix $u\in A^\times$.
For an oriented diagram $D$ with a component labeling
$\sigma:\pi_0(D)\to G$, we construct the finitely presented module
$M_D=M_u(D,\sigma;R[G])$ from the presentation
\eqref{modl-pres-1}, together with the universal affine coloring
$U_{D,\sigma}$ of Definition~\ref{def:universal-affine-coloring}.
Every affine coloring $C$ of $(D,\sigma)$ with values in an
$A$-module $N$ factors uniquely as $C=f_C\circ U_{D,\sigma}$
through an $A$-linear map $f_C:M_D\to N$.
This identifies all such colorings with homomorphisms from $M_D$,
naturally in the target $N$.
The universal property also determines the unique module isomorphism
compatible with each prescribed colored Reidemeister move.
For the colored link $(L,\sigma)$ represented by $(D,\sigma)$,
the resulting invariant isomorphism class is denoted by
$M_u(L,\sigma;R[G])$ and is called its $(G,u)$-deformed Alexander
module (Definition~\ref{def:Gu-module}).

We prove that $M_D$ and $U_{D,\sigma}$ are compatible with arbitrary
scalar extension. The universal Laurent deformation
$\mathscr M_z(D,\sigma)$ over $A[z^{\pm1}]$, for an indeterminate $z$
(Definition~\ref{def:universal-deformation}), realizes changes of
coefficients and all unit choices of the common deformation
parameter through a single presentation.
For classical links, simultaneous rescaling of the component labels
removes that common unit, giving an explicit linear realization
of the biquandle-to-quandle reduction \cite{IT,Murao}.

For a nonempty palette $\Omega$ and an $\Omega$-multibiquandle $X$
(Definition~\ref{def-n-dibiqdle}), we then identify the colorings
by the associated biquandle $\operatorname{As}_\Omega(X)$ with
fibers indexed by component labelings $\sigma:\pi_0(L)\to\Omega$
(Theorem~\ref{thm:associated-biquandle-fibers}).
For finite $X$ and $\Omega$, we define the color-profile polynomial
$\Phi_X^\Omega(L;\alpha,\xi)$
(Definition~\ref{def:color-profile-polynomial}).
It records coloring-fiber sizes together with the numbers of
palette labels used and is an invariant of uncolored oriented links.
In the affine case, the modules $M_u(L,\tau\circ\sigma;A)$,
where $\tau:\Omega\to G$ assigns coefficient actions to palette
labels, represent the individual fibers
(Proposition~\ref{prop:associated-affine-fibers}).
Their specialized dimensions determine the finite-field coloring
counts (Corollary~\ref{cor:finite-field-color-profile}).
The polynomial $\Phi_X^\Omega(L;\alpha,\xi)$ is not determined,
in general, by its unweighted specialization
$\Phi_X^\Omega(L;1,\xi)$, and therefore retains information beyond
the associated-biquandle coloring count.
We now state the structural results in the order in which they
are established in the paper.

With $A=R[G]$ and $u\in A^\times$ as above, the affine operations
on an $A$-module $N$ are
\begin{equation}\label{eq:intro-affine-operations}
 x\mathbin{\underline\triangleright_g}y=gx+(u-g)y,
 \qquad
 x\mathbin{\overline\triangleright_g}y=ux
 \qquad(g\in G).
\end{equation}
These are the operations in \eqref{eq1-mdl-gpring};
Theorem~\ref{thm-multibqdl-mdl} proves that they make $N$ an affine
$G$-multibiquandle.
The index at a crossing is the label of its over-component.
We write $\operatorname{Col}_N^G(D,\sigma)$ for the corresponding
set of affine colorings, using the coloring convention of
Definition~\ref{def-mbqd-coloring}.

The module $M_D=M_u(D,\sigma;A)$ is generated by the semiarcs
with the two affine relations at each crossing.
Each crossingless component contributes one free generator.
The universal affine coloring is the assignment
$U_{D,\sigma}(x)=[x]$ of each semiarc to its class in $M_D$.
Its defining property is the following.

\begin{introtheorem}[Universal affine colorings]
\label{thm:intro-universal-coloring}
The module $M_D$ is finitely presented, and composition with
$U_{D,\sigma}$ gives a bijection
\begin{equation}\label{eq:intro-representability}
 \Hom_A(M_D,N)\xrightarrow{\;\cong\;}
 \operatorname{Col}_N^G(D,\sigma),
 \qquad f\longmapsto f\circ U_{D,\sigma},
\end{equation}
natural in every $A$-module $N$.
Equivalently, every affine coloring $C$ factors uniquely as
$C=f_C\circ U_{D,\sigma}$ through an $A$-linear map
$f_C:M_D\to N$.
The pair $(M_D,U_{D,\sigma})$ is unique up to a unique isomorphism
preserving the universal coloring.

For each specified colored Reidemeister move
$r:(D,\sigma)\to(D',\sigma')$, the coloring extension maps determine
a unique isomorphism $\mu_r:M_{D'}\to M_D$ satisfying
$\mu_r\circ U_{D',\sigma'}
=\Phi^{r,\mathrm{aff}}_{M_D}(U_{D,\sigma})$,
where $\Phi^{r,\mathrm{aff}}$ denotes the affine coloring transport
associated with $r$.
Consequently, the isomorphism class of $M_u(D,\sigma;A)$ is an
invariant of the colored oriented link.
\end{introtheorem}

Theorems~\ref{thm:universal-affine-coloring} and
\ref{thm:reidemeister-representing-module} establish this statement.
Naturality includes every homomorphism of target modules.
For a fixed sequence of moves, the representing isomorphisms compose
as specified in Proposition~\ref{prop:move-sequence-compatibility}.
The uniqueness assertion is relative to the prescribed coloring
transport.
Appendix~\ref{app:matrix-reductions} supplies a second proof of module
invariance: explicit R1, R2, and R3 row and column operations preserve
the exterior presentation and eliminate local generators using unit
pivots. Neither proof requires the differences $u-g$ to be invertible.

The coefficient ring may have torsion or zero divisors.
The next result describes changes of coefficients and the simultaneous
specialization of the common deformation parameter.
Let $z$ be an indeterminate, put $A_z=A[z^{\pm1}]$, and write
$\mathscr M_z(D,\sigma)=M_z(D,\sigma;A_z)$ for the universal Laurent
deformation of Definition~\ref{def:universal-deformation},
obtained by replacing $u$ by $z$ in the crossing relations.

\begin{introtheorem}[Scalar extension and universal evaluation]
\label{thm:intro-base-change}
For every homomorphism $\varphi:A\to S$ of commutative rings,
there is a canonical isomorphism
\begin{equation}\label{eq:intro-base-change}
 M_u(D,\sigma;A)\otimes_{A,\varphi}S
 \cong M_{\varphi(u)}(D,\sigma_S;S),
 \qquad [x]\otimes s\longmapsto s[x]_S,
\end{equation}
where $\sigma_S=\varphi\circ\sigma$, with labels interpreted as
units of $S$. No flatness assumption is required.

For every $v\in S^\times$, let
$\widetilde\varphi_v:A_z\to S$ be the unique homomorphism extending
$\varphi$ and sending $z$ to $v$. Then
\begin{equation}\label{eq:intro-universal-evaluation}
 \mathscr M_z(D,\sigma)
 \otimes_{A_z,\widetilde\varphi_v}S
 \cong M_v(D,\sigma_S;S).
\end{equation}
Both isomorphisms preserve the universal affine colorings and
commute with the transport associated with each specified colored
Reidemeister move.
\end{introtheorem}

The scalar-extension statement is Theorem~\ref{thm:base-change},
with transport compatibility in
Proposition~\ref{prop:base-change-reidemeister}.
Corollary~\ref{cor:universal-evaluation} proves the evaluation
statement. In particular, taking $S=A$, $\varphi=\operatorname{id}_A$,
and $v=u$ recovers $M_u(D,\sigma;A)$ by evaluating $z$ at $u$.
Thus changes of coefficients and all unit choices of the common
parameter arise from one Laurent presentation, with the universal
colorings retained.

For classical links, the common deformation unit can also be
removed by changing the component labels simultaneously.
For a commutative ring $S$, a unit $v\in S^\times$, and a labeling
$\lambda:\pi_0(L)\to S^\times$, the notation $M_v(L,\lambda;S)$
uses the same affine relations with unit-valued component labels.
Its diagram independence follows from the scalar-extension
construction in Subsection~\ref{subsec:base-change}.

\begin{introtheorem}[Removal of the common deformation unit]
\label{thm:intro-normalization}
Let $S$ be a commutative ring, let $v\in S^\times$, and let
$\lambda:\pi_0(L)\to S^\times$ label a classical oriented link.
There is an isomorphism
\begin{equation}\label{eq:intro-unit-removal}
 M_v(L,\lambda;S)\cong M_1(L,v^{-1}\lambda;S).
\end{equation}
For a link with $\mu$ components labeled by independent variables
$a_1,\ldots,a_\mu$, the substitutions $t_i=a_i/z$ express the
universal deformation over
$R[a_1^{\pm1},\ldots,a_\mu^{\pm1},z^{\pm1}]$
as Laurent scalar extension of the parameter-$1$ module from
$R[t_1^{\pm1},\ldots,t_\mu^{\pm1}]$.
If $L$ is nonempty, $M_v(L,\lambda;S)$ has a free rank-one
direct summand.
\end{introtheorem}

Theorem~\ref{thm:all-color-reduction} proves this normalization
explicitly. On an $n$-braid, the diagonal matrix
$\operatorname{diag}(1,v,\ldots,v^{n-1})$ conjugates every crossing
transition to its parameter-$1$ form with all component labels
divided by $v$. The same change of coordinates respects the braid
closure. The rescaled labels are units of $S$, even when they
do not belong to the originally chosen subgroup $G$.
Thus the common unit parametrizes the normalized classical-link
theory through simultaneous rescaling of the labels.
At parameter $1$, the relators have coefficient sum zero, so the
augmentation $\varepsilon:M_1\to S$, defined by
$\varepsilon([x])=1$ on each semiarc class, is well-defined.
For a nonempty link, choosing one semiarc class $[x]$ gives a section
$s\mapsto s[x]$, and hence the free rank-one direct summand.

We next collect the coloring fibers while retaining their
component-label data.
Let $\Omega$ be a nonempty palette and let $X$ be an
$\Omega$-multibiquandle, with operations denoted by
$\underline\triangleright_i$ and $\overline\triangleright_i$.
The associated biquandle $\operatorname{As}_\Omega(X)$ is defined
on $X\times\Omega$ by the operations
$(x,i)\underline\triangleright(y,j)
=(x\underline\triangleright_j y,i)$ and
$(x,i)\overline\triangleright(y,j)
=(x\overline\triangleright_j y,i)$;
see \eqref{eq:associated-biquandle-operations}.

\begin{introtheorem}[Associated biquandle and color-profile polynomial]
\label{thm:intro-i-color-profiles}
These operations define a biquandle $\operatorname{As}_\Omega(X)$,
and there is a bijection, natural in $X$,
\begin{equation}\label{eq:intro-i-fiber-decomposition}
 \operatorname{Col}_{\operatorname{As}_\Omega(X)}(D)
 \cong
 \coprod_{\sigma:\pi_0(D)\to\Omega}
 \operatorname{Col}_X^\Omega(D,\sigma).
\end{equation}
If $X$ and $\Omega$ are finite, the color-profile polynomial
\begin{equation}\label{eq:intro-i-profile}
 \Phi_X^\Omega(L;\alpha,\xi)
 =
 \sum_{\sigma:\pi_0(L)\to\Omega}
 \alpha^{|\operatorname{im}\sigma|}
 \xi^{|\operatorname{Col}_X^\Omega(L,\sigma)|}
\end{equation}
of Definition~\ref{def:color-profile-polynomial} is an invariant
of the uncolored oriented link.
Its specialization $\Phi_X^\Omega(L;1,\xi)$ records the full
unweighted distribution of coloring-fiber sizes, and
$\left.
\partial\Phi_X^\Omega(L;1,\xi)/\partial\xi
\right|_{\xi=1}$
is the associated-biquandle coloring count.
\end{introtheorem}

Theorem~\ref{thm:associated-biquandle-fibers} and
Corollary~\ref{cor:color-profile-invariance} prove these statements.
For a finite palette $\Omega$, a map $\tau:\Omega\to G$, and a
finite affine $A$-module target $N$, the fiber over $\sigma$ is
naturally $\Hom_A(M_u(L,\tau\circ\sigma;A),N)$
(Proposition~\ref{prop:associated-affine-fibers}).
Thus each such fiber is obtained from the universal affine coloring
for the corresponding component labeling.
The map $\tau$ need not be injective: the exponent
$|\operatorname{im}\sigma|$ is computed in $\Omega$, so distinct
palette elements remain distinct labels even when they specify
the same coefficient action.

If $\lambda:A\to k$ specializes to a finite field of order $Q$,
then the fiber with target $k$ has cardinality
$Q^{d_{\lambda,\sigma}}$, where
$d_{\lambda,\sigma}
=\dim_k\bigl(M_u(L,\tau\circ\sigma;A)
\otimes_{A,\lambda}k\bigr)$.
Consequently, the finite-field Betti profile
$\mathcal B_\lambda^\Omega(L;\alpha,t)
=\sum_\sigma
\alpha^{|\operatorname{im}\sigma|}t^{d_{\lambda,\sigma}}$
of Corollary~\ref{cor:finite-field-color-profile} determines the
corresponding color-profile polynomial through the termwise
transformation $\alpha^r t^d\mapsto\alpha^r\xi^{Q^d}$.
For fixed $Q$, this transformation is injective, so the two
polynomials determine each other.

The color-profile polynomial retains information not determined,
in general, by the full unweighted distribution of coloring-fiber
sizes. The following pair demonstrates this strict refinement
for a fixed finite target and palette.

Take $X=\mathbb F_3$ and $\Omega=\{0,1\}$, with operations
$x\underline\triangleright_0y=x$,
$x\underline\triangleright_1y=2x+2y$, and
$x\overline\triangleright_i y=x$ for $i=0,1$.
Let $L_1$ and $L_2$ be the closures of the downward oriented braids
$\sigma_1^4$ and $\sigma_1^3\sigma_2^2$, respectively, using the
crossing convention of \eqref{eq:braid-transitions}.
Both links have two components, and
\begin{equation}\label{eq:intro-i-strict-profiles}
 \begin{aligned}
 \Phi_X^\Omega(L_1;\alpha,\xi)
  &=\alpha(\xi^9+\xi^3)+2\alpha^2\xi^9,\qquad
 \Phi_X^\Omega(L_2;\alpha,\xi)
  =2\alpha\xi^9+\alpha^2(\xi^3+\xi^9).
 \end{aligned}
\end{equation}
Setting $\alpha=1$ gives $\xi^3+3\xi^9$ in both cases,
and the associated-biquandle coloring counts are both
$3+3\cdot9=30$.
However, the coefficients of $\alpha\xi^3$ in the two weighted
profiles are $1$ and $0$, respectively.
For $L_1$, the unique fiber of size $3$ comes from a labeling
using one palette element; for $L_2$, it comes from a labeling
using two. This identifies precisely the information lost on
passing to the unweighted distribution.
Example~\ref{ex:strict-color-profile} gives every transition matrix,
closure matrix, and coloring kernel needed for the computation.

Sections~\ref{sec-bql} and~\ref{multi-biqdl} develop the indexed
algebra and its examples.
Section~\ref{sec-fundbql-dl} proves the colored Reidemeister extension
statements, and Section~\ref{sect-abq} constructs the representing
module and its universal coloring.
Section~\ref{sec:parameters} proves the coefficient and parameter
results, using the braid-closure presentation of
Lemma~\ref{lem:braid-closure}.
Section~\ref{sec:profiles} proves the fiber decomposition and
profile statements and gives the strict-refinement calculation.
Section~\ref{sec:series-i-conclusion} discusses the interpretation
of the representing modules and the information retained or lost
in their finite color profiles.
Appendix~\ref{app:matrix-reductions} contains the independent matrix
proof of Reidemeister invariance.
Part~II~\cite{LeeSeriesII} studies incidence assembly and the integral
comparison with Alexander modules.
All assertions in the four main statements above are proved here.


\section{Quandles and multiquandles}\label{sec-bql}

In this section, we review basic terminology that will be used in the subsequent sections, introduce the notion of a $G$-multiquandle for a group $G$, and present several classes of $G$-multiquandles. We begin with the definition of a quandle.

By a binary operation $\triangleright$ on a set $X$, we mean an arbitrary mapping $\triangleright : X \times X \to X$. The image of a pair $(x, y) \in X \times X$ under this binary operation is denoted $x\triangleright y$. 

\begin{definition}\label{def-qdl}\cite{Jo,Ma}
A {\it quandle} is a set $X$ equipped with a binary operation $\rhd$ on $X$ satisfying the following conditions:
\begin{description}
\item[Q1.] For all $x \in X$, $x \rhd x = x$.
\item[Q2.] For all $x, y \in X$, there exists a unique element $z \in X$ such that $x = z \rhd y$. Equivalently, there exists a binary operation $\rhd^{-1}$ on $X$ such that for all $x, y \in X$,
\((x\rhd y) \rhd^{-1} y = x\) and \((x \rhd^{-1} y) \rhd y = x\).
\item[Q3.] For all $x, y, z \in X$, $(x \rhd y )\rhd z = (x \rhd z)\rhd (y \rhd z)$.
\end{description}
\end{definition}

A quandle $X$ with a binary operation $\rhd$ is denoted by a pair $(X, \rhd)$.  
In \textbf{Q2}, the operation $\rhd^{-1}$ is called the {\it dual operation} of $\rhd$, and $x = z\rhd y$ if and only if $z = x\rhd^{-1} y$.  
A \emph{rack} satisfies \textbf{Q2}--\textbf{Q3}, without requiring \textbf{Q1}.
We note that $(X,\rhd^{-1})$ is also a quandle, called the {\it dual quandle} of $(X,\rhd)$.  
A quandle with self-dual operation, i.e., $\rhd^{-1} = \rhd$, is called a {\it kei} (or an {\it involutory quandle}).  
We refer the reader to \cite{EN,Tu} for further details.  
Here we recall two fundamental examples of quandles.

\begin{example}\label{eg1-qdles}
($n$-conjugation quandle)  
For any group $G$ and each integer $n$, $G$ forms a quandle under the binary operation $x\triangleright_n y = y^{-n}xy^n$ for all $x, y \in G$, with the dual operation $x\triangleright^{-1}_n y = y^{n}xy^{-n}$. 

\noindent(Alexander quandle)  
Any $\mathbb Z[t^{\pm 1}]$-module $M$ is a quandle under the binary operation $x\triangleright y = tx + (1-t)y$ for all $x, y \in M$, with the dual operation $x\triangleright^{-1} y = t^{-1}x + (1-t^{-1})y$.
\end{example}

A binary operation $\triangleright$ on $X$ is said to be {\it non-degenerate} if for any $y\in X$, the map $S_y:X\to X$ defined by $S_y(x)=x\triangleright y$ is bijective, or equivalently, $\triangleright$ satisfies condition \textbf{Q2} above.  
The inverse map of $S_y$ is given by $S_y^{-1}(x) = x\triangleright^{-1} y$ for all $x\in X$.  

In \cite{Tu}, Turaev introduced a multirack and a multiquandle.  
Precisely, a multirack is defined to be a pair consisting of a nonempty set $U$ and a collection $\{\triangleright_s\}_{s\in S}$ of non-degenerate binary operations on $U$ such that for any $s,t\in S$ and $u,v,w \in U$,
\(
(u\triangleright_s v)\triangleright_t w = (u\triangleright_t w)\triangleright_s (v \triangleright_t w).
\)
A multirack $(U, \{\triangleright_s\}_{s\in S})$ is a multiquandle if $u\triangleright_s u = u$ for all $u\in U$ and $s\in S$.  
Multiquandles are algebraic structures associated with the fundamental groups of arbitrary pairs consisting of a path-connected topological space and its path-connected subspace \cite{Tu}.  
Throughout this paper, we write a multirack and a multiquandle without hyphens, and for convenience, we reintroduce their definitions below, followed by the notion of a {\it group $G$-multiquandle}.

\begin{definition}\label{def-n-mqdle}
Let $X$ and $\Omega$ be nonempty sets, and let $\mathcal Q(\Omega)=\{\triangleright_i\mid i\in \Omega\}$ be an indexed family of non-degenerate binary operations $\triangleright_i : X\times X \to X$.  
A {\it $\Omega$-multiquandle} (resp. {\it $\Omega$-multirack}) is a triple $(X,\Omega,\mathcal Q(\Omega))$ satisfying:
\begin{description}  
\item[MQ1.] For each $i\in \Omega$, $(X,\triangleright_i)$ is a quandle (resp. rack).
\item[MQ2.] For any $x,y,z\in X$ and $i,j\in \Omega$,
\(
(x\triangleright_i y)\triangleright_j z
=(x\triangleright_j z)\triangleright_i (y\triangleright_j z).
\)
\end{description}
We call $X$ the {\it underlying set} and $\Omega$ the {\it set of colors}.  
In particular, if $\Omega=\{1,2,\ldots,n\}$ ($n \geq 1$), then the structure is called an {\it $n$-multiquandle} (resp. {\it $n$-multirack}).  
We write $(X,\{\triangleright_i\}_{i\in\Omega})$ for an $\Omega$-multiquandle, and $(X,\{\triangleright_i\}_{i=1}^n)$ for an $n$-multiquandle (resp. $n$-multirack). 
\end{definition}

If $\Omega'$ is a nonempty subset of $\Omega$, restricting every
index in \textbf{MQ1} and \textbf{MQ2} from $\Omega$ to $\Omega'$ leaves the
same identities and bijections. Hence
$(X,\{\triangleright_i\}_{i\in\Omega'})$ is an
$\Omega'$-multiquandle.  
A $1$-multiquandle is a quandle, and a $2$-multiquandle $(X,\{\triangleright_i\}_{i=1}^2)$ is a diquandle \cite{LS}. 

\begin{example}\label{exp-00}
(Trivial multiquandle)  
Let $X$ and $\Omega$ be nonempty sets. For each $i\in\Omega$, define $\triangleright_i$ on $X$ by $x\triangleright_i y= x$ for all $x, y \in X$. Then $(X,\{\triangleright_i\}_{i\in\Omega})$ is an $\Omega$-multiquandle.

\noindent(Quandle multiquandle)  
Let $(X,\triangleright)$ be a quandle. For each $i\in\Omega$, define $\triangleright_i$ on $X$ by $x\triangleright_i y= x\triangleright y$. Then $(X,\{\triangleright_i\}_{i\in\Omega})$ is an $\Omega$-multiquandle.

\noindent($N$-conjugation multiquandle)  
Let $G$ be a group. For each integer $n$, define $\triangleright_n$ on $G$ by $x\triangleright_n y= y^{-n}xy^n$. Then $(G,\{\triangleright_n\}_{n\in\mathbb Z})$ is a $\mathbb Z$-multiquandle \cite{Tu}.
Indeed, $x\triangleright_n x=x$ and the inverse translation is
$x\triangleright_n^{-1}y=y^nxy^{-n}$. For all $n,m\in\mathbb Z$,
\[
\begin{aligned}
(x\triangleright_n y)\triangleright_m z
 &=z^{-m}y^{-n}xy^nz^m,\\
(x\triangleright_m z)\triangleright_n(y\triangleright_m z)
 &=(z^{-m}yz^m)^{-n}(z^{-m}xz^m)(z^{-m}yz^m)^n
 =z^{-m}y^{-n}xy^nz^m.
\end{aligned}
\]
These identities prove \textbf{MQ1} and \textbf{MQ2}.
\end{example}

\begin{definition}\label{def-G-mqdle}
Let $X$ and $G$ be nonempty sets. A $G$-multiquandle (resp. $G$-multirack) $(X,\{\triangleright_g\}_{g\in G})$ is called a {\it group $G$-multiquandle} (resp. {\it group $G$-multirack}) if $G$ is a group.
\end{definition}

\begin{theorem}\label{thm-propt}
Let $G$ be a group and let
$(X,\{\triangleright_g\}_{g\in G})$ be a group-indexed
$G$-multiquandle.  If $H\le G$, then
$(X,\{\triangleright_h\}_{h\in H})$ is an $H$-multiquandle.
\end{theorem}

\begin{proof}
The quandle and mixed distributivity axioms for indices in $H$ are among
the corresponding axioms for indices in $G$.
\end{proof}

\begin{remark}\label{rem:indexed-not-group-law}
The adjective ``group'' in a group $G$-multiquandle refers to the indexing
set.  In general one cannot define a group law on the set of operations by
$\triangleright_g\triangleright_h=\triangleright_{gh}$: two distinct
indices may determine the same operation while their products do not.
Such a law is available only after an additional compatibility hypothesis
(for example, when equality of indexed operations is a congruence on
$G$).  No such hypothesis is part of Definition~\ref{def-G-mqdle}.
\end{remark}

\begin{example}\label{exa-conj-mq}
(Conjugation $\mathbb Z$-multiquandle)  
Let $G$ be a group. For each integer $n\in\mathbb Z$, define $\triangleright_n$ on $G$ by $x\triangleright_n y= y^{-n}xy^n$. Then $(G,\{\triangleright_n\}_{n\in\mathbb Z})$ is a group $\mathbb Z$-multiquandle by Example \ref{exp-00}. 
\end{example}

\begin{definition}\cite{IIJO}\label{defn-gfam-q}
Let $G$ be a group with identity $e$.  
A {\it $G$-family of quandles} is a pair $(X, \{\triangleright^g\}_{g\in G})$ consisting of a nonempty set $X$ and a family $\{\triangleright^g\}_{g\in G}$ of binary operations $\triangleright^g$ on $X$ satisfying:
\begin{itemize}
\item[(i)] For any $x \in X$ and $g \in G$, $x \triangleright^g x = x$.
\item[(ii)] For any $x, y \in X$ and $g,h \in G$,
\(
x\triangleright^{gh}y=(x\triangleright^g y)\triangleright^h y, \quad x \triangleright^e y = x.
\)
\item[(iii)] For any $x, y, z \in X$ and $g, h \in G$,
\(
(x\triangleright^g y)\triangleright^{h}z = (x \triangleright^h z)\triangleright^{h^{-1}gh} (y\triangleright^h z).
\)
\end{itemize}
\end{definition}

\begin{theorem}\label{thm-m-qdl}
Let $G$ be a commutative group and let $(X,\{\triangleright^g\}_{g\in G})$ be a $G$-family of quandles.  
Then $(X,\{\triangleright_g\}_{g\in G})$ is a group $G$-multiquandle such that $\mathcal Q(G)=\{\triangleright_g\mid g\in G\}$ is a commutative group, where $\triangleright_g=\triangleright^g$.
\end{theorem}

\begin{proof}
We verify that $(X,\{\triangleright_g\}_{g\in G})$ satisfies \textbf{MQ1} and \textbf{MQ2} of Definition \ref{def-n-mqdle}.  
For fixed $g$, \textbf{Q1} of Definition~\ref{def-qdl} is
Definition~\ref{defn-gfam-q}(i). For \textbf{Q2}, condition (ii) gives
$(x\triangleright^g y)\triangleright^{g^{-1}}y=x$ and the reverse
composition identity. Condition (iii) with $h=g$ gives \textbf{Q3},
since $g^{-1}gg=g$. Thus \textbf{MQ1} holds; compare
\cite[Prop.~2.1]{IIJO}.  
If $G$ is commutative, then by Definition \ref{defn-gfam-q}(iii), for any $x, y, z \in X$ and $g, h \in G$,
\begin{align*}
(x\triangleright_g y)\triangleright_{h}z &= (x\triangleright^g y)\triangleright^{h}z =(x \triangleright^h z)\triangleright^{h^{-1}gh} (y\triangleright^h z)\\
&= (x \triangleright^h z)\triangleright^{g} (y\triangleright^h z)=(x \triangleright_h z)\triangleright_{g} (y\triangleright_h z).
\end{align*}
Hence \textbf{MQ2} holds, and thus $(X,\{\triangleright_g\}_{g\in G})$ forms a group $G$-multiquandle.  
It remains to justify the final assertion, which uses the additional
$G$-family axiom.  If $\triangleright_g=\triangleright_{g'}$ and
$\triangleright_h=\triangleright_{h'}$, then
Definition~\ref{defn-gfam-q}(ii) gives
\(x\triangleright^{gh}y =(x\triangleright^g y)\triangleright^h y =(x\triangleright^{g'}y)\triangleright^{h'}y =x\triangleright^{g'h'}y\).
Thus multiplication of indexed operations is well defined.  Associativity,
the identity $\triangleright_e$, and the inverse
$\triangleright_{g^{-1}}$ follow from the corresponding facts in $G$;
commutativity follows because $G$ is abelian.
\end{proof}

\begin{theorem}\label{thm:multiquandle-to-quandle-family}
Let $G$ be a group with identity $e$ and let $(X,\{\triangleright_g\}_{g\in G})$ be a group $G$-multiquandle satisfying, for any $x, y \in X$ and $g,h \in G$,  
\begin{equation}\label{g-fam-cond2}
(x\triangleright_g y)\triangleright_h y=x\triangleright_{gh}y, \qquad x \triangleright_e y = x.
\end{equation}  
Then $(X,\{\triangleright^g\}_{g\in G})$ is a $G$-family of quandles, where $\triangleright^g=\triangleright_g$.
\end{theorem}

\begin{proof}
Since $(X,\triangleright_g)$ is a quandle for all $g\in G$, we have $x\triangleright^g x=x$ for all $x\in X$ and $g\in G$, satisfying (i) of Definition \ref{defn-gfam-q}.  
Condition (ii) follows from \eqref{g-fam-cond2} with $\triangleright^g=\triangleright_g$.  
To verify (iii), note from \eqref{g-fam-cond2} that for any $x, y \in X$ and $g, h \in G,$
\(
x\triangleright_{gh}y=(x\triangleright_g y)\triangleright_h y=(x\triangleright_h y)\triangleright_g (y\triangleright_h y)=(x\triangleright_h y)\triangleright_g y=x\triangleright_{hg}y.
\)
Hence, for all $x, y, z \in X$ and $g, h \in G$,
\begin{align*}
(x\triangleright^g y)\triangleright^{h}z &= (x\triangleright_g y)\triangleright_{h}z=(x \triangleright_h z)\triangleright_{g} (y\triangleright_h z)\\
&=(x \triangleright_h z)\triangleright_{(hh^{-1})g} (y\triangleright_h z)
=(x \triangleright_h z)\triangleright_{h(h^{-1}g)} (y\triangleright_h z)\\
&=(x \triangleright_h z)\triangleright_{(h^{-1}g)h} (y\triangleright_h z)
=(x \triangleright^h z)\triangleright^{h^{-1}gh} (y\triangleright^h z).
\end{align*}
The penultimate equality uses the commutativity of the indexed operations
proved in the preceding display, not commutativity of the group $G$.
This completes the proof.
\end{proof}

\begin{example}
Let $(X,\triangleright)$ be a quandle.  
For each $y\in X$, let $S_y : X \to X$ be the bijection defined by $S_y(x) = x\triangleright y$.  
For any integer $n\in \mathbb Z$, define the binary operation $\triangleright^n : X \times X \to X$ by $x \triangleright^0 y=x$ and, for $n\neq 0$, $x \triangleright^n y = S^n_y(x)$.  
Then $(X,\{\triangleright^n\}_{n\in \mathbb Z})$ is a $\mathbb Z$-family of quandles.
Indeed, $S_x^n(x)=x$ and $S_y^{n+k}=S_y^kS_y^n$ give the diagonal
and composition axioms for all integers $n,k$. The quandle
distributivity law makes $S_z$ an automorphism. Its inverse and every
integer power are also automorphisms, so
\(S_z^kS_yS_z^{-k}=S_{S_z^k(y)}\) and \(S_z^kS_y^nS_z^{-k}=S_{S_z^k(y)}^n\).
Consequently,
$(x\triangleright^n y)\triangleright^k z
=S_z^kS_y^n(x)
=S_{S_z^k(y)}^nS_z^k(x)
=(x\triangleright^k z)\triangleright^n(y\triangleright^k z)$,
which is the remaining family axiom.
If there exists a positive integer $m$ such that $S^m_y ={\rm id}_X$ for all $y \in X$, then $(X,\{\triangleright^i\}_{i\in \mathbb Z_m})$ is a $\mathbb Z_m$-family of quandles \cite[Prop.~2.3]{IIJO}.
Here $S_y^{n+m}=S_y^n$ shows directly that the operations and the
identities descend modulo $m$.  
Hence, by Theorem \ref{thm-m-qdl}, $(X,\{\triangleright_n\}_{n\in \mathbb Z})$ with $\triangleright_n=\triangleright^n$ is a group $\mathbb Z$-multiquandle and $(X,\{\triangleright_i\}_{i\in \mathbb Z_m})$ with $\triangleright_i=\triangleright^i$ is a group $\mathbb Z_m$-multiquandle, where $\mathbb Z_m = \mathbb Z/m\mathbb Z$.
\end{example}

\section{Biquandles and multibiquandles}\label{multi-biqdl}

In this section, we introduce the notions of an $\Omega$-multibirack and an $\Omega$-multibiquandle for a set $\Omega$, and the notions of a $G$-multibirack and a $G$-multibiquandle for a group $G$. Then we give general constructions for several special classes of these algebraic structures along with specific examples. We begin by recalling the definition of a biquandle.

\begin{definition}\label{def-biqdle}
A {\it biquandle} is a set $X$ equipped with two binary operations $\underline{\triangleright}$ and $\overline{\triangleright}$ on $X$ satisfying the following conditions:
\begin{description}  
\item[BQ1.] For any $x\in X,$ $x~\underline{\triangleright}~x=x~\overline{\triangleright}~x$.
\item[BQ2.] Both operations $\underline{\triangleright}$ and $\overline{\triangleright}$ are right invertible. Equivalently, there exist binary operations $\underline{\triangleright}^{-1}$ and $\overline{\triangleright}^{-1}$ on $X$ such that for all $x, y\in X$,
\[
({x}~\underline{\triangleright}~{y})~\underline{\triangleright}^{-1}y=x,\quad
({x}~\overline{\triangleright}~{y})~\overline{\triangleright}^{-1}y=x,\quad
{(x~\underline{\triangleright}^{-1}y)}~\underline{\triangleright}~{y}=x,\quad
{(x~\overline{\triangleright}^{-1}y)}~\overline{\triangleright}~{y}=x.
\]
\item[BQ3.] The map $H:X\times X\to X\times X$ defined by $H(x,y)= ({y}~\overline{\triangleright}~{x},{x}~\underline{\triangleright}~{y})$ is bijective.
\item[BQ4.] For any $x,y,z\in X,$ 
\begin{align*}
&{({x}~\underline{\triangleright}~{y})}~\underline{\triangleright}~{({z}~\underline{\triangleright}~{y})}={({x}~\underline{\triangleright}~{z})}~\underline{\triangleright}~{({y}~\overline{\triangleright}~{z})},\\
&{({x}~\underline{\triangleright}~{y})}~\overline{\triangleright}~{({z}~\underline{\triangleright}~{y})}={({x}~\overline{\triangleright}~{z})}~\underline{\triangleright}~{({y}~\overline{\triangleright}~{z})},\\
&{({x}~\overline{\triangleright}~{y})}~\overline{\triangleright}~{({z}~\overline{\triangleright}~{y})}={({x}~\overline{\triangleright}~{z})}~\overline{\triangleright}~{({y}~\underline{\triangleright}~{z})}.
\end{align*}
\end{description}
We denote this biquandle by $(X,\underline{\triangleright}, \overline{\triangleright})$ or simply by $X$.
We call \textbf{BQ1} diagonal compatibility, \textbf{BQ2} right
invertibility, \textbf{BQ3} switch bijectivity, and \textbf{BQ4} the
exchange laws; the proofs below cite these axiom numbers explicitly.
Here a \emph{birack} satisfies \textbf{BQ2}--\textbf{BQ4}, without
requiring \textbf{BQ1}; compare \cite{FJK}. 
\end{definition}

If the operation $\overline{\triangleright}$ is trivial, i.e., $x~\overline{\triangleright}~y=x$ for any $x,y\in X$, then a biquandle $(X,\underline{\triangleright}, \overline{\triangleright})$ may be regarded as a quandle $(X,\underline{\triangleright})$.

\begin{example}\label{exp-3}
(Wada biquandles)
Let $G$ be a group. Define $\underline{\triangleright},\overline{\triangleright}:G\times G\to G$ by 
$x~\underline{\triangleright}~y=y^{-1}x^{-1}y$ and $x~\overline{\triangleright}~y=y^{-2}x$ for any $x,y\in G$. Then $(G,\underline{\triangleright},\overline{\triangleright})$ is a biquandle.
For \textbf{BQ1}, both diagonal values are $x^{-1}$.
The equations $w=y^{-1}x^{-1}y$ and $w=y^{-2}x$ give
$x=yw^{-1}y^{-1}$ and $x=y^2w$, respectively, proving
\textbf{BQ2}. For \textbf{BQ3}, write
$H(x,y)=(r,s)=(x^{-2}y,y^{-1}x^{-1}y)$.
Then $y=x^2r$ and $s=r^{-1}x^{-1}r$, whence
$H^{-1}(r,s)=(rs^{-1}r^{-1},rs^{-2})$.
For \textbf{BQ4}, the two sides of the three exchange laws simplify as follows:
\begin{align*}
(x\underline{\triangleright}y)
 \underline{\triangleright}(z\underline{\triangleright}y)
 &=(y^{-1}zy)(y^{-1}xy)(y^{-1}z^{-1}y)
 =y^{-1}zxz^{-1}y,\\
(x\underline{\triangleright}z)
 \underline{\triangleright}(y\overline{\triangleright}z)
 &=(y^{-1}z^2)(z^{-1}xz)(z^{-2}y)
 =y^{-1}zxz^{-1}y,\\
(x\underline{\triangleright}y)
 \overline{\triangleright}(z\underline{\triangleright}y)
 &=(y^{-1}z^2y)(y^{-1}x^{-1}y)
 =y^{-1}z^2x^{-1}y,\\
(x\overline{\triangleright}z)
 \underline{\triangleright}(y\overline{\triangleright}z)
 &=(y^{-1}z^2)(x^{-1}z^2)(z^{-2}y)
 =y^{-1}z^2x^{-1}y,\\
(x\overline{\triangleright}y)
 \overline{\triangleright}(z\overline{\triangleright}y)
 &=(z^{-1}y^2)(z^{-1}y^2)y^{-2}x
 =z^{-1}y^2z^{-1}x,\\
(x\overline{\triangleright}z)
 \overline{\triangleright}(y\underline{\triangleright}z)
 &=(z^{-1}y^2z)z^{-2}x
 =z^{-1}y^2z^{-1}x.
\end{align*}

\noindent(Alexander biquandles)
Let $X$ be a module over a commutative ring $R$ and let $s,t\in R^\times$. Define \({x}~\underline{\triangleright}~{y}=tx+(s-t)y\) and 
\({x}~\overline{\triangleright}~{y}=sx.\) Then $(X,\underline{\triangleright},\overline{\triangleright})$ is a biquandle. In particular, any $R=\mathbb Z[s^{\pm 1}, t^{\pm 1}]$-module yields a biquandle.

\noindent(Dihedral biquandles)
Let $\mathbb Z_n=\{0,1,\ldots,n-1\}$ and $\mathbb Z_n^*=\{s\in \mathbb Z_n\mid \gcd(s,n)=1\}$. For $s\in\mathbb Z_n^*$ define, for $x,y\in\mathbb Z_n$,
\(
x~\underline{\triangleright}_s~y=(s+1)y-x\) and \(x~\overline{\triangleright}_s~y=sx.
\)
Then $(\mathbb Z_n,\underline{\triangleright}_s,\overline{\triangleright}_s)$ is a biquandle with dual operations $x~\underline{\triangleright}_s^{-1}~y=(s+1)y-x$ and $x~\overline{\triangleright}_s^{-1}~y=s^{-1}x.$
\end{example}

\begin{definition}\label{def-n-dibiqdle}
Let $X$ and $\Omega$ be nonempty sets. An {\it $\Omega$-multibiquandle} (resp. {\it $\Omega$-multibirack}) is a quadruple $(X,\Omega,\underline{\mathcal B}(\Omega),\overline{\mathcal B}(\Omega))$ consisting of $X, \Omega$ and two collections $\underline{\mathcal B}(\Omega)=\{\underline{\triangleright}_i\mid i \in \Omega\}$ and $\overline{\mathcal B}(\Omega)=\{\overline{\triangleright}_i\mid i \in \Omega\}$ of binary operations on $X$ satisfying:
\begin{description}  
\item[MBQ1.] For each $i\in \Omega$, $(X,\underline{\triangleright}_i,\overline{\triangleright}_i)$ is a biquandle (resp. birack).
\item[MBQ2.] For each $(i,j) \in \Omega\times \Omega$, the map $H_{ij}:X\times X\rightarrow X\times X$ defined by $H_{ij}(x,y)=({y}~\overline{\triangleright}_i~{x},{x}~\underline{\triangleright}_j~{y})$ is bijective.
\item[MBQ3.] For any $x,y,z\in X$ and $i,j\in \Omega$,
\begin{align*}
&{({x}~\underline{\triangleright}_i~{y})}~\underline{\triangleright}_j~{({z}~\underline{\triangleright}_i~{y})}={({x}~\underline{\triangleright}_j~{z})}~\underline{\triangleright}_i~{({y}~\overline{\triangleright}_j~{z})},\\
&{({x}~\underline{\triangleright}_i~{y})}~\overline{\triangleright}_j~{({z}~\underline{\triangleright}_i~{y})}={({x}~\overline{\triangleright}_j~{z})}~\underline{\triangleright}_i~{({y}~\overline{\triangleright}_j~{z})},\\
&{({x}~\overline{\triangleright}_i~{y})}~\overline{\triangleright}_j~{({z}~\overline{\triangleright}_i~{y})}={({x}~\overline{\triangleright}_j~{z})}~\overline{\triangleright}_i~{({y}~\underline{\triangleright}_j~{z})}.
\end{align*}
\end{description}
When $\Omega=\{1,2,\ldots,n\}$, we refer to $(X,\Omega,\underline{\mathcal B}(\Omega),\overline{\mathcal B}(\Omega))$ as an {\it $n$-multibiquandle} (resp. {\it $n$-multibirack}). For convenience, we write an $\Omega$-multibiquandle (resp. $\Omega$-multibirack) as $(X,\{\underline{\triangleright}_i,\overline{\triangleright}_i\}_{i\in\Omega})$, and an $n$-multibiquandle (resp. $n$-multibirack) as $(X,\{\underline{\triangleright}_i,\overline{\triangleright}_i\}_{i=1}^n)$.
\end{definition}

\begin{lemma}[Equivalent indexing of the mixed exchange laws]
\label{lem:mbq3-reindexing}
For an indexed family of pairs of binary operations, \textbf{MBQ3} is
equivalent to the following three identities, required for every ordered
pair $i,j\in\Omega$ and all $x,y,z\in X$:
\begin{align*}
&(x\mathbin{\underline{\triangleright}_i}y)
 \mathbin{\underline{\triangleright}_j}
 (z\mathbin{\overline{\triangleright}_i}y)
 =(x\mathbin{\underline{\triangleright}_j}z)
 \mathbin{\underline{\triangleright}_i}
 (y\mathbin{\underline{\triangleright}_j}z),\\
&(x\mathbin{\overline{\triangleright}_i}y)
 \mathbin{\underline{\triangleright}_j}
 (z\mathbin{\overline{\triangleright}_i}y)
 =(x\mathbin{\underline{\triangleright}_j}z)
 \mathbin{\overline{\triangleright}_i}
 (y\mathbin{\underline{\triangleright}_j}z),\\
&(x\mathbin{\overline{\triangleright}_i}y)
 \mathbin{\overline{\triangleright}_j}
 (z\mathbin{\overline{\triangleright}_i}y)
 =(x\mathbin{\overline{\triangleright}_j}z)
 \mathbin{\overline{\triangleright}_i}
 (y\mathbin{\underline{\triangleright}_j}z).
\end{align*}
\end{lemma}

\begin{proof}
Apply the first displayed identity with the ordered pair $(j,i)$ and the
variables $(x,z,y)$, and then reverse the equality.  The result is the
first identity in \textbf{MBQ3}.  The same reindexing and interchange of
$y,z$ converts the second displayed identity into the second identity in
\textbf{MBQ3}.  The third displayed identity is already the third identity
in \textbf{MBQ3}.  Since $(i,j)$ ranges over all ordered pairs, the
argument is reversible.
\end{proof}

\begin{remark}\label{rmk1}
Let $(X,\{\underline{\triangleright}_i,\overline{\triangleright}_i\}_{i\in\Omega})$ be an $\Omega$-multibiquandle. If $\varnothing\ne\Omega'\subset \Omega$, then $(X,\{\underline{\triangleright}_i,\overline{\triangleright}_i\}_{i\in\Omega'})$ is an $\Omega'$-multibiquandle. A $1$-multibiquandle is a biquandle. A $2$-multibiquandle $(X,\{\underline{\triangleright}_i,\overline{\triangleright}_i\}_{i=1}^2)$ is a dibiquandle $(X,\underline{\triangleright}_1,\underline{\triangleright}_2,\overline{\triangleright}_1,\overline{\triangleright}_2)$ \cite{KLS2}. If $(X,\{\underline{\triangleright}_i,\overline{\triangleright}_i\}_{i\in\Omega})$ has trivial $\overline{\triangleright}_i$ (i.e., $x~\overline{\triangleright}_i~y=x$ for all $x,y\in X$), then $x~\underline{\triangleright}_i~x=x~\overline{\triangleright}_i~x=x$ for all $x\in X$ and $i\in \Omega$, so \textbf{Q1} holds. The first \textbf{MBQ3} identity, with
$(i,j,y,z)$ replaced by $(j,i,z,y)$ and its equality reversed, is
\textbf{MQ2} of Definition~\ref{def-n-mqdle}. Its case $i=j$ gives
\textbf{Q3}, while \textbf{BQ2} supplies \textbf{Q2}; hence we obtain an
$\Omega$-multiquandle. Conversely, adjoin
$x~\overline{\triangleright}_i~y=x$ to an $\Omega$-multiquandle.
Then $H_{ij}(x,y)=(y,x\underline{\triangleright}_j y)$ has inverse
$H_{ij}^{-1}(r,s)=(s\underline{\triangleright}_j^{-1}r,r)$, proving
\textbf{MBQ2}. The first \textbf{MBQ3} identity is the same reindexing
of \textbf{MQ2}; its other two identities reduce to identical expressions.
Taking $i=j$ gives \textbf{BQ3}--\textbf{BQ4}, and
\textbf{Q1}--\textbf{Q2} give \textbf{BQ1}--\textbf{BQ2}. Thus
\textbf{MBQ1} also holds.
\end{remark}

A {\it homomorphism} of $\Omega$-multibiquandles (resp. $\Omega$-multibiracks) $f:(X,\{\underline{\triangleright}_i,\overline{\triangleright}_i\}_{i\in\Omega})\to (Y,\{\underline{\ast}_i,\overline{\ast}_i\}_{i\in\Omega})$ is a map such that for all $x,y\in X$ and $i\in\Omega$,
\(f(x~\underline{\triangleright}_i~y)=f(x)~\underline{\ast}_{i}~f(y)\) and \(f(x~\overline{\triangleright}_i~y)=f(x)~\overline{\ast}_{i}~f(y)\).
A bijective homomorphism is an {\it isomorphism}. These objects form a category under composition.

\begin{theorem}\label{exp-2}
For each $s \in \mathbb Z^*_n,$ define operations on $\mathbb Z_n$ by
\begin{equation}\label{eq-dhd-bq}
x~\underline{\triangleright}_s~y=(s+1)y-x,\qquad x~\overline{\triangleright}_s~y=sx\quad (x,y\in \mathbb Z_n).
\end{equation} 
Then $(\mathbb Z_n,\{\underline{\triangleright}_s,\overline{\triangleright}_s\}_{s\in \mathbb Z^*_n})$ is a $\mathbb Z^*_n$-multibiquandle.
\end{theorem}

\begin{proof}
{\bf MBQ1} follows from Example \ref{exp-3}. For {\bf MBQ2}, for $i,j \in \mathbb Z_n^*$ define $G_{ij}:\mathbb Z_n^2 \to \mathbb Z_n^2$ by
\(
G_{ij}(x,y)=\big((j+1)i^{-1}x-y,\ i^{-1}x\big).
\)
Then
\begin{align*}
(H_{ij}\circ G_{ij})(x,y)&=H_{ij}((j+1)i^{-1}x-y,i^{-1}x)\\
&=(i(i^{-1}x), (j+1)i^{-1}x-((j+1)i^{-1}x-y))=(x,y),\\
(G_{ij}\circ H_{ij})(x,y)&=G_{ij}(iy,(j+1)y-x)\\
&=((j+1)i^{-1}(iy)-((j+1)y-x),i^{-1}(iy))=(x,y).
\end{align*}
Hence $G_{ij}=H_{ij}^{-1}$, so $H_{ij}$ is bijective. 
For {\bf MBQ3}, direct computations give (in $\mathbb Z_n$):
\begin{align*}
{({x}~\underline{\triangleright}_i~{y})}~\underline{\triangleright}_j~{({z}~\underline{\triangleright}_i~{y})}
&=(j+1)((i+1)y-z)-((i+1)y-x)=ijy+jy-jz-z+x,\\
{({x}~\underline{\triangleright}_j~{z})}~\underline{\triangleright}_i~{({y}~\overline{\triangleright}_j~{z})}
&=(i+1)(jy)-((j+1)z-x)=ijy+jy-jz-z+x,\\
{({x}~\underline{\triangleright}_i~{y})}~\overline{\triangleright}_j~{({z}~\underline{\triangleright}_i~{y})}
&=j((i+1)y-x)=ijy+jy-jx,\\
{({x}~\overline{\triangleright}_j~{z})}~\underline{\triangleright}_i~{({y}~\overline{\triangleright}_j~{z})}
&=jx~\underline{\triangleright}_i~jy=(i+1)jy-jx=ijy+jy-jx,\\
{({x}~\overline{\triangleright}_i~{y})}~\overline{\triangleright}_j~{({z}~\overline{\triangleright}_i~{y})}
&=j(ix)=jix,\\
{({x}~\overline{\triangleright}_j~{z})}~\overline{\triangleright}_i~{({y}~\underline{\triangleright}_j~{z})}
&=i(jx)=ijx.
\end{align*}
Thus {\bf MBQ3} holds.
\end{proof}

We call $(\mathbb Z_n,\{\underline{\triangleright}_s,\overline{\triangleright}_s\}_{s\in \mathbb Z^*_n})$ the {\it dihedral $\mathbb Z^*_n$-multibiquandle of order $n$}.

\begin{definition}\label{def:finite-multibiquandle}
Let $(X,\{\underline{\triangleright}_i,\overline{\triangleright}_i\}_{i\in\Omega})$ be an $\Omega$-multibiquandle. If $X$ and $\Omega$ are both finite, then $(X,\{\underline{\triangleright}_i,\overline{\triangleright}_i\}_{i\in\Omega})$ is called a {\it finite $\Omega$-multibiquandle of order $|X|$}.
\end{definition}

The dihedral $\mathbb Z^*_n$-multibiquandle in Theorem \ref{exp-2} is finite of order $n$.

\bigskip

Let $(X,\{\underline{\triangleright}_k,\overline{\triangleright}_k\}_{k=1}^n)$ be a finite $n$-multibiquandle of order $m$, where $X=\{x_1,\ldots,x_m\}$. The {\it operation table} of $(X,\{\underline{\triangleright}_k,\overline{\triangleright}_k\}_{k=1}^n)$ is the block matrix
\[
\left[\begin{array}{c|c|c|c|c|c}
\underline{M}^1&\cdots&\underline{M}^n&\overline{M}^1&\cdots&\overline{M}^n
\end{array} \right],
\]
where $\underline{M}^k=(\underline{a}^k_{ij})$ and $\overline{M}^k=(\overline{a}^k_{ij})$ are $m\times m$ matrices with entries
\(
\underline{a}^k_{ij}=x_i\underline{\triangleright}_k x_j\) and
\(\overline{a}^k_{ij}=x_i\overline{\triangleright}_k x_j.
\)

\begin{example}\label{example2.6}
Let $\mathbb Z_3=\{0,1,2\}$. Then $\mathbb Z^*_3=\{1,2\}$ and from \eqref{eq-dhd-bq},
\[
\begin{aligned}
&s=1:\quad x~\underline{\triangleright}_1~y=2y-x,\quad x~\overline{\triangleright}_1~y=x\quad (x,y\in \mathbb Z_3),\\
&s=2:\quad x~\underline{\triangleright}_2~y=-x,\qquad\ \ x~\overline{\triangleright}_2~y=2x\quad (x,y\in \mathbb Z_3).
\end{aligned}
\]
Hence the dihedral $\mathbb Z^*_3$-multibiquandle $(\mathbb Z_3,\{\underline{\triangleright}_s,\overline{\triangleright}_s\}_{s\in \mathbb Z^*_3})$ is a $2$-multibiquandle of order $3$ with operation table:
\[
\left[\begin{array}{ccc|ccc|ccc|ccc} 
0 & 2 & 1 & 0 & 0 & 0 & 0 & 0& 0 & 0 & 0 & 0 \\ 
2 & 1 & 0 & 2 & 2 & 2 & 1 & 1 & 1 & 2 & 2 & 2 \\ 
1 & 0 & 2 & 1 & 1 & 1 & 2 & 2& 2 & 1 & 1 & 1
\end{array} \right].
\]
\end{example}

\begin{definition}\cite{IN}\label{defn-n-pall-bql}
Let $(X,\underline{\triangleright},\overline{\triangleright})$ be a biquandle. For $n \geq 0$, the {\it $n$-parallel biquandle} of $X$ is $(X,\underline{\triangleright}^{[n]},\overline{\triangleright}^{[n]})$, where $\underline{\triangleright}^{[n]}$ and $\overline{\triangleright}^{[n]}$ are recursively defined by
\[
\begin{aligned}
&a~\underline{\triangleright}^{[0]}~b=a,\qquad a~\underline{\triangleright}^{[n]}~b=(a~\underline{\triangleright}^{[n-1]}~b)~\underline{\triangleright}~(b~\underline{\triangleright}^{[n-1]}~b),\\ 
&a~\overline{\triangleright}^{[0]}~b=a,\qquad a~\overline{\triangleright}^{[n]}~b=(a~\overline{\triangleright}^{[n-1]}~b)~\overline{\triangleright}~(b~\overline{\triangleright}^{[n-1]}~b).
\end{aligned}
\]
Here the recursive formulas apply for $n\ge1$.
\end{definition}

\begin{theorem}\label{thm-m-bqdl}
Let $(X,\underline{\triangleright},\overline{\triangleright})$ be a biquandle and let $(X,\underline{\triangleright}^{[n]},\overline{\triangleright}^{[n]})$ be its $n$-parallel biquandle for each $n\geq 0$. Then $(X,\{\underline{\triangleright}_{n},\overline{\triangleright}_{n}\}_{n\in \Omega})$ is an $\Omega$-multibiquandle, where $\Omega=\{n\in \mathbb Z\mid n\geq 0\}$, $\underline{\triangleright}_{n}=\underline{\triangleright}^{[n]}$, and $\overline{\triangleright}_{n}=\overline{\triangleright}^{[n]}$.
\end{theorem}

\begin{proof}
Since $(X,\underline{\triangleright}^{[n]},\overline{\triangleright}^{[n]})$ is a biquandle (see \cite[Lemmas~3.4–3.6]{IN}), {\bf MBQ1} holds. {\bf MBQ2} follows from \cite[Lemma~3.6(ii)]{IN} with $H_{mn}=S_{m,n}$.  The three identities in \cite[Lemma~3.5]{IN} are the reindexed form displayed in Lemma~\ref{lem:mbq3-reindexing}; hence they imply {\bf MBQ3}.
\end{proof}

Let $n\ge 1$ and put $\Omega_n=\{0,1,\ldots,n\}\subset\Omega$.
Theorem~\ref{thm-m-bqdl} yields an $\Omega_n$-multibiquandle; its
restriction to $\{1,\ldots,n\}$ is called the {\it $n$-parallel
multibiquandle of $(X,\underline{\triangleright},\overline{\triangleright})$}.

\begin{example}\label{exa-alx-bqdl}
Let $A$ be an abelian group and let $s,t:A\to A$ be commuting automorphisms. Define \(x~\underline{\triangleright}~y=t(x-y)+s(y)\) and 
\(x~\overline{\triangleright}~y=s(x).\)
Then $(A,\underline{\triangleright}, \overline{\triangleright})$ is a biquandle: the commuting automorphisms make $A$ a module over
$\mathbb Z[s^{\pm1},t^{\pm1}]$, so Example~\ref{exp-3} applies.
For $n\ge 0$, define
\(
x~\underline{\triangleright}^{[n]}~y=t^n(x-y)+s^n(y)\) and \(x~\overline{\triangleright}^{[n]}~y=s^n(x).
\)
Then $(A,\underline{\triangleright}^{[n]},\overline{\triangleright}^{[n]})$ is the $n$-parallel biquandle of $(A,\underline{\triangleright},\overline{\triangleright})$ \cite[Prop.~3.8]{IN}.
To verify the formulas directly, the case $n=0$ is the identity
operation. If the formulas hold at $n$, their diagonal values are
$s^n(y)$, and Definition~\ref{defn-n-pall-bql} gives
\[
\begin{aligned}
x\underline{\triangleright}^{[n+1]}y
 &=\bigl(t^n(x-y)+s^n(y)\bigr)
   \underline{\triangleright}s^n(y)
 =t\bigl(t^n(x-y)+s^n(y)-s^n(y)\bigr)+s^{n+1}(y)\\
 &=t^{n+1}(x-y)+s^{n+1}(y),\\
x\overline{\triangleright}^{[n+1]}y
 &=s(s^n(x))=s^{n+1}(x).
\end{aligned}
\]
Induction proves the formulas for every $n\ge0$.
With $\Omega=\{n\in\mathbb Z\mid n\ge 0\}$, setting $\underline{\triangleright}_{n}=\underline{\triangleright}^{[n]}$ and $\overline{\triangleright}_{n}=\overline{\triangleright}^{[n]}$ gives an $\Omega$-multibiquandle; restricting to $\{1,\ldots,n\}$ gives the $n$-parallel multibiquandle.
\end{example}

\begin{example}\label{exa-3-pall}
Let $X=\mathbb Z_3$ and define
\(
 x\mathbin{\underline{\triangleright}}y=x\) and
 \(x\mathbin{\overline{\triangleright}}y=2x-y
 \quad(x,y\in\mathbb Z_3).
\)
We check \textbf{BQ1}--\textbf{BQ4} of Definition~\ref{def-biqdle}.
For \textbf{BQ1}, both diagonal values are $x$.
For \textbf{BQ2}, the under-translation is the identity and the
over-translation is its own inverse, since $2(2x-y)-y=x$ in $\mathbb Z_3$.
For \textbf{BQ3}, the switch $H(x,y)=(2y-x,x)$ has inverse
$H^{-1}(r,s)=(s,2r+2s)$.
For \textbf{BQ4}, the two sides of its three identities simplify,
respectively, to $x$, $2x-z$, and $x+2y+z$ in $\mathbb Z_3$.
The diagonal value of the over-operation is $y$, and
$(2x-y)\mathbin{\overline{\triangleright}}y=x$.  Hence its parallel
operations are
\[
 x\mathbin{\underline{\triangleright}^{[n]}}y=x
 \quad(n\ge0),
 \qquad
 x\mathbin{\overline{\triangleright}^{[n]}}y=
 \begin{cases}
 x,&n\text{ even},\\
 2x-y,&n\text{ odd}.
 \end{cases}
\]
In particular, with
$\underline{\triangleright}_n=\underline{\triangleright}^{[n]}$ and
$\overline{\triangleright}_n=\overline{\triangleright}^{[n]}$ for
$n=1,2,3$, we obtain a $3$-parallel multibiquandle in which
\(
 x\mathbin{\underline{\triangleright}_n}y=x\,(n=1,2,3),
 \,
 x\mathbin{\overline{\triangleright}_2}y=x\) and
 \(
 x\mathbin{\overline{\triangleright}_1}y
 =x\mathbin{\overline{\triangleright}_3}y=2x-y.
\)
\end{example}

\begin{definition}\label{def-G-mbqdle}
Let $X$ and $G$ be nonempty sets. A $G$-multibiquandle (resp. $G$-multibirack) $(X,\{\underline{\triangleright}_g,\overline{\triangleright}_g\}_{g\in G})$ is called a {\it group $G$-multibiquandle} (resp. {\it group $G$-multibirack}) if $G$ is a group.
\end{definition}

The class of group $G$-multibiquandles (resp. group $G$-multibiracks) forms a category under composition of homomorphisms. If $H\le G$ and $(X,\{\underline{\triangleright}_g,\overline{\triangleright}_g\}_{g\in G})$ is a group $G$-multibiquandle, then $(X,\{\underline{\triangleright}_h,\overline{\triangleright}_h\}_{h\in H})$ is a group $H$-multibiquandle. If all $\overline{\triangleright}_g$ are trivial, then $(X,\{\underline{\triangleright}_g\}_{g\in G})$ is a group $G$-multiquandle; conversely, adjoining trivial $\overline{\triangleright}_g$ yields a group $G$-multibiquandle.

\begin{remark}\label{rmk-g-biqdle}
In Definition~\ref{def-G-mbqdle}, as in
Remark~\ref{rem:indexed-not-group-law}, the families
$\{\underline{\triangleright}_g\}_{g\in G}$ and
$\{\overline{\triangleright}_g\}_{g\in G}$ are indexed by $G$, but do not
automatically inherit a well-defined group law from $G$ when indexed
operations repeat.
\end{remark}

\begin{example}\label{exp-4}
Let $G$ be a group. For $n\in\mathbb Z$, define $\underline{\triangleright}_n$ on $G$ by $x~\underline{\triangleright}_n~y= y^{-n}xy^n$. Then $(G,\{\underline{\triangleright}_n\}_{n\in\mathbb Z})$ is a group $\mathbb Z$-multiquandle (cf. Example \ref{exa-conj-mq}), and $(G,\{\underline{\triangleright}_n,\overline{\triangleright}_n\}_{n\in\mathbb Z})$ is a group $\mathbb Z$-multibiquandle if we set $x~\overline{\triangleright}_n~y=x$ for all $x,y\in G$.
\end{example}

\begin{definition}\cite{IN}\label{defn-gfam-biq}
Let $G$ be a group with identity $e$. A {\it $G$-family of biquandles} is a triple $(X,\{\underline{\triangleright}^g\}_{g\in G}, \{\overline{\triangleright}^g\}_{g\in G})$ consisting of a nonempty set $X$ and two families of operations on $X$ satisfying:
\begin{itemize}
\item[(i)] $x~\underline{\triangleright}^g~x=x~\overline{\triangleright}^g~x$ for any $g\in G$ and $x\in X$.
\item[(ii)] For each $g\in G$ and $y\in X$, the maps $\underline{\triangleright}^g y, \overline{\triangleright}^g y:X\to X$, $\underline{\triangleright}^g y(x)=x~\underline{\triangleright}^g~y$, $\overline{\triangleright}^g y(x)=x~\overline{\triangleright}^g~y$, are bijections.
\item[(iii)] For $g,h\in G$, the map $S_{g,h}:X^2\to X^2$, $S_{g,h}(x,y)=(y~\overline{\triangleright}^g~x, x~\underline{\triangleright}^h~y)$, is a bijection.
\item[(iv)] For any $g,h\in G$ and $x,y,z \in X$,
\begin{align*}
&(x~\underline{\triangleright}^g~y)\ \underline{\triangleright}^h\ (z~\overline{\triangleright}^g~y)=(x~\underline{\triangleright}^h~z)\ \underline{\triangleright}^{h^{-1}gh}\ (y~\underline{\triangleright}^h~z),\\
&(x~\overline{\triangleright}^g~y)\ \underline{\triangleright}^h\ (z~\overline{\triangleright}^g~y)=(x~\underline{\triangleright}^h~z)\ \overline{\triangleright}^{h^{-1}gh}\ (y~\underline{\triangleright}^h~z),\\
&(x~\overline{\triangleright}^g~y)\ \overline{\triangleright}^h\ (z~\overline{\triangleright}^g~y)=(x~\overline{\triangleright}^h~z)\ \overline{\triangleright}^{h^{-1}gh}\ (y~\underline{\triangleright}^h~z).
\end{align*}
\item[(v)] For any $g,h\in G$ and $x,y \in X$,
\(
x~\underline{\triangleright}^{gh}~y=(x~\underline{\triangleright}^g~y)\ \underline{\triangleright}^{h}\ (y~\underline{\triangleright}^g~y)\) and \(
x~\overline{\triangleright}^{gh}~y=(x~\overline{\triangleright}^g~y)\ \overline{\triangleright}^{h}\ (y~\overline{\triangleright}^g~y).
\)
\end{itemize}
\end{definition}

\begin{theorem}\label{thm-gfam-m-bqdle}
Let $G$ be a commutative group and let $(X,\{\underline{\triangleright}^g\}_{g\in G}, \{\overline{\triangleright}^g\}_{g\in G})$ be a $G$-family of biquandles. Then $(X,\{\underline{\triangleright}_{g},\overline{\triangleright}_{g}\}_{g\in G})$ is a group $G$-multibiquandle, where $\underline{\triangleright}_{g}=\underline{\triangleright}^{g}$ and $\overline{\triangleright}_{g}=\overline{\triangleright}^{g}$.
\end{theorem}

\begin{proof}
We check {\bf MBQ1}–{\bf MBQ3}. Definition
\ref{defn-gfam-biq}(iii) says that $S_{g,h}$ is bijective for every
$g,h$. Since $H_{gh}=S_{g,h}$, this proves {\bf MBQ2}. If $G$ is commutative, then $h^{-1}gh=g$, and Definition \ref{defn-gfam-biq}(iv) yields, for $x,y,z\in X$ and $g,h\in G$,
\begin{equation}\label{eqn-gfam-biq}
\begin{split}
&(x~\underline{\triangleright}^g~y) \underline{\triangleright}^h (z~\overline{\triangleright}^g~y)=(x~\underline{\triangleright}^h~z) \underline{\triangleright}^{g} (y~\underline{\triangleright}^h~z),\\
&(x~\overline{\triangleright}^g~y) \underline{\triangleright}^h (z~\overline{\triangleright}^g~y)=(x~\underline{\triangleright}^h~z) \overline{\triangleright}^{g} (y~\underline{\triangleright}^h~z),\\
&(x~\overline{\triangleright}^g~y) \overline{\triangleright}^h (z~\overline{\triangleright}^g~y)=(x~\overline{\triangleright}^h~z) \overline{\triangleright}^{g} (y~\underline{\triangleright}^h~z).
\end{split}
\end{equation}
These are the reindexed mixed exchange laws in
Lemma~\ref{lem:mbq3-reindexing}; hence they are equivalent to {\bf MBQ3}
with $\underline{\triangleright}_{g}=\underline{\triangleright}^{g}$ and
$\overline{\triangleright}_{g}=\overline{\triangleright}^{g}$.

Next, for fixed $g$, {\bf BQ1} is Definition \ref{defn-gfam-biq}(i). For {\bf BQ2}, define $\underline{\triangleright}_{g}^{-1},\overline{\triangleright}_{g}^{-1}$ by $x\underline{\triangleright}_{g}^{-1}y=(\underline{\triangleright}^gy)^{-1}(x)$ and $x\overline{\triangleright}_{g}^{-1}y=(\overline{\triangleright}^gy)^{-1}(x)$; Definition \ref{defn-gfam-biq}(ii) implies the required identities. {\bf BQ3} follows from (iii) with $H=S_{g,g}$.
For {\bf BQ4}, set both indices equal to $g$ in
\eqref{eqn-gfam-biq} and apply
Lemma~\ref{lem:mbq3-reindexing}; this gives the three ordinary biquandle
exchange laws.  Hence {\bf MBQ1} holds and we have a group
$G$-multibiquandle.
\end{proof}

\begin{definition}[Type of a biquandle]\label{def:biquandle-type}
Let $(X,\underline{\triangleright},\overline{\triangleright})$ be a biquandle. The \emph{type of $X$}, denoted ${\rm type}\,X$, is
\[
{\rm type}\,X=\min\{n>0\mid a~\underline{\triangleright}^{[n]}~b=a=a~\overline{\triangleright}^{[n]}~b\ \text{for all }a,b\in X\}.
\]
If the displayed set is empty, put ${\rm type}\,X=\infty$.
If $X$ is finite then ${\rm type}\,X<\infty$ \cite[Prop.~4.6]{IN}.
\end{definition}

\begin{example}
Let $(X,\underline{\triangleright},\overline{\triangleright})$ be a biquandle with ${\rm type}\,X<\infty$ and let $\mathbb Z_{n}$ be cyclic of order $n={\rm type}\,X$. Then $(X,\{\underline{\triangleright}^{[k]}\}_{k\in \mathbb Z_{n}},\{\overline{\triangleright}^{[k]}\}_{k\in \mathbb Z_{n}})$ is a $\mathbb Z_{n}$-family of biquandles \cite[Thm.~4.5]{IN}. By Theorem \ref{thm-gfam-m-bqdle}, $(X,\{\underline{\triangleright}_{k},\overline{\triangleright}_{k}\}_{k\in \mathbb Z_{n}})$ is a group $\mathbb Z_{n}$-multibiquandle with $\underline{\triangleright}_{k}=\underline{\triangleright}^{[k]}$ and $\overline{\triangleright}_{k}=\overline{\triangleright}^{[k]}$.

In particular, let $X=\mathbb Z_5$ with $x\underline{\triangleright}y=2x+y$ and $x\overline{\triangleright}y=3x$.
Example~\ref{exa-alx-bqdl} gives
$x\underline{\triangleright}^{[n]}y=2^nx+(3^n-2^n)y$ and
$x\overline{\triangleright}^{[n]}y=3^nx$.
For $n=0,1,2,3,4$, the powers in $\mathbb Z_5$ are
$(2^n)=(1,2,4,3,1)$ and $(3^n)=(1,3,4,2,1)$; hence
$(3^n-2^n)=(0,1,0,4,0)$. Then, for $x,y\in X$,
\[
\begin{aligned}
&x\underline{\triangleright}^{[0]}y=x,\quad
x\underline{\triangleright}^{[1]}y=2x+y,\quad
x\underline{\triangleright}^{[2]}y=4x,\quad
x\underline{\triangleright}^{[3]}y=3x+4y,\quad
x\underline{\triangleright}^{[4]}y=x,\\
&x\overline{\triangleright}^{[0]}y=x,\quad
x\overline{\triangleright}^{[1]}y=3x,\quad
x\overline{\triangleright}^{[2]}y=4x,\quad
x\overline{\triangleright}^{[3]}y=2x,\quad
x\overline{\triangleright}^{[4]}y=x.
\end{aligned}
\]
Both operations are the identity at $n=4$. For $n=1,2,3$, the
over-operation sends $1$ to $3,4,2$, respectively, so it is not the
identity. Hence ${\rm type}\,X=4$ and the resulting
$(\mathbb Z_5,\{\underline{\triangleright}_{n},
\overline{\triangleright}_{n}\}_{n\in \mathbb Z_{4}})$ is a group
$\mathbb Z_{4}$-multibiquandle.
\end{example}

\begin{theorem}\label{thm-gmbql-gfbqdl}
Let $G$ be a commutative group with identity $e$ and let $(X,\{\underline{\triangleright}_g, \overline{\triangleright}_g\}_{g\in G})$ be a group $G$-multibiquandle satisfying, for all $x,y \in X$ and $g,h \in G$, 
\[
x~\underline{\triangleright}_{gh}~y=(x~\underline{\triangleright}_g~y) \underline{\triangleright}_{h} (y~\underline{\triangleright}_g~y),\qquad
x~\overline{\triangleright}_{gh}~y=(x~\overline{\triangleright}_g~y) \overline{\triangleright}_{h} (y~\overline{\triangleright}_g~y).
\]
Then $(X,\{\underline{\triangleright}^g\}_{g\in G}, \{\overline{\triangleright}^g\}_{g\in G})$ is a $G$-family of biquandles, where $\underline{\triangleright}^g=\underline{\triangleright}_g$ and $\overline{\triangleright}^g=\overline{\triangleright}_g$.
\end{theorem}

\begin{proof}
{\bf (i)} follows from {\bf MBQ1}. For {\bf (ii)}, for each $g\in G$
and $y\in X$, the maps
$x\mapsto x~\underline{\triangleright}^g~y$ and
$x\mapsto x~\overline{\triangleright}^g~y$ are bijections with inverses
given by the dual operations in {\bf BQ2}. Condition {\bf (iii)} is
{\bf MBQ2} with $S_{g,h}=H_{gh}$. For {\bf (iv)}, abelianness gives
$h^{-1}gh=g$, and Lemma~\ref{lem:mbq3-reindexing} converts {\bf MBQ3}
to the three identities in Definition~\ref{defn-gfam-biq}(iv). Finally,
{\bf (v)} is the hypothesis. Hence we obtain a $G$-family of
biquandles.
\end{proof}

From now on we introduce the notion of multibiquandle modules over a commutative ring.

\begin{definition}\label{def:multibiquandle-module}
Let $G$ be a set and $R$ a ring.
A {\it $G$-multibiquandle $R$-module} is an $R$-module $M$ equipped with collections $\{\underline{\triangleright}_g\mid g\in G\}$ and $\{\overline{\triangleright}_g\mid g\in G\}$ of binary operations $M\times M\to M$ such that $(M,\{\underline{\triangleright}_g, \overline{\triangleright}_g\}_{g\in G})$ is a $G$-multibiquandle. If $G$ is a group, then $M$ is called a {\it group $G$-multibiquandle $R$-module}.
A \emph{$G$-multiquandle $R$-module} is defined similarly using a $G$-multiquandle.
\end{definition}

Let $G$ be a group. The {\it center} of $G$ is $Z(G)=\{c\in G\mid cg=gc\ \text{for all }g\in G\}$. For a ring $R$ with identity $1_{\rm R}$, the {\it center} of $R$ is $Z(R)=\{c\in R\mid cr=rc\ \text{for all }r\in R\}$. Elements of $Z(G)$ and $Z(R)$ are referred to as {\it central elements} of $G$ and $R$, respectively. From this point onward, $R$ is a commutative ring with identity, $G$ is an abelian multiplicative group with identity $e$, and
\(A:=R[G]\)
is the resulting commutative group ring.  We use the standard notation
$A^\times$ for its group of units.  Every $g\in G$ is identified with the
unit $1_Rg\in A^\times$.

Let $M_{R[G]}$ be a (left) $R[G]$-module and let $u\in A^\times$. For each $g\in G$, define affine operations
\begin{equation}\label{eq1-mdl-gpring}
x~\underline{\triangleright}_g~y = gx + (u-g)y,\qquad 
x~\overline{\triangleright}_g~y = ux
\end{equation}
for any $x, y \in M_{R[G]}$.

\begin{lemma}\label{lem-multibqdl-mdl}
For each $g\in G$,
$(M_{R[G]},\underline{\triangleright}_g,\overline{\triangleright}_g)$ is a biquandle with dual operations
\begin{equation}\label{eq2-mdl-gpring}
x~\underline{\triangleright}_g^{-1}~y = g^{-1}x + (1_A-g^{-1}u)y,\qquad 
x~\overline{\triangleright}_g^{-1}~y = u^{-1}x.
\end{equation}
\end{lemma}

\begin{proof}
We verify \textbf{BQ1}--\textbf{BQ4} of Definition~\ref{def-biqdle}.
For \textbf{BQ1},
$x~\underline{\triangleright}_g~x=gx+(u-g)x=ux
=x~\overline{\triangleright}_g~x$.
For \textbf{BQ2}, solving $w=gx+(u-g)y$ for $x$ gives
$x=g^{-1}w+(1-g^{-1}u)y$, and solving $w=ux$ gives $x=u^{-1}w$.
These unique solutions are the inverse translations in
\eqref{eq2-mdl-gpring}, proving both composition identities for each operation.
For \textbf{BQ3}, the equation
$H(x,y)=(uy,gx+(u-g)y)=(r,s)$ has the unique solution
$y=u^{-1}r$ and $x=g^{-1}s+(u^{-1}-g^{-1})r$.
Thus $H$ is bijective with the stated inverse. 
For {\bf BQ4}, using $ug=gu$ gives
\begin{align*}
(x\underline{\triangleright}_g y)
 \underline{\triangleright}_g
 (z\underline{\triangleright}_g y)
 &=g^2x+u(u-g)y+g(u-g)z
 =(x\underline{\triangleright}_g z)
 \underline{\triangleright}_g
 (y\overline{\triangleright}_g z),\\
(x\underline{\triangleright}_g y)
 \overline{\triangleright}_g
 (z\underline{\triangleright}_g y)
 &=ugx+u(u-g)y
 =(x\overline{\triangleright}_g z)
 \underline{\triangleright}_g
 (y\overline{\triangleright}_g z),\\
(x\overline{\triangleright}_g y)
 \overline{\triangleright}_g
 (z\overline{\triangleright}_g y)
 &=u^2x
 =(x\overline{\triangleright}_g z)
 \overline{\triangleright}_g
 (y\underline{\triangleright}_g z).
\end{align*}
\end{proof}

\begin{theorem}\label{thm-multibqdl-mdl}
Let $R$ be a commutative ring, let $G$ be an abelian group, put
$A=R[G]$, and let $M$ be an $A$-module.  For $u\in A^\times$, the
operations in \eqref{eq1-mdl-gpring} make
$(M,\{\underline{\triangleright}_g,
\overline{\triangleright}_g\}_{g\in G})$ a $G$-multibiquandle $A$-module
in the sense of Definition~\ref{def:multibiquandle-module}.
In particular, if $u=1_A$, then
$(M,\{\underline{\triangleright}_g\}_{g\in G})$ is a
$G$-multiquandle module.
\end{theorem}

\begin{proof}
By Lemma \ref{lem-multibqdl-mdl}, {\bf MBQ1} holds. For $(g,h)\in G^2$, define
\(
H_{gh}(x,y)=(uy,\; hx+(u-h)y)\) and \(
G_{gh}(x,y)=((u^{-1}-h^{-1})x+h^{-1}y,\;u^{-1}x).
\)
Then $H_{gh}\circ G_{gh}=G_{gh}\circ H_{gh}=\mathrm{id}$, so {\bf MBQ2} holds. For {\bf MBQ3}, substitute the two affine operations into the six
sides of the three identities. Using commutativity of $G$ and centrality
of $u$ gives
\begin{align*}
({x}~\underline{\triangleright}_g~{y})~\underline{\triangleright}_h~({z}~\underline{\triangleright}_g~{y})
&=hgx+(u^2-ug)y+(ug-hg)z,\\
({x}~\underline{\triangleright}_h~{z})~\underline{\triangleright}_g~({y}~\overline{\triangleright}_h~{z})
&=ghx+(u^2-gu)y+(gu-gh)z,\\
({x}~\underline{\triangleright}_g~{y})~\overline{\triangleright}_h~({z}~\underline{\triangleright}_g~{y})
&=ugx+(u^2-ug)y,\\
({x}~\overline{\triangleright}_h~{z})~\underline{\triangleright}_g~({y}~\overline{\triangleright}_h~{z})
&=gux+(u^2-gu)y,\\
({x}~\overline{\triangleright}_g~{y})~\overline{\triangleright}_h~({z}~\overline{\triangleright}_g~{y})
&=u^2x,\\
({x}~\overline{\triangleright}_h~{z})~\overline{\triangleright}_g~({y}~\underline{\triangleright}_h~{z})
&=u^2x,
\end{align*}
which match \textbf{MBQ3} of Definition~\ref{def-n-dibiqdle}.
If $u=1_A$, then $\overline{\triangleright}_g$ is trivial and
Remark~\ref{rmk1} applies. Explicitly, the mixed distributivity identity is
\(
(x\underline{\triangleright}_g y)\underline{\triangleright}_h z
=hgx+h(1-g)y+(1-h)z
=(x\underline{\triangleright}_h z)
 \underline{\triangleright}_g
 (y\underline{\triangleright}_h z),
\)
using $gh=hg$. Hence the under-operations form a $G$-multiquandle module.
\end{proof}

\begin{remark}
The parameter in Theorem~\ref{thm-multibqdl-mdl} must be a unit.  Hence an
additional assumption that $u$ is idempotent forces $u=1_A$ and produces
only the quandle specialization already stated in the theorem.
\end{remark}

\begin{corollary}\label{thm-A-multibqdl}
Let $G=\langle s,t_1,\dots,t_n\rangle\cong \mathbb Z^{n+1}$ be free abelian, $R=\mathbb Z$ or $\mathbb Z_m$, and $M_{R[G]}$ any $R[G]$-module. For $x,y\in M_{R[G]}$ and $k=1,\dots,n$, define 
\begin{equation}\label{asdq-ops}
x~\underline{\triangleright}_k~y = t_kx + (s-t_k)y,\qquad 
x~\overline{\triangleright}_k~y = sx.
\end{equation}
Then $(M_{R[G]},\{\underline{\triangleright}_k,\overline{\triangleright}_k\}_{k=1}^n)$ is an $n$-multibiquandle $\Lambda_n$-module with dual operations 
\begin{equation}\label{asdq-ops-1}
x~\underline{\triangleright}_k^{-1}~y = t_k^{-1}x + (1-st_k^{-1})y,\qquad 
x~\overline{\triangleright}_k^{-1}~y = s^{-1}x\quad (x,y \in M_{R[G]}),
\end{equation}
where $\Lambda_n=R[s^{\pm 1}, t_1^{\pm 1},\ldots,t_n^{\pm 1}]$ is the Laurent polynomial ring.  Moreover, let
\(q:\Lambda_n\longrightarrow \Lambda_n^{(1)}:=R[t_1^{\pm1},\ldots,t_n^{\pm1}]\)
be the $R$-algebra homomorphism defined by $q(s)=1$ and $q(t_i)=t_i$,
and put
$M^{(1)}=M_{R[G]}\otimes_{\Lambda_n,q}\Lambda_n^{(1)}$.
The induced operation
\(
x\mathbin{\underline{\triangleright}_k}y
=t_kx+(1-t_k)y
\, (k=1,\ldots,n)
\)
makes $M^{(1)}$ an $n$-multiquandle $\Lambda_n^{(1)}$-module.
\end{corollary}

\begin{proof}
Since $R[G]\cong \Lambda_n$, set $u=s$ in \eqref{eq1-mdl-gpring}. Theorem \ref{thm-multibqdl-mdl} yields a group $G$-multibiquandle $\Lambda_n$-module. With $\Omega=\{t_1,\ldots,t_n\}\subset G$, Remark \ref{rmk1} implies $(M_{R[G]},\{\underline{\triangleright}_{t_k},\overline{\triangleright}_{t_k}\}_{k=1}^n)$ is an $n$-multibiquandle; the formulas \eqref{asdq-ops}–\eqref{asdq-ops-1} are instances of \eqref{eq1-mdl-gpring}–\eqref{eq2-mdl-gpring}.  Base change along $q$ sends the over-operation to $x\mathbin{\overline{\triangleright}_k}y=x$ and the under-operation to $t_kx+(1-t_k)y$.  The final assertion now follows from the last clause of Theorem~\ref{thm-multibqdl-mdl}, applied over $\Lambda_n^{(1)}$.
\end{proof}

\begin{remark}[Comparison with $G$-Alexander families]
For a fixed index $h$, the operations above have the usual
Alexander-biquandle form. In the conventional $G$-Alexander
construction \cite{IN,GAB}, a homomorphism
$\varphi:G\to Z(G)$ gives, on a right $R[G]$-module,
\[
(m,g)\mathbin{\underline *}(n,h)
=(mh+n(\varphi(h)-h),h^{-1}gh),\qquad
(m,g)\mathbin{\overline *}(n,h)=(m\varphi(h),g).
\]
For abelian $G$, comparison on the regular module with the same full
indexing requires $\varphi(h)=u$ for every $h$, hence $u=1_A$.
More precisely, on a specified $A$-module $M$, the common-unit family
satisfies the $G$-family axioms if and only if $(u-1_A)M=0$:
Definition~\ref{defn-gfam-biq}(v) gives $ux=u^2x$, and conversely an
identity action of $u$ gives the usual $G$-family of Alexander quandles.
This distinguishes the literal indexed families, not their invariants
under changes of labels or coefficients. The general classical-link
quandle reduction \cite{IT,Murao} is realized by the explicit module
isomorphism in Theorem~\ref{thm:all-color-reduction} below.
\end{remark}

We now turn to multibiquandle colorings of component-colored links.
The common-unit affine family constructed above will later provide the
coefficient multibiquandles from which the corresponding
module-valued invariants are obtained.

\section{Multibiquandle counting invariants for colored links}\label{sec-fundbql-dl}

We fix the conventions for $\Omega$-colored and $G$-colored links and their multibiquandle colorings. Finite-color counting constructions appear in~\cite{KLS2,LeeSolid}; the naturality statement below supplies the module-valued construction in Section~\ref{sect-abq}.

\begin{definition}\label{def-colored-link}
Let $\Omega$ be a nonempty set. 
An {\it $\Omega$-colored link} (or {\it colored link by $\Omega$}) is an oriented link 
$
L = K_1 \cup \cdots \cup K_\mu \subset \mathbb R^3 \ \text{or}\ S^3
$
together with a map $\sigma : \pi_0(L) \to \Omega$. (Here, $\pi_0(L)$ denotes the set of all path-connected components of $L\subset S^3$ and hence $\pi_0(L)=\{K_1, \ldots, K_\mu\}$.)
We denote such a link by $(L,\sigma)$, or simply by $L$ when $\sigma$ is clear from the context. 
Each smoothly embedded circle $K_i$ is called a {\it component} of $L$, and $\sigma(K_i) \in \Omega$ is called its {\it color} (or {\it label}). 
Different components of an $\Omega$-colored link may share the same color. 
When $\Omega = \{1,2,\ldots,n\}$, an $\Omega$-colored link is called an {\it $n$-colored link}; in particular, a $2$-colored link is sometimes referred to as a {\it dichromatic link}.
\end{definition}

\begin{definition}\label{def-colored-equivalence}
Two $\Omega$-colored links $(L,\sigma)$ and $(L',\sigma')$ in $S^3$ are
said to be {\it equivalent} if there is an ambient isotopy
$H:S^3\times[0,1]\to S^3$ with $H_0=\mathrm{id}_{S^3}$ such that
$H_1(L)=L'$, the restriction of $H_1$ preserves the orientation of every
component, and \(\sigma(K)=\sigma'(H_1(K))\,(K\in\pi_0(L)).\)
\end{definition}

An $\Omega$-colored link $(L,\sigma)$ can be represented by a {\it diagram} 
$D = D_1 \cup \cdots \cup D_\mu \subset \mathbb R^2,$
where each $D_i$ corresponds to $K_i$ and is assigned the color $\sigma(K_i)$. 
We denote such a diagram by $(D,\sigma)$, or simply by $D$ when $\sigma$ is clear from the context. 
The following is the colored form of the Reidemeister theorem.

\begin{theorem}\label{thm-equi-nchr}
Let $(D,\sigma)$ and $(D',\sigma')$ be two $\Omega$-colored link diagrams in $\mathbb R^2$ representing $\Omega$-colored links $(L,\sigma)$ and $(L',\sigma')$ in $S^3$, respectively. 
Then $(L,\sigma)$ and $(L',\sigma')$ are equivalent if and only if $(D,\sigma)$ and $(D',\sigma')$ can be transformed into each other by a finite sequence of {\it generalized Reidemeister moves for $\Omega$-colored links} that preserve the color of each component (see Figure~\ref{fig-ormoves}), and ambient isotopies of $\mathbb R^2$.
\end{theorem}

\begin{proof}
Forgetting the labels reduces a colored Reidemeister sequence to an
ordinary oriented Reidemeister sequence, which realizes an ambient
isotopy of the represented link.  Because every local move preserves the
components of the strands meeting its disk, the labels remain constant
on components throughout the sequence.

Conversely, suppose the colored links are ambient isotopic.  Apply the
oriented Reidemeister theorem to the underlying oriented diagrams, after
identifying components by the given isotopy.  Every strand occurring in a
move belongs to the same component before and after that move; assign it
the fixed label of that component.  The resulting sequence is therefore
a sequence of color-preserving generalized Reidemeister moves and planar
isotopies.
\end{proof}

\begin{figure}[ht!]
\centering
\includegraphics[width=.65\textwidth]{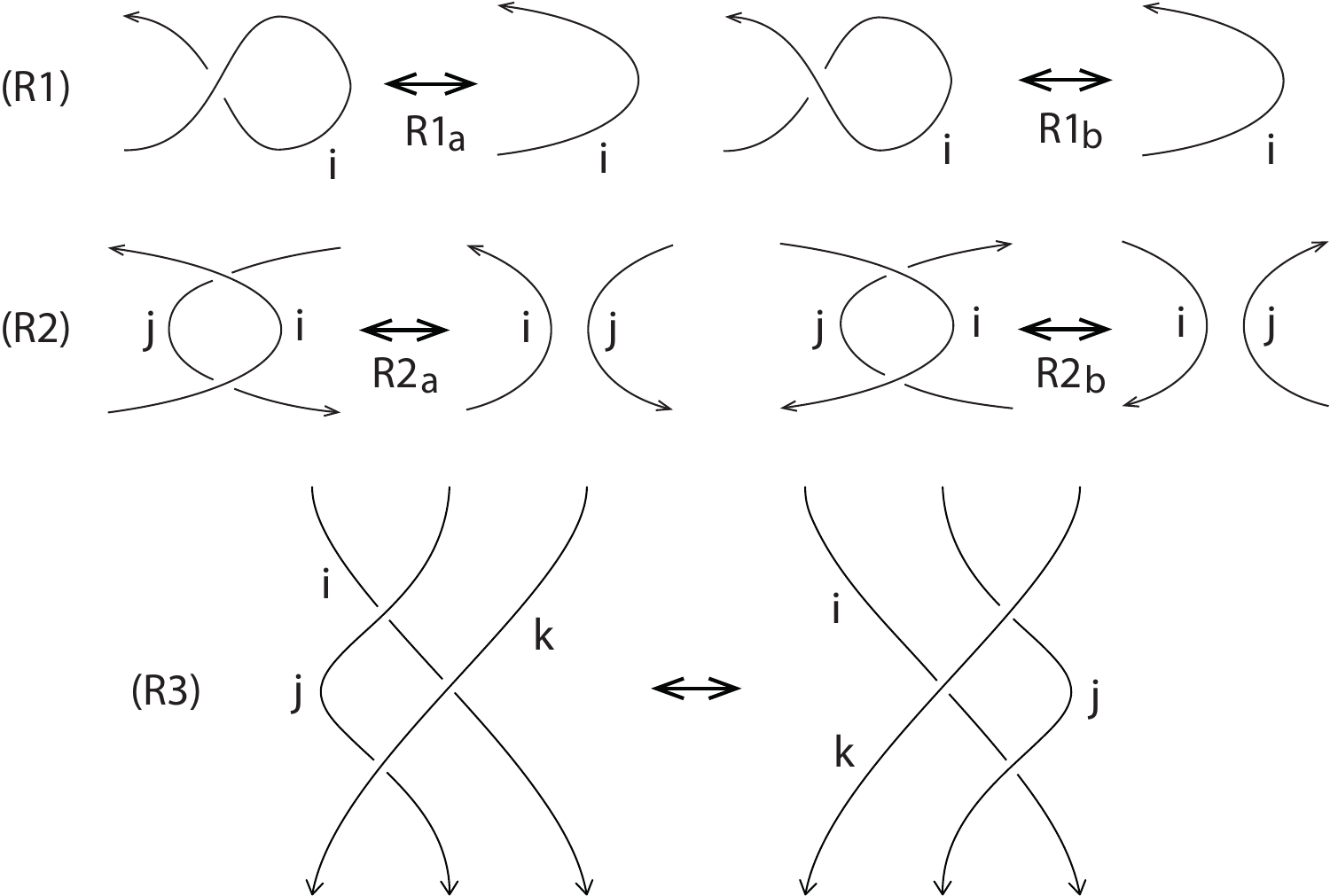}
\caption{Generalized Reidemeister moves for $\Omega$-colored links ($i,j,k \in \Omega$)}\label{fig-ormoves}
\end{figure}

Let $(D,\sigma)$ be an $\Omega$-colored link diagram in $\mathbb R^2$. 
We divide $(D,\sigma)$ at every crossing point (viewed as a 4-valent vertex) to obtain a collection $\mathcal S(D)$ of {\it semiarcs}. 
The {\it co-orientation} of a semiarc $s$ of $(D,\sigma)$ is a normal vector $\mathbf{n}$, indicated by a small arrow whose initial point lies on $s$, such that the ordered pair (orientation of $s$, $\mathbf{n}$) agrees with the right-handed orientation of $\mathbb R^2$ (cf.~Figure~\ref{fig-dibiquandle-col}). 
An $\Omega$-colored link diagram $(D,\sigma)$ in $\mathbb R^2$ is said to be {\it co-oriented} if each semiarc of $(D,\sigma)$ is equipped with a co-orientation. 
Throughout this paper, unless otherwise specified, every $\Omega$-colored link diagram is assumed to be co-oriented. 

At a crossing, if the ordered pair consisting of the co-orientation of the over-arc and that of the under-arc agrees with the right-handed orientation of $\mathbb R^2$, then the crossing is called {\it positive}; otherwise, it is called {\it negative}. 
In Figure~\ref{fig-dibiquandle-col}, the crossing on the left is positive, while that on the right is negative. 
Among the four quadrants around a crossing $c$, the unique quadrant from which both co-orientations of the adjacent arcs point outward is called the {\it source region} of $c$, and is marked with the symbol $\circledS$, if necessary.

\begin{figure}[ht!]
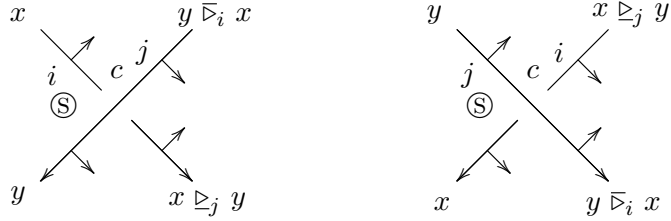

\centerline{
\xy (-10,10);(-2,2) **@{-},(10,-10);(2,-2)**@{-} **@{-} ?<*\dir{<}, 
(10,10);(-10,-10) **@{-} **@{-} ?>*\dir{>},(12,-13) *{x~\underline{\triangleright}_j~y},
(-13,-12) *{y},(-13,12) *{x}, (13,12) *{y~\overline{\triangleright}_i~x},(-8.5,4) *{i},(3.5,7) *{j},(-3,9);(-6,6)**@{-} ?<*\dir{<},(0,4.5) *{c},
(-3,-9);(-6,-6)**@{-} ?<*\dir{<},(-7,0)*{\circledS},
(9,-3);(6,-6)**@{-} ?<*\dir{<},(9,3);(6,6)**@{-} ?<*\dir{<},
\endxy\qquad\qquad\qquad
\xy (10,-10);(-10,10) **@{-} **@{-} ?<*\dir{<},(10,10);(2,2) **@{-},
(-2,-2);(-10,-10) **@{-} **@{-} ?>*\dir{>},(12,-13) *{y~\overline{\triangleright}_i~x}, (-13,12) *{y},(-12,-13) *{x},(13,12) *{x~\underline{\triangleright}_j~y},(-8.5,4) *{j},(3.5,7) *{i},(-3,9);(-6,6)**@{-} ?<*\dir{<},(0,4.5) *{c},
(-3,-9);(-6,-6)**@{-} ?<*\dir{<},(-7,0)*{\circledS},
(9,-3);(6,-6)**@{-} ?<*\dir{<},(9,3);(6,6)**@{-} ?<*\dir{<},
\endxy}
\vskip.1cm
\caption{$\Omega$-multibiquandle coloring condition at a crossing $c$ ($i,j \in \Omega$)}\label{fig-dibiquandle-col}
\end{figure}

\begin{definition}\label{def-mbqd-coloring}
Let $(D,\sigma)$ be an $\Omega$-colored link diagram and let 
$(X,\{\underline{\triangleright}_i,\overline{\triangleright}_i\}_{i\in\Omega})$ 
be an $\Omega$-multibiquandle. 
An {\it $\Omega$-multibiquandle coloring} of $(D,\sigma)$ by $X$, or simply an {\it $\Omega$-multibiquandle $X$-coloring}, is a map 
$\mathcal{C}_\sigma: \mathcal{S}(D) \longrightarrow X$
such that at each crossing the $\Omega$-multibiquandle coloring condition shown in Figure~\ref{fig-dibiquandle-col} is satisfied. 
For a semiarc $s$, the element $\mathcal{C}_\sigma(s)\in X$ is called the {\it color} of $s$. 
The resulting {\it $X$-colored link diagram} is denoted by $(D,\sigma;\mathcal{C}_\sigma)$, or simply by $(D;\mathcal{C})$ when $\sigma$ is clear from the context. 
The set of all $\Omega$-multibiquandle $X$-colorings of $(D,\sigma)$ is denoted by ${\rm Col}_X^\Omega(D,\sigma)$, or ${\rm Col}_X^\Omega(D)$.
\end{definition}

Condition {\bf MBQ1} in Definition~\ref{def-n-dibiqdle} corresponds to the generalized Reidemeister moves ${\bf R1_a}$, ${\bf R1_b}$, ${\bf R2_a}$, ${\bf R2_b}$, and {\bf R3} with $i=j=k$. 
Similarly, {\bf MBQ2} and {\bf MBQ3} correspond to the $X$-colored generalized Reidemeister moves ${\bf R2_a}$, ${\bf R2_b}$, and {\bf R3}, respectively, with $i,j,k \in \Omega$, as illustrated in Figure~\ref{fig-dibiqdl-rm}.

\begin{definition}
Let \((D,\sigma)\) be an \(\Omega\)-colored link diagram. Each pair of local diagrams displayed in Figure~\ref{fig-dibiqdl-rm} is understood to lie
in a move disk \(B\), with the diagrams and their \(\Omega\)-labels
identified outside \(B\). Let \(C\in{\rm Col}_X^\Omega(D,\sigma)\). A coloring
\(C'\in{\rm Col}_X^\Omega(D',\sigma')\) is called a
\emph{local extension of \(C\) across the move} if \(C\) and \(C'\)
agree on the identified parts outside \(B\). 
\end{definition}

Equivalently, the two colorings \(C\) and \(C'\) assign the same colors to the corresponding strands meeting \(\partial B\), while the colors on the strands inside \(B\) satisfy the local \(\Omega\)-multibiquandle coloring conditions in Figure~\ref{fig-dibiquandle-col}. Thus only the colors on the strands created or modified inside the move disk are to be determined.

Let \(X\) and \(Y\) be \(\Omega\)-multibiquandles and let
\(f\colon X\longrightarrow Y\)
be an \(\Omega\)-multibiquandle homomorphism. For each coloring
\(C\colon\mathcal S(D)\longrightarrow X\) in \({\rm Col}_X^\Omega(D,\sigma),\)
define
\begin{equation}\label{eq:ind-color-map}
f_*^{D,\sigma}\colon {\rm Col}_X^\Omega(D,\sigma) \longrightarrow {\rm Col}_Y^\Omega(D,\sigma)~\text{by}~f_*^{D,\sigma}(C)=f\circ C.
\end{equation}
Equivalently, for every semiarc \(x\in\mathcal S(D)\),
\(\bigl(f_*^{D,\sigma}(C)\bigr)(x)=f\bigl(C(x)\bigr).\)
Since \(f\) preserves all the indexed multibiquandle operations,
\(f\circ C\) satisfies the coloring conditions at every crossing; hence
\(f_*^{D,\sigma}\) is well defined. Moreover,
\((\id_X)_*^{D,\sigma}= \id_{{\rm Col}_X^\Omega(D,\sigma)}\) and \((k\circ h)_*^{D,\sigma} =k_*^{D,\sigma}\circ h_*^{D,\sigma}\),
where \(h\colon X\to Y\) and \(k\colon Y\to Z\) are \(\Omega\)-multibiquandle homomorphisms. We use this notation throughout.

\begin{theorem}[Reidemeister extension bijection]
\label{thm-n-mbqd-cinv}
Let \((D,\sigma)\) and \((D',\sigma')\) differ by one generalized
Reidemeister move in Figure~\ref{fig-ormoves}. For every \(\Omega\)-multibiquandle \(X\) and every coloring \(C\in{\rm Col}_X^\Omega(D,\sigma),\)
there exists a unique local extension \(C'\in{\rm Col}_X^\Omega(D',\sigma')\) of \(C\) across the move. Consequently, the assignment
\begin{equation}\label{eq-re-bij}
\Phi_X\colon
{\rm Col}_X^\Omega(D,\sigma)
\longrightarrow
{\rm Col}_X^\Omega(D',\sigma'),
\qquad
C\longmapsto\Phi_X(C)=C',
\end{equation}
is a bijection.
Moreover, these bijections are natural with respect to homomorphisms
of \(\Omega\)-multibiquandles. That is, for every homomorphism
\(f\colon X\to Y\), \(\Phi_Y\circ f_*^{D,\sigma}=f_*^{D',\sigma'}\circ\Phi_X:\)
\[
\begin{tikzcd}[row sep=1.8em, column sep=3.6em]
{\rm Col}_X^\Omega(D,\sigma)
  \arrow[r, "{\Phi_X}"]
  \arrow[d, "{f_*^{D,\sigma}}"']
& {\rm Col}_X^\Omega(D',\sigma')
  \arrow[d, "{f_*^{D',\sigma'}}"]
\\
{\rm Col}_Y^\Omega(D,\sigma)
  \arrow[r, "{\Phi_Y}"]
& {\rm Col}_Y^\Omega(D',\sigma').
\end{tikzcd}
\]
Consequently, the cardinal number
\(
\bigl|{\rm Col}_X^\Omega(D,\sigma)\bigr|
\)
is invariant under generalized Reidemeister moves, and hence is an
invariant of the \(\Omega\)-colored link represented by
\((D,\sigma)\).
\end{theorem}

\begin{figure}[t!]
\centering
\includegraphics[width=.75\textwidth]{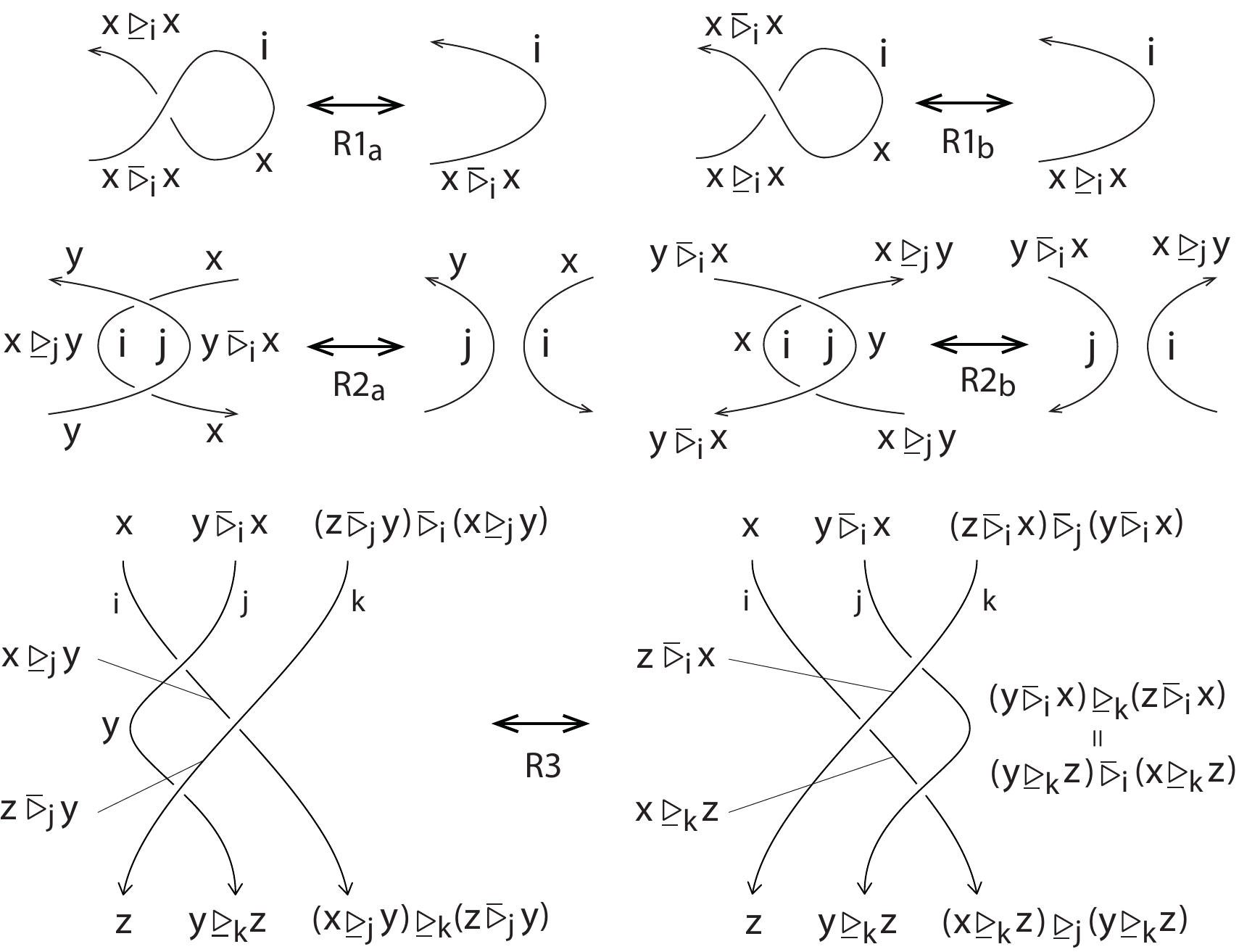}
\caption{$\Omega$-multibiquandle colored generalized Reidemeister moves
($i,j,k \in \Omega$).  The internal labels ${\bf R1_a},{\bf R1_b},
{\bf R2_a},{\bf R2_b},{\bf R3}$ correspond, up to the displayed rotations
and planar isotopy, to Polyak's
$\Omega_{1a},\Omega_{1c},\Omega_{2c},\Omega_{2d},\Omega_{3b}$,
respectively \cite{Polyak}.}\label{fig-dibiqdl-rm}
\end{figure}

\begin{proof}
Let \(B\) be the move disk introduced immediately before the theorem.
The diagrams \(D\) and \(D'\), as well as their component labels, are
identified outside \(B\).
Up to the rotations, planar isotopies, and internal notation recorded in
the caption of Figure~\ref{fig-dibiqdl-rm}, the displayed moves contain
Polyak's five-move generating set
\(
\{\Omega_{1a},\Omega_{1c},\Omega_{2c},\Omega_{2d},\Omega_{3b}\}
\)
\cite[Theorem~1.2]{Polyak}. Hence it suffices to establish the unique
local-extension property for these generators and their inverses.
Polyak's decompositions preserve the oriented strands, and assigning to
each strand the fixed label of its component makes every intermediate
move color-preserving.

\medskip
\noindent\emph{The R1 moves.}
Fix \(i\in\Omega\). By \textbf{MBQ1} of Definition~\ref{def-n-dibiqdle},
the fixed-index structure
\(
\bigl(X,\underline{\triangleright}_i,
        \overline{\triangleright}_i\bigr)
\)
is a biquandle. In particular, \textbf{BQ1} of Definition~\ref{def-biqdle} gives
\(
x\mathbin{\underline{\triangleright}_i}x
=
x\mathbin{\overline{\triangleright}_i}x
\, (x\in X).
\)
The common value therefore defines the \(i\)-th \emph{diagonal map}
\(
\pi_i\colon X\longrightarrow X\) given by \(\pi_i(x)
=
x\mathbin{\underline{\triangleright}_i}x
=
x\mathbin{\overline{\triangleright}_i}x.
\)
Notice that \textbf{BQ1} does not assert that \(\pi_i(x)=x\); it only
asserts that the two diagonal outputs are equal. We prove bijectivity
directly from the stated axioms. Write $U(a,b)=a\underline{\triangleright}_i b$
and $O(a,b)=a\overline{\triangleright}_i b$ for this paragraph.
Since $H_{ii}(a,a)=(\pi_i(a),\pi_i(a))$, the injectivity of $H_{ii}$
implies that of $\pi_i$. Given $y\in X$, choose $(a,b)$ with
$H_{ii}(a,b)=(y,y)$, so that $U(a,b)=O(b,a)=y$.
The first exchange law, applied to $(a,b,a)$, gives
\[
 \pi_i(y)=U(U(a,b),U(a,b))
 =U(U(a,a),O(b,a))=U(\pi_i(a),y).
\]
Since $\pi_i(y)=U(y,y)$ and $U(-,y)$ is injective by \textbf{BQ2},
we obtain $\pi_i(a)=y$. Thus $\pi_i$ is also surjective.
This proves the existence and uniqueness needed for an R1 extension,
without any finiteness assumption on $X$.

More explicitly, in each of the two R1 diagrams displayed in
Figure~\ref{fig-dibiqdl-rm}, let \(x\) be the color on the loop and let
\(y\) be the color on the corresponding exterior semiarc. Then the two
crossing relations give
\(
y
=
x\mathbin{\underline{\triangleright}_i}x
=
x\mathbin{\overline{\triangleright}_i}x
=
\pi_i(x).
\)
Since \(\pi_i\) is bijective, the prescribed exterior color
\(y\) determines the unique loop color \(x=\pi_i^{-1}(y).\)
Consequently, a coloring of the straight side extends uniquely to the
kinked side. In the reverse direction, deleting the kink is restriction,
and these two operations are mutually inverse.

\medskip
\noindent\emph{The R2 moves.}
For \(i,j\in\Omega\), consider the mixed switch
\(
H_{ij}\colon X\times X\longrightarrow X\times X\) defined by \(
H_{ij}(x,y)
=
\bigl(
y\mathbin{\overline{\triangleright}_i}x,\,
x\mathbin{\underline{\triangleright}_j}y
\bigr).
\)
By \textbf{MBQ2} of Definition~\ref{def-n-dibiqdle}, $H_{ij}$ is bijective.
Use the coordinate orders in Figure~\ref{fig-dibiqdl-rm}.
If the prescribed exterior pair is $(a,b)$, the internal pair in
${\bf R2_a}$ is $H_{ij}(a,b)$, and the second crossing returns
$H_{ij}^{-1}H_{ij}(a,b)=(a,b)$.
In ${\bf R2_b}$ the internal pair is $H_{ij}^{-1}(a,b)$, and the
second crossing returns $H_{ij}H_{ij}^{-1}(a,b)=(a,b)$.
Bijectivity makes both internal pairs unique, so restriction and
extension are inverse.

\medskip
\noindent\emph{The R3 move.}
Consider the displayed braid-type generator R3 in
Figure~\ref{fig-dibiqdl-rm}, and
let \(x,b,z\) be the prescribed colors at its upper left, upper
middle, and lower left boundary endpoints. By \textbf{BQ2}, there is a
unique \(y\in X\) such that
\(y\mathbin{\overline{\triangleright}_i}x=b\).
Starting from $x,b,z$, the crossing rules and \textbf{BQ2} determine
the internal colors on either side. The first and third identities
below compare the boundary values on the $i$- and $k$-strands.
For the $j$-strand, the second identity shows that the proposed lower
middle boundary value $y\mathbin{\underline{\triangleright}_k}z$
satisfies the remaining middle crossing relation. By \textbf{BQ2}, the
map $w\mapsto w\mathbin{\overline{\triangleright}_i}
(x\mathbin{\underline{\triangleright}_k}z)$ is bijective, making this value unique,
so that boundary value also agrees. The required identities are
\begin{align}
&(x\mathbin{\underline{\triangleright}_j}y)
 \mathbin{\underline{\triangleright}_k}
 (z\mathbin{\overline{\triangleright}_j}y)
 =
 (x\mathbin{\underline{\triangleright}_k}z)
 \mathbin{\underline{\triangleright}_j}
 (y\mathbin{\underline{\triangleright}_k}z),
 \label{eq:r3-boundary-i}\\
&(y\mathbin{\overline{\triangleright}_i}x)
 \mathbin{\underline{\triangleright}_k}
 (z\mathbin{\overline{\triangleright}_i}x)
 =
 (y\mathbin{\underline{\triangleright}_k}z)
 \mathbin{\overline{\triangleright}_i}
 (x\mathbin{\underline{\triangleright}_k}z),
 \label{eq:r3-boundary-j}\\
&(z\mathbin{\overline{\triangleright}_j}y)
 \mathbin{\overline{\triangleright}_i}
 (x\mathbin{\underline{\triangleright}_j}y)
 =
 (z\mathbin{\overline{\triangleright}_i}x)
 \mathbin{\overline{\triangleright}_j}
 (y\mathbin{\overline{\triangleright}_i}x).
 \label{eq:r3-boundary-k}
\end{align}
Equation~\eqref{eq:r3-boundary-i} is the first identity of
Lemma~\ref{lem:mbq3-reindexing} for the ordered pair \((j,k)\).
Equation~\eqref{eq:r3-boundary-j} is the second identity of that lemma
for \((i,k)\), after the substitution
\(
(x,y,z)\longmapsto(y,x,z).
\)
Finally, Equation~\eqref{eq:r3-boundary-k} is the third identity in
\textbf{MBQ3} for the ordered pair \((i,j)\), after the substitution
\(
(x,y,z)\longmapsto(z,x,y),
\)
read from right to left. Thus all corresponding boundary colors agree
on the two sides. Since all internal colors are already uniquely
determined by the crossing rules, every coloring on either side has a
unique local extension to the other side.

Reading the preceding constructions backward gives the corresponding
unique extensions for the inverse generators. Rotations and planar
isotopies merely change the positions of the local pictures, and every
other oriented Reidemeister move is generated by the displayed moves
and their inverses. Composing the corresponding local extensions
therefore gives, for every generalized Reidemeister move, a bijection
\(
\Phi_X\colon
{\rm Col}_X^\Omega(D,\sigma)
\longrightarrow
{\rm Col}_X^\Omega(D',\sigma').
\)

\medskip
It remains to verify naturality.  Let \(f\colon X\to Y\) be a
homomorphism of \(\Omega\)-multibiquandles.  Since \(f\) preserves every
indexed operation, for all \(i,j\in\Omega\) we have
\(f\circ\pi_i^X=\pi_i^Y\circ f\) and \((f\times f)\circ H_{ij}^X = H_{ij}^Y\circ(f\times f)\).
Since the diagonal maps and mixed switches are bijective, these identities
also imply
\(f\circ(\pi_i^X)^{-1} = (\pi_i^Y)^{-1}\circ f\) and \((f\times f)\circ(H_{ij}^X)^{-1} = (H_{ij}^Y)^{-1}\circ(f\times f)\).
Moreover, applying \(f\) to both sides of each of the three
\textbf{MBQ3} identities displayed above gives the corresponding
identity in \(Y\).

Now let \(B\) be the disk supporting the move, and let
\(C\in{\rm Col}_X^\Omega(D,\sigma).\)
For every semiarc \(s\) outside \(B\), the definition of the extension
gives \(\Phi_X(C)(s)=C'(s)=C(s)\).  Hence
\(\bigl(f_*^{D',\sigma'}(\Phi_X(C))\bigr)(s) = f\bigl(\Phi_X(C)(s)\bigr) = f(C(s)) = \bigl(f_*^{D,\sigma}(C)\bigr)(s)\).
Thus the two colorings agree outside the move disk.

Inside \(B\), the preceding commutation identities show, move by move,
that applying \(f\) commutes with the rule determining the local colors.
For an R1 move, the new color is obtained using
\(\pi_i^X\) or \((\pi_i^X)^{-1}\), and its image under \(f\) is the
color obtained using the corresponding map
\(\pi_i^Y\) or \((\pi_i^Y)^{-1}\).  For an R2 move, the ordered pair of
local colors is obtained using \(H_{ij}^X\) or
\((H_{ij}^X)^{-1}\), and applying \(f\times f\) gives precisely the
pair determined by \(H_{ij}^Y\) or \((H_{ij}^Y)^{-1}\).  Finally, for
an R3 move, applying \(f\) to each successive indexed-operation
expression gives the corresponding expression in \(Y\), while the
three \textbf{MBQ3} identities in \(Y\) ensure the required agreement
of the resulting boundary colors. Consequently,
\(f_*^{D',\sigma'}\bigl(\Phi_X(C)\bigr)\)
is a \(Y\)-coloring of \((D',\sigma')\) that extends
\(f_*^{D,\sigma}(C)\) across the move.  By the uniqueness of the local
\(Y\)-valued extension, it must equal
\(\Phi_Y\bigl(f_*^{D,\sigma}(C)\bigr).\)
Therefore
\(f_*^{D',\sigma'}\circ\Phi_X=\Phi_Y\circ f_*^{D,\sigma},\)
which proves naturality.

Finally, the bijection \(\Phi_X\) gives \(\bigl|{\rm Col}_X^\Omega(D,\sigma)\bigr|=\bigl|{\rm Col}_X^\Omega(D',\sigma')\bigr|.\)
Hence the cardinal number is invariant under generalized Reidemeister
moves and thus depends only on the \(\Omega\)-colored link represented
by \((D,\sigma)\). This completes the proof.
\end{proof}

\begin{corollary}\label{cor-n-mbqd-invariant}
For any finite $\Omega$-multibiquandle $X$, the coloring number
\(|{\rm Col}_X^\Omega(L,\sigma)|:=|{\rm Col}_X^\Omega(D,\sigma)|\)
is a well-defined invariant of the $\Omega$-colored link $L$.
\end{corollary}

\begin{proof}
Choose a diagram $(D,\sigma)$ of the colored link.  If
$(D',\sigma')$ is another diagram, Theorem~\ref{thm-n-mbqd-cinv}
gives a bijection between the two coloring sets.  Since $X$ is finite,
the two sets have the same finite cardinality, so the displayed number
is independent of the chosen diagram.
\end{proof}

\begin{example}\label{exmp-odl631}
Let $L_1$ and $L_2$ be the two $3$-colored links with diagrams $D_1$
and $D_2$ shown in Figure~\ref{fig-ordl2}, and use the $3$-parallel
multibiquandle $X=\mathbb Z_3$ from Example~\ref{exa-3-pall}.  Let
$\Omega=\{1,2,3\}$ and let $a,b,c,d,e,f,g,h,p,q,r,s$ denote the
semiarcs indicated in the figure.  The crossingless diagram $D_1$ has
three freely colored semiarcs, and hence
\(
 |{\rm Col}_{X}^{\Omega}(L_1)|=3^3=27.
\)
\begin{figure}[ht]
\centering
\includegraphics[width=.46\textwidth]{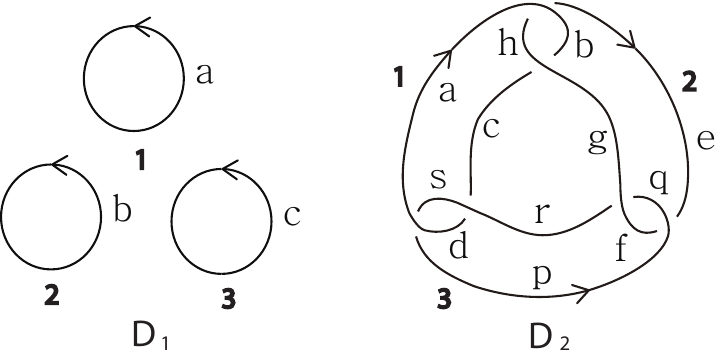}
\caption{$3$-colored links of three components}\label{fig-ordl2}
\end{figure}

For $D_2$, the twelve crossing equations are
\begin{align*}
e&=h\underline{\triangleright}_1b,&
a&=b\overline{\triangleright}_2h,&
c&=b\underline{\triangleright}_2h,&
g&=h\overline{\triangleright}_1b,\\
f&=e\underline{\triangleright}_3q,&
p&=q\overline{\triangleright}_2e,&
r&=q\underline{\triangleright}_2g,&
f&=g\overline{\triangleright}_3q,\\
p&=s\underline{\triangleright}_1a,&
d&=a\overline{\triangleright}_3s,&
d&=c\underline{\triangleright}_3s,&
r&=s\overline{\triangleright}_1c.
\end{align*}
Using $x\underline{\triangleright}_n y=x$ for all $n$,
$x\overline{\triangleright}_2y=x$, and
$x\overline{\triangleright}_1y=x\overline{\triangleright}_3y=2x-y$,
the twelve equations become
\[
\begin{aligned}
e&=h,&a&=b,&c&=b,&g&=2h-b,\\
f&=e,&p&=q,&r&=q,&f&=2g-q,\\
p&=s,&d&=2a-s,&d&=c,&r&=2s-c.
\end{aligned}
\]
First, $a=b=c$ and $d=c$ give $a=b=c=d=k$.
Then $d=2a-s$ forces $s=k$, and $p=s$, $p=q$, and $r=q$
give $p=q=r=s=k$. Also $e=h$ and $f=e$ give $e=f=h=\ell$,
and $g=2h-b$ gives $g=2\ell-k$. The remaining equations hold
for every $k,\ell\in\mathbb Z_3$, because
$2g-q=4\ell-3k=\ell=f$ and $2s-c=2k-k=k=r$.
Thus these equations have exactly the solutions
\(
 (a,b,c,d,e,f,g,h,p,q,r,s)
 =(k,k,k,k,\ell,\ell,2\ell-k,\ell,k,k,k,k)\) for each \(k,\ell\in\mathbb Z_3.
\) Therefore $|{\rm Col}_{X}^{\Omega}(L_2)|=3^2=9$.  In particular,
$L_1$ and $L_2$ are not equivalent as $3$-colored links.
\end{example}

The coloring bijection will be used below for the affine
$A$-module multibiquandle.  Passing from the set of colorings to its
representing module retains linear information that can subsequently be
specialized or decategorified by Fitting ideals.

\section{\texorpdfstring{$(G,u)$}{(G,u)}-deformed Alexander modules of
\texorpdfstring{$G$}{G}-colored links}\label{sect-abq}

In this section, we define a group $G$-multibiquandle module $M_u(D,\sigma;R[G])$ associated with a given $G$-colored link diagram $(D,\sigma)$, where $G$ is a commutative group and $R$ is a commutative ring with identity, and we present invariants of the $G$-colored link $(L,\sigma)$ in $S^3$ represented by $(D,\sigma)$ derived from this module. For convenience, we begin by recalling the presentation of a finitely generated $R$-module. 

Let $M$ be a finitely generated module over a commutative ring $R$ with identity $1$. A {\it finite presentation} of $M$ is an exact sequence
\[F \overset{\alpha}{\longrightarrow} E \overset{\phi}{\longrightarrow} M \longrightarrow 0,\]
where $F$ and $E$ are free $R$-modules with finite bases. 
Thus $M \cong {\rm coker}(\alpha) = E / {\rm Im}(\alpha)$. 

Let $B=\{f_1, f_2, \ldots, f_n\}$ and $B'=\{e_1, e_2, \ldots, e_m\}$ be bases of $F$ and $E$, respectively. 
Suppose that for each $j=1,\ldots,n$, 
\(\alpha(f_j) = \sum_{i=1}^m a_{ij} e_i\) and \(a_{ij}\in R\).
The $m\times n$ matrix
\[
A=[a_{ij}]_{m\times n}
= \begin{bmatrix}
a_{11} & a_{12} & \cdots & a_{1n} \\
a_{21} & a_{22} & \cdots & a_{2n} \\
\vdots & \vdots & \ddots & \vdots \\
a_{m1} & a_{m2} & \cdots & a_{mn}
\end{bmatrix}
\]
is called a {\it presentation matrix} (or {\it relation matrix}) for $M$ with respect to the bases $B$ and $B'$. 

Since $\phi$ is surjective, the elements $\phi(e_1),\ldots,\phi(e_m)$
generate $M$. The $j$-th column records the relation
$\sum_{i=1}^m a_{ij}\phi(e_i)=0$, and these columns generate all
relations among the chosen generators because
$\ker\phi=\operatorname{im}\alpha$. 
We now recall the following {\it Tietze moves for finite module presentations}.

\begin{theorem}\cite[Theorem 6.1]{Lic}
\label{thm-mod-pres}
Any two finite presentation matrices of isomorphic modules over a
commutative ring are related by a finite sequence of the following moves
and their inverses:
\begin{itemize}
\item[(1)] permutation of rows or columns;
\item[(2)] stabilization by an identity entry,
$A\mapsto\begin{bmatrix}A&0\\0&1\end{bmatrix}$;
\item[(3)] addition of a zero column;
\item[(4)] addition of a scalar multiple of one row (respectively column)
to another row (respectively column).
\end{itemize}
Each specified move induces an explicit isomorphism of cokernels.
\end{theorem}

Let $A$ and $A'$ be matrices with entries in $R$. 
We say that $A$ is {\it equivalent} to $A'$, denoted $A \sim A'$, if there exists a finite sequence of matrices 
$
A = A_1, A_2, \ldots, A_n = A'
$
such that $A_{i+1}$ is obtained from $A_i$ by one of the matrix moves $(1)$–$(4)$ in Theorem~\ref{thm-mod-pres} or their inverses. 

\begin{proposition}\label{prop-eq-mat}
Let ${\bf a}$ and ${\bf b}$ be arbitrary row and column vectors, respectively, and let $\lambda$ be a unit in $R$. 
Then the following matrix equivalences hold:
\begin{itemize}
\item[(5)] $A \sim \left[\begin{array}{cc}
A & {\bf 0}\\
{\bf a} & 1
\end{array}\right].$\quad 
(6) $A \sim \left[\begin{array}{cc}
A & {\bf b}\\
{\bf 0} & 1
\end{array}\right].$\quad
(7) $\left[\begin{array}{cc}
A & {\bf b}
\end{array}\right] \sim 
\left[\begin{array}{cc}
A & \lambda{\bf b}
\end{array}\right].$\quad 
(8) $\left[\begin{array}{c}
A \\[2pt]
{\bf a}
\end{array}\right] \sim 
\left[\begin{array}{c}
A \\[2pt]
\lambda{\bf a}
\end{array}\right].$
\end{itemize}
\end{proposition}

\begin{proof} 
Let $A$ be an $m\times n$ matrix. 
Write ${\bf a}=(a_1,\ldots,a_n)$ and ${\bf b}=(b_1,\ldots,b_m)^{\mathsf T}$.
After stabilization by (2), apply
$C_j\leftarrow C_j+a_jC_{n+1}$ for $1\le j\le n$ to obtain (5),
and $R_i\leftarrow R_i+b_iR_{m+1}$ for $1\le i\le m$ to obtain (6).
These are respectively $n$ and $m$ applications of (4), so
\[A \overset{(2)}{\sim} \left[\begin{array}{cc}
A & {\bf 0}\\
{\bf 0} & 1
\end{array}\right] \overset{n\times (4)}{\sim} \left[\begin{array}{cc}
A & {\bf 0}\\
{\bf a} & 1
\end{array}\right],\qquad
A \overset{(2)}{\sim} \left[\begin{array}{cc}
A & {\bf 0}\\
{\bf 0} & 1
\end{array}\right] \overset{m\times (4)}{\sim} \left[\begin{array}{cc}
A & {\bf b}\\
{\bf 0} & 1
\end{array}\right].\]

The matrix equivalences in (7) and (8) are obtained as follows:
\begin{align*}
\begin{bmatrix}
A & {\bf b}
\end{bmatrix} 
&\;\overset{(3)}{\sim}\;
\begin{bmatrix}
A & {\bf b} & {\bf 0}
\end{bmatrix} 
\;\overset{(4)}{\sim}\;
\begin{bmatrix}
A & {\bf b} & \lambda{\bf b}
\end{bmatrix}
\;\overset{(1)}{\sim}\;
\begin{bmatrix}
A & \lambda{\bf b} & {\bf b}
\end{bmatrix}\\
&\;\overset{(4)}{\sim}\;
\begin{bmatrix}
A & \lambda{\bf b} & {\bf b}-\lambda^{-1}\lambda{\bf b}
\end{bmatrix} 
=\begin{bmatrix}
A & \lambda{\bf b} & {\bf 0}
\end{bmatrix} 
\;\overset{(3)}{\sim}\;
\begin{bmatrix}
A & \lambda{\bf b}
\end{bmatrix}.
\\[1.2ex]
\begin{bmatrix}
A \\[2pt]
{\bf a}
\end{bmatrix} 
&\;\overset{(6)}{\sim}\;
\begin{bmatrix}
A & {\bf 0}\\
{\bf a} & -\lambda^{-1}\\
{\bf 0} & 1
\end{bmatrix}
\;\overset{(4)}{\sim}\;
\begin{bmatrix}
A & {\bf 0}\\
{\bf a} & -\lambda^{-1}\\
\lambda{\bf a} & 0
\end{bmatrix}
\;\overset{(1)}{\sim}\;
\begin{bmatrix}
A & {\bf 0}\\
\lambda{\bf a} & 0\\
{\bf a} & -\lambda^{-1}
\end{bmatrix}\\
&\;\overset{(4)}{\sim}\;
\begin{bmatrix}
A & {\bf 0}\\
\lambda{\bf a} & 0\\
{\bf 0} & -\lambda^{-1}
\end{bmatrix}
\;\overset{(7)}{\sim}\;
\begin{bmatrix}
A & {\bf 0}\\
\lambda{\bf a} & 0\\
{\bf 0} & 1
\end{bmatrix}
\;\overset{(2)}{\sim}\;
\begin{bmatrix}
A \\[2pt]
\lambda{\bf a}
\end{bmatrix}.
\end{align*}
This completes the proof.
\end{proof}

Let $G$ be a multiplicative commutative group with identity $e$, $R$ a commutative ring with identity $1$, and $R[G]$ the group ring of $G$ over $R$. 
Note that $R[G]$ is a commutative ring with identity $1e$. 
Let $(D,\sigma)$ be a $G$-colored link diagram with $\sigma:\pi_0(D)\to G$, $\mathcal{S}(D)=\{x_1,\ldots,x_m\}$ the set of semiarcs of $D$, and $C(D)=\{c_1,\ldots,c_\ell\}$ the set of crossings of $D$. 
For each crossing $c_k$ ($1 \leq k \leq \ell$), let $R_1(c_k)$ and $R_2(c_k)$ denote the linear expressions defined in Figure~\ref{fig-qd-rels-0}, depending on whether $c_k$ is positive or negative, where $u$ is an invertible element of $R[G]$.

\begin{figure}[ht!]
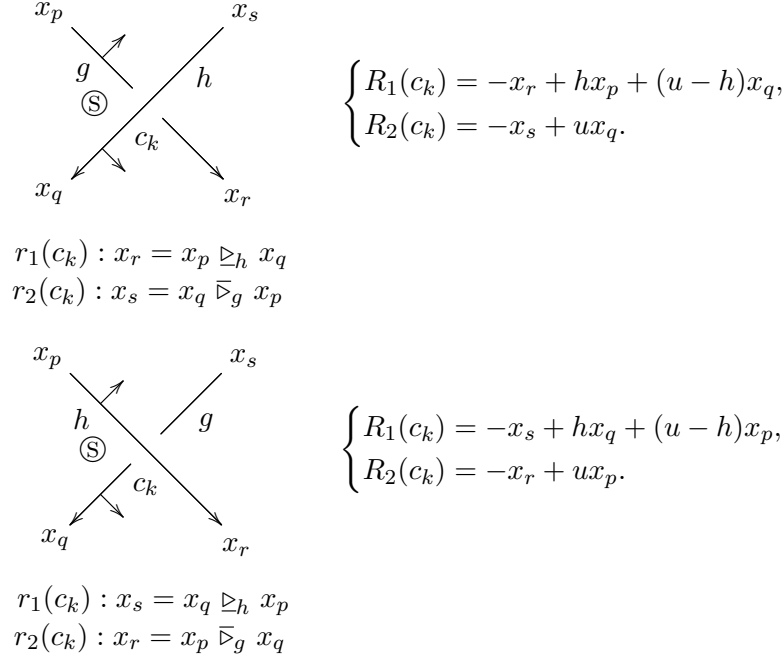

\centerline{\xy 
(-10,10);(-2,2)**@{-} **@{-},(10,-10);(2,-2)**@{-} **@{-} ?<*\dir{<}, 
(10,10);(-10,-10) **@{-} **@{-} ?>*\dir{>},(12,-12.5) *{x_{r}},
(-13,-12) *{x_{q}},(-13,12) *{x_{p}}, (13,12) *{x_{s}},(-8.5,4) *{g},(7.5,3.5) *{h},
(0.5,-20) *{r_1(c_k):x_r=x_p~\underline{\triangleright}_h~x_q}, 
(0,-25) *{r_2(c_k):x_s=x_q~\overline{\triangleright}_g~x_p},
(0,-5) *{c_k},(-3,9);(-6,6)**@{-} ?<*\dir{<},(11,-3.5) *{},(-8.5,-4) *{},
(-3,-9);(-6,-6)**@{-} ?<*\dir{<},(-7,0)*{\circledS},
(57,0)*{\begin{cases}
R_1(c_k)=-x_{r} + hx_{p} + (u-h)x_{q}, \\
R_2(c_k)=-x_{s} + ux_{q}.
\end{cases}},
\endxy}
\vskip.5cm
\centerline{\xy 
(10,-10);(-10,10) **@{-} **@{-} ?<*\dir{<},(10,10);(2,2)**@{-} **@{-},(-2,-2);(-10,-10) **@{-} **@{-} ?>*\dir{>},(12,-13) *{x_{r}}, (-13,12) *{x_{p}},(-12,-12) *{x_{q}},(13,12) *{x_{s}},(-8.5,4) *{h},(8,3.5) *{g},(9,-4) *{},
(1,-20) *{r_1(c_k):x_s=x_q~\underline{\triangleright}_h~x_p}, 
(0.5,-25) *{r_2(c_k):x_r=x_p~\overline{\triangleright}_g~x_q},
(0,-5) *{c_k},(-3,9);(-6,6)**@{-} ?<*\dir{<}, (-8.5,-4) *{},
(-3,-9);(-6,-6)**@{-} ?<*\dir{<},(-7,0)*{\circledS},
(57,0)*{\begin{cases}
R_1(c_k)=-x_{s} + hx_{q} + (u-h)x_{p},\\
R_2(c_k)=-x_{r} + ux_{p}.
\end{cases}},
\endxy}
\vskip.1cm
\caption{Linear expressions $R_1(c_k)$ and $R_2(c_k)$ associated with a crossing $c_k$}\label{fig-qd-rels-0}
\end{figure}

Let $F_u^m(D,\sigma)$ be the free $R[G]$-module of rank $2\ell$ with basis 
$
B=\{c_{11},c_{21},c_{12},c_{22},\ldots,c_{1\ell},c_{2\ell}\},
$
and let $E_u^m(D,\sigma)$ be the free $R[G]$-module of rank $m$ with basis 
$
B'=\{x_1,\ldots,x_m\}.
$
Thus 
\begin{equation}\label{defn-modls-1}
F_u^m(D,\sigma)=\bigoplus_{i=1}^{\ell}\big(R[G]c_{1i}\oplus R[G]c_{2i}\big), 
\qquad 
E_u^m(D,\sigma)=\bigoplus_{i=1}^m R[G]x_i.
\end{equation}
Define an $R[G]$-module homomorphism $\alpha_{D,\sigma}: F_u^m(D,\sigma)\to E_u^m(D,\sigma)$ by
\(\alpha_{D,\sigma}(c_{1k})=R_1(c_k)\) and \(\alpha_{D,\sigma}(c_{2k})=R_2(c_k)\)
for $k=1,\ldots,\ell$,
and extend $\alpha_{D,\sigma}$ linearly to all of $F_u^m(D,\sigma)$.  
Then the sequence
\begin{equation}\label{modl-pres-1}
F_u^m(D,\sigma)\overset{\alpha_{D,\sigma}}{\longrightarrow} E_u^m(D,\sigma) \overset{\phi_{D,\sigma}}{\longrightarrow} \mathrm{coker}(\alpha_{D,\sigma}) \longrightarrow 0
\end{equation}
is exact, and hence gives a finite presentation of the $R[G]$-module
\(\mathrm{coker}(\alpha_{D,\sigma}) =E_u^m(D,\sigma)/\mathrm{Im}(\alpha_{D,\sigma})\),
where $\phi_{D,\sigma}$ is the canonical quotient map defined by $\phi_{D,\sigma}(x)=[x]$ for all $x\in E_u^m(D,\sigma)$. We write
\[
 M_u(D,\sigma;A):=\mathrm{coker}(\alpha_{D,\sigma}),
 \qquad A:=R[G].
\]

\medskip
Now, for each $k=1,\ldots,\ell$, we rewrite the images of $\alpha_{D,\sigma}$ as
\begin{align*}
&\alpha_{D,\sigma}(c_{1k})
= R_1(c_k)
= a_{1k}x_{1}+a_{2k}x_{2}+\cdots+ a_{mk}x_{m},\\ 
&\alpha_{D,\sigma}(c_{2k})
= R_2(c_k)
= b_{1k}x_{1}+b_{2k}x_{2}+\cdots+ b_{mk}x_{m},
\end{align*}
where $a_{ij}, b_{ij} \in R[G]$. Here, we use the \emph{relations-as-columns} convention: a matrix
$M\in M_{a\times b}(R[G])$ represents an $R[G]$--linear map $R[G]^b\to R[G]^a$, and the $j$-th column of $M$ is the image of the $j$-th basis vector of the domain, expressed in the chosen basis of the codomain. Hence, the homomorphism $\alpha_{D,\sigma}$ in~\eqref{modl-pres-1} is represented with respect to the bases $B$ and $B'$ by the $m\times 2\ell$ matrix
\begin{equation}\label{pres-mat-1}
\mathcal{A}_u(D,\sigma)=
\begin{bmatrix}
a_{11} & b_{11} & a_{12} & b_{12} & \cdots & a_{1\ell} & b_{1\ell} \\
a_{21} & b_{21} & a_{22} & b_{22} & \cdots & a_{2\ell} & b_{2\ell} \\
\vdots & \vdots & \vdots & \vdots & \ddots & \vdots & \vdots \\
a_{m1} & b_{m1} & a_{m2} & b_{m2} & \cdots & a_{m\ell} & b_{m\ell}
\end{bmatrix}.
\end{equation}
Thus, $\mathcal{A}_u(D,\sigma)$ is a \emph{presentation matrix} (or {\it relation matrix}) for the $R[G]$-module $M_u(D,\sigma;A)$, and is denoted simply by $\mathcal{A}_u(D)$ when $\sigma$ is clear from the context.

\begin{remark}
\label{rem-square-presentation}
Let $q(D)$ be the number of components of $D$ that contain no crossing.
The crossing graph has $2\ell$ edge-semiarcs, and each vertex-free
component contributes one further semiarc.  Hence \(m=2\ell+q(D).\)
The matrix $\mathcal A_u(D,\sigma)$ is square precisely when
$q(D)=0$; in general it has $q(D)$ additional zero rows, one for each
crossingless component, and those rows give a free summand $A^{q(D)}$ in
\(M_u(D,\sigma;A).\)
\end{remark}

\begin{theorem}[Hom--coloring correspondence]
\label{thm:representability}
Let \((D,\sigma)\) be a \(G\)-colored link diagram with semiarc set
\(\mathcal S(D)\) and let \(N\) be an \(A\)-module equipped with the affine operations in \eqref{eq1-mdl-gpring} for the fixed \(u\in A^\times\). Let \(\phi_{D,\sigma}\colon E_u^m(D,\sigma)\longrightarrow M_u(D,\sigma;A)\) be the canonical quotient map.  Then the assignment
\[
\Theta_{D,\sigma,N}\colon
\Hom_A\bigl(M_u(D,\sigma;A),N\bigr)
\longrightarrow
{\rm Col}_N^G(D,\sigma),
\qquad
f\longmapsto \Theta_{D,\sigma,N}(f)=C_f,
\]
is a bijection, where \(C_f:\mathcal S(D)\to N\) is the \(G\)-multibiquandle \(N\)-coloring of \((D,\sigma)\) defined by \(C_f(x):=f(\phi_{D,\sigma}(x))=f([x])\) for all \(x\in\mathcal S(D).\)
Moreover, the family of bijections
\(\{\Theta_{D,\sigma,N}\}_N\) is natural in \(N\). More precisely,
let \(\psi\colon N\to N'\) be an \(A\)-linear map, and define
\(\psi_{\#}\colon \Hom_A\bigl(M_u(D,\sigma;A),N\bigr) \longrightarrow \Hom_A\bigl(M_u(D,\sigma;A),N'\bigr)\), \(f\longmapsto\psi\circ f\).
Then \(\Theta_{D,\sigma,N'}\circ\psi_{\#}=\psi_*^{D,\sigma}\circ\Theta_{D,\sigma,N}.\)  That is, the following diagram commutes:
\[
\begin{tikzcd}[row sep=1.8em, column sep=3.2em]
\Hom_A\bigl(M_u(D,\sigma;A),N\bigr)
  \arrow[r, "{\Theta_{D,\sigma,N}}"]
  \arrow[d, "{\psi_{\#}}"']
& {\rm Col}_N^G(D,\sigma)
  \arrow[d, "{\psi_*^{D,\sigma}}"]
\\
\Hom_A\bigl(M_u(D,\sigma;A),N'\bigr)
  \arrow[r, "{\Theta_{D,\sigma,N'}}"]
& {\rm Col}_{N'}^G(D,\sigma).
\end{tikzcd}
\]
\end{theorem}

\begin{proof}
Let \(f\colon M_u(D,\sigma;A)\longrightarrow N\) be an \(A\)-linear map.
We verify that \(C_f\) is a \(G\)-multibiquandle \(N\)-coloring of
\((D,\sigma)\). The claim is that for every crossing \(c_k\) with semiarc labels \(x_p,x_q,x_r,x_s\in \mathcal S(D)\) and strand labels \(g,h\in G\) as shown in Figure~\ref{fig-qd-rels-0}, \(C_f(x_p),C_f(x_q),C_f(x_r),C_f(x_s)\in N\) satisfies the coloring conditions in Figure~\ref{fig-dibiquandle-col} with the operations in \eqref{eq1-mdl-gpring}.

For every crossing \(c_k\), the relators \(R_1(c_k)\) and \(R_2(c_k)\)
belong to \(\im\alpha_{D,\sigma}\) by definition of \(\alpha_{D,\sigma}\). Hence
\(\phi_{D,\sigma}\bigl(R_1(c_k)\bigr)=
\phi_{D,\sigma}\bigl(R_2(c_k)\bigr)=0\)
in \(M_u(D,\sigma;A)=\mathrm{coker}(\alpha_{D,\sigma})\). Applying \(f\) gives
\(f\!\left(\phi_{D,\sigma}\bigl(R_1(c_k)\bigr)\right)
=f\!\left(\phi_{D,\sigma}\bigl(R_2(c_k)\bigr)\right)=f(0)=0.\)

Suppose first that \(c_k\) is the positive crossing. By linearity,
\begin{align*}
-C_f(x_r)+hC_f(x_p)+(u-h)C_f(x_q)
&=f\bigl(-[x_r]+h[x_p]+(u-h)[x_q]\bigr)\\
&=f\bigl(\phi_{D,\sigma}(R_1(c_k))\bigr)=0,\\
-C_f(x_s)+uC_f(x_q)
&=f\bigl(-[x_s]+u[x_q]\bigr)\\
&=f\bigl(\phi_{D,\sigma}(R_2(c_k))\bigr)=0.
\end{align*}
Thus
\begin{align*}
C_f(x_r)&=hC_f(x_p)+(u-h)C_f(x_q)
=C_f(x_p)\mathbin{\underline{\triangleright}_h}C_f(x_q),\\
C_f(x_s)&=uC_f(x_q)
=C_f(x_q)\mathbin{\overline{\triangleright}_g}C_f(x_p).
\end{align*}
These are precisely the positive-crossing conditions in
Figure~\ref{fig-dibiquandle-col}. At a negative crossing, the same
linearity argument applied to its displayed relators gives
\begin{align*}
-C_f(x_s)+hC_f(x_q)+(u-h)C_f(x_p)
&=f\bigl(\phi_{D,\sigma}(R_1(c_k))\bigr)=0,\\
-C_f(x_r)+uC_f(x_p)
&=f\bigl(\phi_{D,\sigma}(R_2(c_k))\bigr)=0.
\end{align*}
Consequently,
\begin{align*}
C_f(x_s)&=hC_f(x_q)+(u-h)C_f(x_p)
=C_f(x_q)\mathbin{\underline{\triangleright}_h}C_f(x_p),\\
C_f(x_r)&=uC_f(x_p)
=C_f(x_p)\mathbin{\overline{\triangleright}_g}C_f(x_q).
\end{align*}
Hence \(C_f\) satisfies the required coloring conditions in Figure~\ref{fig-dibiquandle-col} at every crossing of \(D\), and thus
\(C_f\in{\rm Col}_N^G(D,\sigma).\)
Notice also that a crossingless component contributes one free semiarc
generator and no crossing relation, exactly as required for an arbitrary
color on that component. Therefore \(\Theta_{D,\sigma,N}\) is
well-defined.

Now let \(C\colon\mathcal S(D)\to N\) be an element of \({\rm Col}_N^G(D,\sigma).\)
Because \(E_u^m(D,\sigma)\) is the free \(A\)-module on the semiarc
set \(\mathcal S(D)\), there exists a unique \(A\)-linear map
\(
\widetilde C\colon E_u^m(D,\sigma)\longrightarrow N
\)
such that \(\widetilde C(x)=C(x)\) for all \(x\in\mathcal S(D).\)
Since \(C\) satisfies the affine coloring conditions at each crossing
\(c_k\), Figure~\ref{fig-qd-rels-0} gives
\(
\widetilde C\bigl(R_1(c_k)\bigr)=\widetilde C\bigl(R_2(c_k)\bigr)=0
\, (k=1,\ldots,\ell).
\)
As these relators generate \(\im\alpha_{D,\sigma}\), it follows that
\(\im\alpha_{D,\sigma}\subseteq\ker\widetilde C.\)
Consequently, the universal property of the quotient provides a unique
\(A\)-linear map
\(
f_C\colon M_u(D,\sigma;A)\longrightarrow N
\)
such that \(f_C\circ\phi_{D,\sigma}=\widetilde C.\)
Then for every semiarc \(x\in\mathcal S(D)\), we have
\(
\Theta_{D,\sigma,N}(f_C)(x)=f_C\bigl(\phi_{D,\sigma}(x)\bigr)=\widetilde C(x)=C(x).
\)
Hence \(\Theta_{D,\sigma,N}(f_C)=C.\)

Conversely, let \(f\in \Hom_A\bigl(M_u(D,\sigma;A),N\bigr)\), then the \(A\)-linear extension of \(C_f\,(=\Theta_{D,\sigma,N}(f)) : \mathcal S(D) \to N\) to \(E_u^m(D,\sigma)\) is
\(f\circ\phi_{D,\sigma}\). Since \(\phi_{D,\sigma}\) is surjective, the
uniqueness of the induced quotient map gives \(f_{C_f}=f.\)
Thus \(f\mapsto C_f\) and \(C\mapsto f_C\) are mutually inverse, proving
that \(\Theta_{D,\sigma,N}\) is a bijection.

\par\medskip
\noindent
\begin{minipage}[t]{0.52\linewidth}
\centering
\(C\longmapsto f_C\)
\[
\begin{tikzcd}[
  column sep=1.5em,
  row sep=2.5em,
  font=\small
]
\mathcal S(D)
  \arrow[r, "\iota"]
  \arrow[d, "C"']
&
E_u^m(D,\sigma)
  \arrow[d, "\phi_{D,\sigma}"]
  \arrow[dl, "\widetilde C" description]
\\
N
&
M_u(D,\sigma;A)
  \arrow[l, dashed, "f_C"']
\end{tikzcd}
\]
\(\Theta_{D,\sigma,N}(f_C)=C.\)
\end{minipage}
\begin{minipage}[t]{0.40\linewidth}
\centering
\(f\longmapsto C_f\)
\[
\begin{tikzcd}[
  column sep=1.5em,
  row sep=2.5em,
  font=\small
]
\mathcal S(D)
  \arrow[r, "\iota"]
  \arrow[d, "C_f"']
&
E_u^m(D,\sigma)
  \arrow[d, "\phi_{D,\sigma}"]
  \arrow[dl, "\widetilde{C_f}" description]
\\
N
&
M_u(D,\sigma;A)
  \arrow[l, "f"']
\end{tikzcd}
\]
\(\widetilde{C_f}=f\circ\phi_{D,\sigma},
  \qquad f_{C_f}=f.\)
\end{minipage}
\par\medskip

Finally, let \(\psi\colon N\to N'\) be an \(A\)-linear map. Such a
map preserves all the affine operations. Indeed, for every \(g\in G\),
\begin{align*}
\psi\bigl(x\mathbin{\underline{\triangleright}_g}y\bigr)
&=\psi\bigl(gx+(u-g)y\bigr)=g\psi(x)+(u-g)\psi(y)
=\psi(x)\mathbin{\underline{\triangleright}_g}\psi(y),
\\
\psi\bigl(x\mathbin{\overline{\triangleright}_g}y\bigr)
&=\psi(ux)=u\psi(x)=\psi(x)\mathbin{\overline{\triangleright}_g}\psi(y).
\end{align*}
Hence applying \(\psi\) to every semiarc color defines the map
\(
\psi_*^{D,\sigma}\colon
{\rm Col}_N^G(D,\sigma)
\longrightarrow
{\rm Col}_{N'}^G(D,\sigma).
\)

For every \(f\in\Hom_A(M_u(D,\sigma;A),N)\) and every semiarc \(x\in\mathcal S(D)\), we have
\begin{align*}
&\bigl(\Theta_{D,\sigma,N'}\circ\psi_{\#}\bigr)(f)(x)
=\Theta_{D,\sigma,N'}(\psi\circ f)(x)=(\psi\circ f)(\phi_{D,\sigma}(x))
=\psi\bigl(f(\phi_{D,\sigma}(x))\bigr),\\
&\bigl(\psi_*^{D,\sigma}\circ \Theta_{D,\sigma,N}\bigr)(f)(x)
=\psi_*^{D,\sigma}(C_f)(x) =\psi\bigl(C_f(x)\bigr)
=\psi\bigl(f(\phi_{D,\sigma}(x))\bigr).
\end{align*}
Hence the two colorings agree on every semiarc of \(D\). Therefore
\(
\Theta_{D,\sigma,N'}\circ\psi_{\#}
=\psi_*^{D,\sigma}\circ\Theta_{D,\sigma,N},
\)
so the displayed diagram commutes. This proves naturality in \(N\).
\end{proof}

The affine coloring set is naturally an $A$-module under pointwise
addition and scalar multiplication. Indeed, writing the semiarc colors
as a column $v\in N^m$, the crossing equations give
\[
 {\rm Col}_N^G(D,\sigma)
 =\ker\!\left(\mathcal A_u(D,\sigma)^{\mathsf T}
               :N^m\longrightarrow N^{2\ell}\right).
\]
Here the matrix acts by its coefficients on $N$; the transpose occurs
because the presentation records relations as columns.
For $a,b\in A$ and $f,g\in\Hom_A(M_u(D,\sigma;A),N)$, evaluation
on a semiarc $x$ gives
$$\Theta_{D,\sigma,N}(af+bg)(x)
=a\Theta_{D,\sigma,N}(f)(x)+b\Theta_{D,\sigma,N}(g)(x).$$
Thus the natural bijection $\Theta_{D,\sigma,N}$ is also an
$A$-module isomorphism.

\begin{definition}[Universal affine coloring]
\label{def:universal-affine-coloring}
Put \(M_D:=M_u(D,\sigma;A),\) for brevity, and let \(\phi_{D,\sigma}\colon E_u^m(D,\sigma)\longrightarrow M_D\) be the canonical quotient map.  The coloring
\[
U_{D,\sigma}
:=
\Theta_{D,\sigma,M_D}
\bigl(\operatorname{id}_{M_D}\bigr)
\in
{\rm Col}_{M_D}^G(D,\sigma)
\]
is called the \emph{universal affine coloring} of
\((D,\sigma)\).  Thus, for every
\(x\in\mathcal S(D)\),
\(U_{D,\sigma}(x)=\operatorname{id}_{M_D}(\phi_{D,\sigma}(x)) =\phi_{D,\sigma}(x)=[x]\).
\end{definition}

By Theorem~\ref{thm:representability}, \(U_{D,\sigma}\) is an affine coloring.  Concretely, this also follows from the fact that
\(\operatorname{Im}(\alpha_{D,\sigma})\) is generated by the defining
linear crossing relations: their images vanish in \(M_D\), and hence
the elements \(U_{D,\sigma}(x)=\phi_{D,\sigma}(x)\) satisfy all the coloring conditions.

\begin{theorem}[Universal property of \((M_D,U_{D,\sigma})\)]
\label{thm:universal-affine-coloring}
For every \(A\)-module \(N\) and every \(A\)-linear map
\(f\colon M_D\to N\), one has
\begin{equation}
\label{eq:theta-via-universal-coloring}
\Theta_{D,\sigma,N}(f)=f\circ U_{D,\sigma}.
\end{equation}
Consequently, \((M_D,U_{D,\sigma})\) has the following universal property: for every \(A\)-module \(N\)
and every affine coloring \(C\in{\rm Col}_N^G(D,\sigma),\)
there exists a unique \(A\)-linear map
\(f_C\colon M_D\longrightarrow N\)
such that
\begin{equation}
\label{eq:universal-coloring-factorization}
C=f_C\circ U_{D,\sigma}.
\end{equation}
Moreover, \(f_C=\Theta_{D,\sigma,N}^{-1}(C).\)
\end{theorem}

\begin{proof}
For each semiarc $x$, the definitions give
\(
\Theta_{D,\sigma,N}(f)(x)
=f(\phi_{D,\sigma}(x))=f(U_{D,\sigma}(x))=(f\circ U_{D,\sigma})(x).
\)
This proves \eqref{eq:theta-via-universal-coloring}. The bijectivity in
Theorem~\ref{thm:representability} supplies the unique map
$f_C=\Theta_{D,\sigma,N}^{-1}(C) \in \Hom_A\bigl(M_D,N\bigr)$, and 
$(f_C\circ U_{D,\sigma})(x)=f_C(U_{D,\sigma}(x))=f_C(\phi_{D,\sigma}(x))=\Theta_{D,\sigma,N}(f_C)(x)=C(x)$ for all $x\in \mathcal S(D)$. This
gives \eqref{eq:universal-coloring-factorization}.
\end{proof}

In what follows, we call $(M_D,U_{D,\sigma})$ the
\emph{universal representing pair}, or simply the \emph{representing pair},
for the assignment $N\mapsto{\rm Col}_N^G(D,\sigma)$.
Here \emph{universal} refers to colorings with values in $A$-modules
equipped with the fixed affine operations \eqref{eq1-mdl-gpring}.
The coloring functor formalized below records both these coloring sets
and their maps induced by $A$-linear homomorphisms of targets;
it does not range over arbitrary $G$-multibiquandle targets.

\bigskip
For two functors $F,G:\mathbf C\to\mathbf D$, recall that a {\it natural transformation $\eta:F\Rightarrow G$} means a family \(\{\eta_X\}_X\) of morphisms $\eta_X:F(X)\to G(X)$ satisfying $G(f)\circ\eta_X=\eta_Y\circ F(f)$ for every $f:X\to Y$.
A natural transformation $\eta:F\Rightarrow G$ is called a {\it natural isomorphism}, denoted by $\eta:F\xRightarrow{\;\cong\;} G$, if each $\eta_X$ is invertible; the inverses then form a natural transformation $\eta^{-1}:G\Rightarrow F$
\cite{MacLane}.

\medskip
Let \(\Omega\text{-}\mathbf{MBQ}\) denote the category of
\(\Omega\)-multibiquandles and their homomorphisms and let
\(\mathbf{Set}\) denote the category of sets.  For a fixed
\(\Omega\)-colored link diagram \((D,\sigma)\), define the
\emph{coloring functor}
\[\mathscr C_{D,\sigma}\colon \Omega\text{-}\mathbf{MBQ} \longrightarrow \mathbf{Set}\]
by
\(
\mathscr C_{D,\sigma}(X)
=
{\rm Col}_X^\Omega(D,\sigma)
\)
on objects and by
\(
\mathscr C_{D,\sigma}(f)
=
f_*^{D,\sigma}
\)
on a homomorphism \(f\colon X\to Y\).
For a generalized Reidemeister move
\(
r\colon(D,\sigma)\longrightarrow(D',\sigma'),
\)
Theorem~\ref{thm-n-mbqd-cinv} gives a natural isomorphism
\[\Phi^r\colon \mathscr C_{D,\sigma} \xRightarrow{\;\cong\;} \mathscr C_{D',\sigma'}.\]
Its component at an \(\Omega\)-multibiquandle \(X\) is the
Reidemeister extension bijection
\[
(\Phi^r)_X=\Phi_X^r\colon
{\rm Col}_X^\Omega(D,\sigma)
\xrightarrow{}
{\rm Col}_X^\Omega(D',\sigma'),
\]
whose inverse is \((\Phi_X^r)^{-1}\).  Naturality means that, for every
homomorphism \(f\colon X\to Y\),
\[
f_*^{D',\sigma'}\circ\Phi_X^r
=
\Phi_Y^r\circ f_*^{D,\sigma}.
\]

\medskip
We now specialize to \(\Omega=G\), a group.  Let \(R\) be a commutative ring
with identity, let \(G\) be an abelian group, put \(A=R[G]\), and fix
\(u\in A^\times\).  Let \(\mathbf{Mod}_A\) denote the category of
\(A\)-modules and their \(A\)-linear maps.  By
Theorem~\ref{thm-multibqdl-mdl}, every \(A\)-module \(N\) becomes a
\(G\)-multibiquandle \((N,\mathbin{\underline{\triangleright}_g},\mathbin{\overline{\triangleright}_g})\), where \(\mathbin{\underline{\triangleright}_g}\) and \(\mathbin{\overline{\triangleright}_g}\) are the affine operations in \eqref{eq1-mdl-gpring}, namely
\(x\mathbin{\underline{\triangleright}_g}y
=gx+(u-g)y\) and \(x\mathbin{\overline{\triangleright}_g}y=ux \,(g\in G).\)
Every \(A\)-linear map preserves these operations.  Hence
they define a functor
\(\mathfrak A_{A,u}\colon \mathbf{Mod}_A \longrightarrow G\text{-}\mathbf{MBQ}\)
which sends an \(A\)-module \(N\) to its associated affine
\(G\)-multibiquandle \((N,\mathbin{\underline{\triangleright}_g},\mathbin{\overline{\triangleright}_g})\), and acts as the identity on the underlying maps. Then, for a fixed \(G\)-colored diagram \((D,\sigma)\), define the \emph{affine coloring functor} to be the composite
\[\mathscr C^{\mathrm{aff}}_{D,\sigma;u} := \mathscr C_{D,\sigma}\circ\mathfrak A_{A,u} \colon \mathbf{Mod}_A \longrightarrow \mathbf{Set}.\]
Thus
\(\mathscr C^{\mathrm{aff}}_{D,\sigma;u}(N)
={\rm Col}_N^G(D,\sigma)\) and 
\(\mathscr C^{\mathrm{aff}}_{D,\sigma;u}(f)
=f_*^{D,\sigma}\) for every \(A\)-linear map \(f\colon N\to N'\), where
\(
f_*^{D,\sigma}\colon
{\rm Col}_N^G(D,\sigma)
\longrightarrow
{\rm Col}_{N'}^G(D,\sigma)\) is the map in (\ref{eq:ind-color-map}) defined by
\(f_*^{D,\sigma}(C)=f\circ C.\)
Here and below, the adjective \emph{affine} refers to the
\(G\)-multibiquandle structure and its coloring theory; the defining
relations and all morphisms in \(\mathbf{Mod}_A\) are \(A\)-linear.

We recall the standard notions of a representable functor and
a universal element (see \cite[Chapter~III, Sections~1--2]{MacLane}).

\begin{definition}[Representable functor and universal element]
\label{def:representable-functor}
A covariant functor
\(\mathscr F\colon \mathbf{Mod}_A\to\mathbf{Set}\)
is said to be \emph{representable} if there exist an \(A\)-module \(M\) and
a natural family of bijections
\(\Theta_N\colon \operatorname{Hom}_A(M,N)\to\mathscr F(N)\).
In this case, \(M\) is called a \emph{representing object} for the functor 
\(\mathscr F\).

An element \(U\in\mathscr F(M)\) is called a \emph{universal element}
if, for every \(A\)-module \(N\) and every \(C\in\mathscr F(N)\),
there exists a unique \(A\)-linear map \(f\colon M\to N\) such that
\(\mathscr F(f)(U)=C\).
For a representation \(\Theta\), the corresponding universal element
is \(U=\Theta_M(\operatorname{id}_M)\): naturality gives
\(\Theta_N(f)=\mathscr F(f)(U)\), and the bijectivity of
\(\Theta_N\) gives the required existence and uniqueness.
\end{definition}

\begin{corollary}[Representability of the affine coloring functor]
\label{cor:affine-coloring-representability}
The family \(\bigl\{\Theta_{D,\sigma,N}\bigr\}_N\) of the maps
\(
\Theta_{D,\sigma,N}\colon
\operatorname{Hom}_A(M_u(D,\sigma;A),N)
\longrightarrow
{\rm Col}_N^G(D,\sigma)\) given by 
\(\Theta_{D,\sigma,N}(f)=f\circ U_{D,\sigma}\)
forms a natural isomorphism
\[
\Theta_{D,\sigma}\colon
\operatorname{Hom}_A(M_u(D,\sigma;A),-)
\xRightarrow{\;\cong\;}
\mathscr C^{\mathrm{aff}}_{D,\sigma;u}.
\]
Thus \(M_u(D,\sigma;A)\) represents the affine coloring functor and
\(U_{D,\sigma}\) is the corresponding universal element.
\end{corollary}

\begin{proof}
Put \(M=M_u(D,\sigma;A)\).
By Theorem~\ref{thm:representability}, the maps
\(\Theta_{D,\sigma,N}\) are bijective and natural in \(N\).
By Theorem~\ref{thm:universal-affine-coloring}, they are given by
\(\Theta_{D,\sigma,N}(f)=f\circ U_{D,\sigma}\).
Hence they form the asserted natural isomorphism.

Taking \(N=M\) and \(f=\operatorname{id}_M\), we obtain
\(\Theta_{D,\sigma,M}(\operatorname{id}_M)=U_{D,\sigma}\).
For every \(A\)-module \(N\) and every coloring
\(C\in{\rm Col}_N^G(D,\sigma)\), the bijectivity of
\(\Theta_{D,\sigma,N}\) gives a unique \(A\)-linear map
\(f_C\colon M\to N\) such that
\(\Theta_{D,\sigma,N}(f_C)=C\).
By the formula above and the definition of the coloring functor,
this equality is precisely
\(\mathscr C^{\mathrm{aff}}_{D,\sigma;u}(f_C)(U_{D,\sigma})
=f_C\circ U_{D,\sigma}=C\).
Thus, by Definition~\ref{def:representable-functor},
\(M\) represents \(\mathscr C^{\mathrm{aff}}_{D,\sigma;u}\),
and \(U_{D,\sigma}\) is the corresponding universal element.
\end{proof}

The preceding results identify both a representing module and its
universal affine coloring. The following remark records the
uniqueness of this representing pair.

\begin{remark}
\label{rem:universal-affine-coloring}
Recall that an object \(I\) of a category \(\mathcal C\) is
\emph{initial} if, for every object \(X\) of \(\mathcal C\),
there exists exactly one morphism \(I\to X\); see \cite{MacLane}.
Any two initial objects \(I\) and \(J\) are connected by a unique
isomorphism. Indeed, the unique morphisms
\(a\colon I\to J\) and \(b\colon J\to I\) satisfy
\(b\circ a=\operatorname{id}_I\) and
\(a\circ b=\operatorname{id}_J\), since each initial object has
exactly one endomorphism.

For the fixed colored diagram \((D,\sigma)\), consider the category
whose objects are pairs \((N,C)\), where \(N\) is an \(A\)-module
and \(C\in{\rm Col}_N^G(D,\sigma)\), and whose morphisms
\(\psi\colon(N,C)\to(N',C')\) are \(A\)-linear maps
\(\psi\colon N\to N'\) satisfying
\(C'=\psi\circ C\).
Here the coefficient ring \(A=R[G]\), the parameter \(u\), and
the affine operations remain fixed.

By Theorem~\ref{thm:universal-affine-coloring},
\((M_D,U_{D,\sigma})\) is an initial object of this category:
every coloring \(C\in{\rm Col}_N^G(D,\sigma)\) determines a unique
morphism \(f_C\colon(M_D,U_{D,\sigma})\to(N,C)\).
Thus the initial-object formulation expresses the same universal
property as the natural isomorphism
\(\Theta_{D,\sigma}\colon \Hom_A(M_D,-)\xRightarrow{\;\cong\;} \mathscr C^{\mathrm{aff}}_{D,\sigma;u}\)
in Corollary~\ref{cor:affine-coloring-representability}.

Consequently, if \((M',U')\) is another representing pair for
this same affine coloring functor, there exists a unique
\(A\)-module isomorphism
\(\eta\colon M_D\xrightarrow{\;\cong\;}M'\) such that \(\eta\circ U_{D,\sigma}=U'\).
This uniqueness concerns the representing module together with
its universal coloring. The Yoneda criterion below describes
how natural transformations between representable functors
induce maps between their representing modules.
\end{remark}

To pass from natural transformations between affine coloring functors
to homomorphisms between their representing modules, we use the
following standard covariant form of the Yoneda lemma.  More precisely,
a natural transformation
\(
\eta\colon
\operatorname{Hom}_A(M,-)
\Longrightarrow
\operatorname{Hom}_A(M',-)
\)
between the covariant representable functors associated with
\(A\)-modules \(M\) and \(M'\) induces an \(A\)-linear homomorphism in
the reversed direction \(M'\longrightarrow M.\)
We include the proof both to record the isomorphism criterion
needed below and to make this reversal of direction explicit; see
\cite[Chapter~III, Section~2]{MacLane}.
Here, \( A=R[G] \) denotes the fixed group ring introduced above, and \(\mathbf{Mod}_A\) denotes the category of unital left \(A\)-modules and \(A\)-linear maps. Although the lemma is stated for this fixed coefficient ring, it is noted that its statement and proof remain valid, without change, for modules over any unital ring. 

\begin{lemma}[Yoneda criterion in \(\mathbf{Mod}_A\)] 
\label{lem:yoneda-criterion} 
Let \(M\) and \(M'\) be objects of \(\mathbf{Mod}_A\).  For every natural
transformation \(\eta\colon\Hom_A(M,-)\Longrightarrow\Hom_A(M',-),\)
there is a unique \(A\)-linear map \(\mu_\eta\colon M'\longrightarrow M\)
such that \(\eta_N(f)=f\circ\mu_\eta\) for every object \(N\) of \(\mathbf{Mod}_A\) and every \(f\in\Hom_A(M,N)\). This map is given by \(\mu_\eta=\eta_M(\operatorname{id}_M).\)
Moreover, \(\eta\) is a natural isomorphism if and only if
\(\mu_\eta\) is an \(A\)-module isomorphism.
\end{lemma}

\begin{proof}
Define \(\mu_\eta:=\eta_M(\operatorname{id}_M)\in\Hom_A(M',M).\)
Let \(N\) be an \(A\)-module and let \(f\in\Hom_A(M,N).\)
We may regard \(f\) as a morphism \(f\colon M\to N\) in
\(\mathbf{Mod}_A\).  Naturality of \(\eta\) with respect to this
morphism means that the following diagram commutes:
\[
\begin{tikzcd}[row sep=1.8em, column sep=3.2em]
\Hom_A(M,M)
  \arrow[r, "{\eta_M}"]
  \arrow[d, "{\Hom_A(M,f)}"']
& \Hom_A(M',M)
  \arrow[d, "{\Hom_A(M',f)}"]
\\
\Hom_A(M,N)
  \arrow[r, "{\eta_N}"]
& \Hom_A(M',N).
\end{tikzcd}
\]
Equivalently, since both vertical maps are given by postcomposition
with \(f\), this diagram has the form
\[
\begin{tikzcd}[row sep=1.8em, column sep=3.2em]
\Hom_A(M,M)
  \arrow[r, "{\eta_M}"]
  \arrow[d, "{f\circ(-)}"']
& \Hom_A(M',M)
  \arrow[d, "{f\circ(-)}"]
\\
\Hom_A(M,N)
  \arrow[r, "{\eta_N}"]
& \Hom_A(M',N).
\end{tikzcd}
\]
Applying this commutative diagram to
\(\operatorname{id}_M\in\Hom_A(M,M)\), we have
\(
\eta_N\bigl(f\circ\operatorname{id}_M\bigr)
=f\circ\eta_M(\operatorname{id}_M).
\)
This gives that for every \(A\)-module \(N\) and every \(f\in\Hom_A(M,N)\),
\begin{equation}
\label{eq:yoneda-component-formula}
\eta_N(f)=f\circ\mu_\eta.
\end{equation}
This formula also proves uniqueness.  Indeed, suppose that
\(\mu\colon M'\to M\) is another \(A\)-linear map satisfying
\(\eta_N(f)=f\circ\mu\) for every \(N\) and every \(f\colon M\to N\).  Taking \(N=M\) and \(f=\operatorname{id}_M\), we obtain
\(\mu=\operatorname{id}_M\circ\mu=\eta_M(\operatorname{id}_M)=\mu_\eta.\)
Hence \(\mu_\eta\) is uniquely determined by \(\eta\).

\medskip
Now suppose that \(\eta\) is a natural isomorphism.  Its inverse is
a natural transformation
\(
\eta^{-1}\colon
\Hom_A(M',-)
\Longrightarrow
\Hom_A(M,-).
\)
Applying the construction above to \(\eta^{-1}\), define
\(\nu := (\eta^{-1})_{M'} (\operatorname{id}_{M'}) \in \Hom_A(M,M')\).
Naturality of \(\eta^{-1}\) with respect to an \(A\)-linear map
\(g\colon M'\to N\) gives the commutative diagram:
\[
\begin{tikzcd}[row sep=1.8em, column sep=3.2em]
\Hom_A(M',M')
  \arrow[r, "{(\eta^{-1})_{M'}}"]
  \arrow[d, "{g\circ(-)}"']
& \Hom_A(M,M')
  \arrow[d, "{g\circ(-)}"]
\\
\Hom_A(M',N)
  \arrow[r, "{(\eta^{-1})_N}"]
& \Hom_A(M,N).
\end{tikzcd}
\]
Applying this diagram to \(\operatorname{id}_{M'}\), we have
\(
(\eta^{-1})_N\bigl(g\circ\operatorname{id}_{M'}\bigr)
=g\circ(\eta^{-1})_{M'}(\operatorname{id}_{M'}).
\) This gives
\begin{equation}
\label{eq:yoneda-inverse-component-formula}
(\eta^{-1})_N(g)=g\circ\nu
\end{equation}
for every \(g\colon M'\to N\). We now use the identities of natural transformations \(\eta^{-1}\circ\eta=\operatorname{id}_{\Hom_A(M,-)}\)
and \(\eta\circ\eta^{-1}=\operatorname{id}_{\Hom_A(M',-)}.\)
Evaluating the first identity at the object \(M\), and then at
\(\operatorname{id}_M\in\Hom_A(M,M)\), we obtain
\(
\operatorname{id}_M=
(\eta^{-1})_M
\bigl(\eta_M(\operatorname{id}_M)\bigr)=
(\eta^{-1})_M(\mu_\eta)=
\mu_\eta\circ\nu,
\)
where the last equality follows from \eqref{eq:yoneda-inverse-component-formula} with \(N=M\) and \(g=\mu_\eta\).
Similarly, evaluating the second identity at the object \(M'\), and then at
\(\operatorname{id}_{M'}\in\Hom_A(M',M')\), we obtain
\(
\operatorname{id}_{M'}=\eta_{M'}\bigl((\eta^{-1})_{M'}
(\operatorname{id}_{M'})\bigr)=\eta_{M'}(\nu)=\nu\circ\mu_\eta,
\)
where the last equality follows from
\eqref{eq:yoneda-component-formula} with \(N=M'\) and \(f=\nu\).
Thus
\(\mu_\eta\circ\nu=\operatorname{id}_M\,\text{and}\,
\nu\circ\mu_\eta=\operatorname{id}_{M'}.\)
Therefore \(\mu_\eta\) is an \(A\)-module isomorphism with inverse
\(\nu\).

Conversely, suppose that \(\mu_\eta\colon M'\to M\) is an \(A\)-module isomorphism.  For every \(A\)-module \(N\), define
\(\theta_N\colon\Hom_A(M',N)\longrightarrow\Hom_A(M,N)\) by
\(\theta_N(g):=g\circ\mu_\eta^{-1}.\) Then the family \(\{\theta_N\}_N\) is natural.  Indeed, if \(h\colon N\to P\) is \(A\)-linear and
\(g\colon M'\to N\), then
\(\theta_P(h\circ g)=(h\circ g)\circ\mu_\eta^{-1}=
h\circ(g\circ\mu_\eta^{-1})=h\circ\theta_N(g).\)
Equivalently, the following diagram commutes:
\[
\begin{tikzcd}[row sep=1.8em, column sep=3.6em]
\Hom_A(M',N)
  \arrow[r, "{\theta_N}"]
  \arrow[d, "{h\circ(-)}"']
& \Hom_A(M,N)
  \arrow[d, "{h\circ(-)}"]
\\
\Hom_A(M',P)
  \arrow[r, "{\theta_P}"]
& \Hom_A(M,P).
\end{tikzcd}
\]
Finally, using \(\eta_N(f)=f\circ\mu_\eta\) in (\ref{eq:yoneda-component-formula}), we have, for every \(f\colon M\to N\),
\(\theta_N\bigl(\eta_N(f)\bigr)=(f\circ\mu_\eta)\circ\mu_\eta^{-1}=f,\)
and, for every \(g\colon M'\to N\), \(\eta_N\bigl(\theta_N(g)\bigr)
=(g\circ\mu_\eta^{-1})\circ\mu_\eta=g.\) Hence \(\theta_N=\eta_N^{-1}\)
for every \(N\), and the inverses form a natural transformation.
Therefore \(\eta\) is a natural isomorphism.
\end{proof}

We now apply Lemma~\ref{lem:yoneda-criterion} directly to the
natural coloring bijection induced by a generalized Reidemeister move.
Let \(r\colon(D,\sigma)\longrightarrow(D',\sigma')\)
be a generalized Reidemeister move.  For every \(A\)-module \(N\),
define
\(\Phi_N^{r,\mathrm{aff}} := \Phi_{\mathfrak A_{A,u}(N)}^r \colon {\rm Col}_N^G(D,\sigma) \longrightarrow {\rm Col}_N^G(D',\sigma')\).
Thus \(\Phi_N^{r,\mathrm{aff}}(C)\) is the unique \(N\)-valued affine
coloring extending the boundary colors across the move that agrees with \(C\) outside the disk supporting the move. Every \(A\)-linear map \(\psi\colon N\to N'\) is a homomorphism from the associated affine \(G\)-multibiquandle \(\mathfrak A_{A,u}(N)\) to \(\mathfrak A_{A,u}(N')\).  Hence the naturality of the Reidemeister extension bijections (see Theorem~\ref{thm-n-mbqd-cinv}) gives
\begin{equation}
\label{eq:affine-reidemeister-extension-naturality}
\psi_*^{D',\sigma'}\circ\Phi_N^{r,\mathrm{aff}}
=
\Phi_{N'}^{r,\mathrm{aff}}\circ\psi_*^{D,\sigma}.
\end{equation}
Equivalently, the following diagram commutes:
\[
\begin{tikzcd}[row sep=1.8em, column sep=3.2em]
{\rm Col}_N^G(D,\sigma)
  \arrow[r, "{\Phi_N^{r,\mathrm{aff}}}"]
  \arrow[d, "{\psi_*^{D,\sigma}}"']
& {\rm Col}_N^G(D',\sigma')
  \arrow[d, "{\psi_*^{D',\sigma'}}"]
\\
{\rm Col}_{N'}^G(D,\sigma)
  \arrow[r, "{\Phi_{N'}^{r,\mathrm{aff}}}"]
& {\rm Col}_{N'}^G(D',\sigma').
\end{tikzcd}
\]
Since each \(\Phi_N^{r,\mathrm{aff}}\) is a bijection, the family
\(\Phi^{r,\mathrm{aff}}:=\{\Phi_N^{r,\mathrm{aff}}\}_N\)
is a natural isomorphism
\begin{equation}
\label{eq:affine-reidemeister-natural-isomorphism}
\Phi^{r,\mathrm{aff}}\colon
\mathscr C^{\mathrm{aff}}_{D,\sigma;u}
\xRightarrow{\;\cong\;}
\mathscr C^{\mathrm{aff}}_{D',\sigma';u}.
\end{equation}

For the generalized Reidemeister move
\(r\colon(D,\sigma)\longrightarrow(D',\sigma')\) above,
the natural isomorphism \(\Phi^{r,\mathrm{aff}}\) transports
the universal coloring \(U_{D,\sigma}\) to
\(\Phi_{M_D}^{r,\mathrm{aff}}(U_{D,\sigma}) \in{\rm Col}_{M_D}^G(D',\sigma')\).
This transported coloring is a universal element for
\(\mathscr C^{\mathrm{aff}}_{D',\sigma';u}\).
Indeed, for every \(A\)-module \(N\) and every
\(A\)-linear map \(f\colon M_D\to N\), naturality gives
\(\Phi_N^{r,\mathrm{aff}} \bigl(\Theta_{D,\sigma,N}(f)\bigr) = \Phi_N^{r,\mathrm{aff}}(f\circ U_{D,\sigma}) = f\circ\Phi_{M_D}^{r,\mathrm{aff}}(U_{D,\sigma})\).
Since both \(\Theta_{D,\sigma,N}\) and
\(\Phi_N^{r,\mathrm{aff}}\) are bijective, every coloring
in \({\rm Col}_N^G(D',\sigma')\) factors uniquely through
the transported coloring.

Hence
\(\bigl(M_D,\Phi_{M_D}^{r,\mathrm{aff}}(U_{D,\sigma})\bigr)\) and \((M_{D'},U_{D',\sigma'})\)
are two initial objects of the coloring category associated with
\((D',\sigma')\) in Remark~\ref{rem:universal-affine-coloring}.
Their unique compatible isomorphism
\(\mu_r\colon M_{D'}\xrightarrow{\;\cong\;}M_D\)
is characterized by
\(\mu_r\circ U_{D',\sigma'} = \Phi_{M_D}^{r,\mathrm{aff}}(U_{D,\sigma})\).
The following theorem constructs this same isomorphism by the
Yoneda criterion and establishes its compatibility with
Reidemeister extension for every module target.

\begin{theorem}[Reidemeister invariance via representability]
\label{thm:reidemeister-representing-module}
Let \(R\) be a commutative ring with identity, let \(G\) be an abelian
group, put \(A=R[G]\), and fix \(u\in A^\times\). Let
\(r\colon(D,\sigma)\longrightarrow(D',\sigma')\) be a generalized
Reidemeister move, and write \(M_D:=M_u(D,\sigma;A)\) and
\(M_{D'}:=M_u(D',\sigma';A)\). Then there exists a unique
\(A\)-module isomorphism
\(\mu_r\colon M_{D'}\xrightarrow{\;\cong\;}M_D\)
compatible with the Reidemeister extension natural isomorphism, in the
sense that
\begin{equation}
\label{eq:mu-r-compatibility}
\Theta_{D',\sigma',N}(f\circ\mu_r)
=\Phi_N^{r,\mathrm{aff}}\bigl(\Theta_{D,\sigma,N}(f)\bigr)
\end{equation}
for every \(A\)-module \(N\) and every \(f\in\Hom_A(M_D,N)\).
Equivalently, for every \(N\), the following diagram commutes:
\[
\begin{tikzcd}[row sep=1.8em, column sep=3.2em]
\Hom_A(M_D,N)
  \arrow[r, "{\Theta_{D,\sigma,N}}"]
  \arrow[d, "{(-)\circ\mu_r}"']
& {\rm Col}_N^G(D,\sigma)
  \arrow[d, "{\Phi_N^{r,\mathrm{aff}}}"]
\\
\Hom_A(M_{D'},N)
  \arrow[r, "{\Theta_{D',\sigma',N}}"]
& {\rm Col}_N^G(D',\sigma').
\end{tikzcd}
\]
The map \(\mu_r\) is determined explicitly by
\begin{equation}
\label{eq:mu-r-yoneda-formula}
\mu_r
=\Theta_{D',\sigma',M_D}^{-1}
\left(\Phi_{M_D}^{r,\mathrm{aff}}(U_{D,\sigma})\right),
\end{equation}
or, equivalently, by the universal-coloring identity
\begin{equation}
\label{eq:mu-r-universal-coloring}
\mu_r\circ U_{D',\sigma'}
=\Phi_{M_D}^{r,\mathrm{aff}}(U_{D,\sigma}).
\end{equation}
If \(r^{-1}\colon(D',\sigma')\to(D,\sigma)\) is the inverse move,
then \(\mu_{r^{-1}}=\mu_r^{-1}\). Consequently,
\(M_u(D,\sigma;A)\cong M_u(D',\sigma';A)\) and
\(\mathcal A_u(D,\sigma)\sim\mathcal A_u(D',\sigma')\).
\end{theorem}

\begin{proof}
For each \(A\)-module \(N\), define
\(\eta_N^r:=\Theta_{D',\sigma',N}^{-1}\circ
\Phi_N^{r,\mathrm{aff}}\circ\Theta_{D,\sigma,N}\).
Then \(\eta_N^r\colon\Hom_A(M_D,N)\to\Hom_A(M_{D'},N)\)
is a bijection because each of its three factors is a bijection.
We first verify that the family \(\eta^r:=\{\eta_N^r\}_N\)
is natural in \(N\).

Let \(\psi\colon N\to N'\) be \(A\)-linear and let
\(f\in\Hom_A(M_D,N)\). Naturality of \(\Theta_{D,\sigma}\)
in Theorem~\ref{thm:representability} gives
\begin{equation}
\label{eq:theta-D-naturality-in-reidemeister-proof}
\Theta_{D,\sigma,N'}(\psi\circ f)
=\psi_*^{D,\sigma}\bigl(\Theta_{D,\sigma,N}(f)\bigr).
\end{equation}
Naturality of \(\Theta_{D',\sigma'}\), together with the bijectivity
of its components, gives
\begin{equation}
\label{eq:inverse-theta-Dprime-naturality}
\begin{aligned}
\Theta_{D',\sigma',N'}^{-1}
\bigl(\psi_*^{D',\sigma'}(C')\bigr)
&=\psi_\#\bigl(\Theta_{D',\sigma',N}^{-1}(C')\bigr)
=\psi\circ\Theta_{D',\sigma',N}^{-1}(C')
\end{aligned}
\end{equation}
for every \(C'\in{\rm Col}_N^G(D',\sigma')\). Using
\eqref{eq:theta-D-naturality-in-reidemeister-proof},
\eqref{eq:affine-reidemeister-extension-naturality}, and
\eqref{eq:inverse-theta-Dprime-naturality}, we obtain
\begin{align*}
\eta_{N'}^r(\psi\circ f)
&=\Theta_{D',\sigma',N'}^{-1}
  \left(\Phi_{N'}^{r,\mathrm{aff}}
  \bigl(\Theta_{D,\sigma,N'}(\psi\circ f)\bigr)\right)
=\Theta_{D',\sigma',N'}^{-1}
  \left(\Phi_{N'}^{r,\mathrm{aff}}
  \bigl(\psi_*^{D,\sigma}(\Theta_{D,\sigma,N}(f))\bigr)\right)\\
&=\Theta_{D',\sigma',N'}^{-1}
  \left(\psi_*^{D',\sigma'}
  \bigl(\Phi_N^{r,\mathrm{aff}}(\Theta_{D,\sigma,N}(f))\bigr)\right)
=\psi\circ\Theta_{D',\sigma',N}^{-1}
  \left(\Phi_N^{r,\mathrm{aff}}
  \bigl(\Theta_{D,\sigma,N}(f)\bigr)\right)\\
&=\psi\circ\eta_N^r(f).
\end{align*}
Thus the following diagram commutes:
\[
\begin{tikzcd}[row sep=1.8em, column sep=3.2em]
\Hom_A(M_D,N)
  \arrow[r, "{\eta_N^r}"]
  \arrow[d, "{\psi\circ(-)}"']
& \Hom_A(M_{D'},N)
  \arrow[d, "{\psi\circ(-)}"]
\\
\Hom_A(M_D,N')
  \arrow[r, "{\eta_{N'}^r}"]
& \Hom_A(M_{D'},N').
\end{tikzcd}
\]
Hence \(\eta^r\colon\Hom_A(M_D,-)\xRightarrow{\;\cong\;}
\Hom_A(M_{D'},-)\) is a natural isomorphism.
By Lemma~\ref{lem:yoneda-criterion}, there exists a unique
\(A\)-module isomorphism \(\mu_r\colon M_{D'}\to M_D\) such that
\begin{equation}
\label{eq:eta-r-induced-by-mu-r}
\eta_N^r(f)=f\circ\mu_r
\end{equation}
for every \(A\)-module \(N\) and every \(f\colon M_D\to N\).
In particular, \(M_u(D,\sigma;A)\cong M_u(D',\sigma';A)\).

The evaluation formula in Lemma~\ref{lem:yoneda-criterion}, the
definition of \(\eta_{M_D}^r\), and
Definition~\ref{def:universal-affine-coloring} give
\begin{align*}
\mu_r
&=\eta_{M_D}^r(\operatorname{id}_{M_D})
=\Theta_{D',\sigma',M_D}^{-1}
  \left(\Phi_{M_D}^{r,\mathrm{aff}}
  \bigl(\Theta_{D,\sigma,M_D}(\operatorname{id}_{M_D})\bigr)\right)
=\Theta_{D',\sigma',M_D}^{-1}
  \left(\Phi_{M_D}^{r,\mathrm{aff}}(U_{D,\sigma})\right),
\end{align*}
which proves \eqref{eq:mu-r-yoneda-formula}.
Substituting the definition of \(\eta_N^r\) into
\eqref{eq:eta-r-induced-by-mu-r} gives
\(\Theta_{D',\sigma',N}^{-1}
(\Phi_N^{r,\mathrm{aff}}(\Theta_{D,\sigma,N}(f)))
=f\circ\mu_r\).
Applying \(\Theta_{D',\sigma',N}\) to this equality proves
\eqref{eq:mu-r-compatibility}.
Since \(\Theta_{D,\sigma,N}(f)=f\circ U_{D,\sigma}\) and
\(\Theta_{D',\sigma',N}(f\circ\mu_r)
=(f\circ\mu_r)\circ U_{D',\sigma'}\), the same equality becomes
\begin{equation}
\label{eq:mu-r-factorization-all-N}
\Phi_N^{r,\mathrm{aff}}\bigl(f\circ U_{D,\sigma}\bigr)
=(f\circ\mu_r)\circ U_{D',\sigma'}.
\end{equation}
Taking \(N=M_D\) and \(f=\operatorname{id}_{M_D}\) yields
\(\Phi_{M_D}^{r,\mathrm{aff}}(U_{D,\sigma})
=\mu_r\circ U_{D',\sigma'}\), proving
\eqref{eq:mu-r-universal-coloring}.

Conversely, suppose that an \(A\)-linear map
\(\mu\colon M_{D'}\to M_D\) satisfies
\(\mu\circ U_{D',\sigma'}
=\Phi_{M_D}^{r,\mathrm{aff}}(U_{D,\sigma})\).
For every \(A\)-module \(N\) and every \(f\colon M_D\to N\),
naturality of \(\Phi^{r,\mathrm{aff}}\) with respect to \(f\) gives
\begin{align*}
\Theta_{D',\sigma',N}(f\circ\mu)
&=(f\circ\mu)\circ U_{D',\sigma'}
 =f\circ\Phi_{M_D}^{r,\mathrm{aff}}(U_{D,\sigma})\\
&=\Phi_N^{r,\mathrm{aff}}\bigl(f\circ U_{D,\sigma}\bigr)
 =\Phi_N^{r,\mathrm{aff}}\bigl(\Theta_{D,\sigma,N}(f)\bigr).
\end{align*}
Thus \(\mu\) satisfies \eqref{eq:mu-r-compatibility} with
\(\mu\) in place of \(\mu_r\). This condition is equivalent to
\(\eta_N^r(f)=f\circ\mu\) for every \(N\) and \(f\).
The uniqueness assertion in Lemma~\ref{lem:yoneda-criterion}
therefore gives \(\mu=\mu_r\).
Consequently, \eqref{eq:mu-r-universal-coloring} and
\eqref{eq:mu-r-compatibility} are equivalent characterizations of
\(\mu_r\), even among all \(A\)-linear maps \(M_{D'}\to M_D\).

It remains to identify the map induced by the inverse move.
By uniqueness of local extension,
\(\Phi_N^{r^{-1},\mathrm{aff}}
=(\Phi_N^{r,\mathrm{aff}})^{-1}\) for every \(N\).
The definition of \(\eta^r\) therefore gives
\(\eta^{r^{-1}}=(\eta^r)^{-1}\).
Since \(\eta_N^r(f)=f\circ\mu_r\), its inverse sends
\(g\in\Hom_A(M_{D'},N)\) to
\((\eta_N^r)^{-1}(g)=g\circ\mu_r^{-1}\).
On the other hand, the map associated with \(r^{-1}\) satisfies
\(\eta_N^{r^{-1}}(g)=g\circ\mu_{r^{-1}}\).
Uniqueness in Lemma~\ref{lem:yoneda-criterion} now implies
\(\mu_{r^{-1}}=\mu_r^{-1}\).

Finally, \(\mathcal A_u(D,\sigma)\) and
\(\mathcal A_u(D',\sigma')\) present the isomorphic modules
\(M_D\) and \(M_{D'}\), respectively.
Theorem~\ref{thm-mod-pres} therefore gives
\(\mathcal A_u(D,\sigma)\sim\mathcal A_u(D',\sigma')\).
\end{proof}

Appendix~\ref{app:matrix-reductions} gives a direct reduction of the
presentation matrices for Reidemeister moves, establishing their
presentation equivalence and hence an isomorphism of their cokernels.
Theorem~\ref{thm:reidemeister-representing-module} also specifies how
the module isomorphism acts on affine colorings for every target.
The Reidemeister extension theorem supplies the natural isomorphism
\(\Phi^{r,\mathrm{aff}}\colon
\mathscr C^{\mathrm{aff}}_{D,\sigma;u}
\xRightarrow{\;\cong\;}
\mathscr C^{\mathrm{aff}}_{D',\sigma';u}\).
Representability transports it to the corresponding Hom functors,
and the Yoneda criterion determines the unique compatible
\(A\)-module isomorphism \(\mu_r\colon M_{D'}\to M_D\).
Its compatibility is exactly \eqref{eq:mu-r-compatibility}, for
every \(A\)-module \(N\) and every \(f\colon M_D\to N\).
A bijection of coloring sets for one fixed target alone would not
determine this module map.
By \eqref{eq:mu-r-compatibility} and the $A$-linearity of the
Hom--coloring correspondence, the affine Reidemeister extension
$\Phi_N^{r,\mathrm{aff}}$ is itself $A$-linear for every target $N$.

The invariance proof via representability does not require choosing
bases or carrying out a matrix reduction. Nevertheless, the universal-coloring identity
\eqref{eq:mu-r-universal-coloring} gives a concrete description of
\(\mu_r\) on semiarc generators for the specified move.
For \(x'\in\mathcal S(D')\), write
\(\overline{x'}=\phi_{D',\sigma'}(x')\in M_{D'}\).
Likewise, for \(x\in\mathcal S(D)\), write
\(\overline x=\phi_{D,\sigma}(x)\in M_D\).
Since \(U_{D',\sigma'}(x')=\overline{x'}\), one has
\begin{equation}
\label{eq:mu-r-on-semiarc-generators}
\mu_r(\overline{x'})
=\bigl(\Phi_{M_D}^{r,\mathrm{aff}}(U_{D,\sigma})\bigr)(x').
\end{equation}
Thus the image is obtained by applying the local Reidemeister
extension rule to the universal \(M_D\)-valued coloring.
If a portion of \(x'\) outside the move disk corresponds to a portion
of a semiarc \(x\) of \(D\), then
\(\mu_r(\overline{x'})=\overline{x}\).
The images of the remaining semiarcs are the unique affine
expressions determined by the boundary colors.

The well-definedness of this description can also be checked directly
from the presentation. Put
\(C_r:=\Phi_{M_D}^{r,\mathrm{aff}}(U_{D,\sigma})
\in{\rm Col}_{M_D}^G(D',\sigma')\).
The assignment \(x'\mapsto C_r(x')\) extends uniquely to an
\(A\)-linear map
\(\widetilde\mu_r\colon E_u^{m'}(D',\sigma')\to M_D\), where
\(m'=|\mathcal S(D')|\).
Since \(C_r\) satisfies every defining crossing relation of \(D'\),
one has \(\widetilde\mu_r\circ\alpha_{D',\sigma'}=0\), or
equivalently,
\(\operatorname{Im}(\alpha_{D',\sigma'})
\subseteq\ker\widetilde\mu_r\).
The map therefore descends uniquely through
\(M_{D'}=E_u^{m'}(D',\sigma')/
\operatorname{Im}(\alpha_{D',\sigma'})\) to an \(A\)-linear map
\(\widehat\mu_r\colon M_{D'}\to M_D\) satisfying
\(\widehat\mu_r\circ U_{D',\sigma'}=C_r\).
It follows that
\(\Theta_{D',\sigma',M_D}(\widehat\mu_r)
=C_r=\Theta_{D',\sigma',M_D}(\mu_r)\).
Injectivity of \(\Theta_{D',\sigma',M_D}\) gives
\(\widehat\mu_r=\mu_r\).
Thus the quotient construction realizes the same map as the
Yoneda criterion.

The local formulas make
\eqref{eq:mu-r-on-semiarc-generators} explicit.
For a component labeled by \(g\), the affine diagonal map is
\(\pi_g(x)=x\mathbin{\underline{\triangleright}_g}x
=x\mathbin{\overline{\triangleright}_g}x=ux\), with inverse
\(\pi_g^{-1}(x)=u^{-1}x\).
For insertion of either R1 kink displayed in
Figure~\ref{fig-dibiqdl-rm}, the convention used in the proof of
Theorem~\ref{thm-n-mbqd-cinv} gives exterior color
\(y=\pi_g(x)\), where \(x\) is the loop color.
Consequently, if \(x'\) is the added loop semiarc and \(y\) is the
corresponding exterior semiarc before insertion, then
\(\mu_r(\overline{x'})=u^{-1}\overline y\).
The exterior semiarcs retain their boundary colors, and the reverse
move induces \(\mu_{r^{-1}}=\mu_r^{-1}\).
Other oriented configurations are read with the corresponding
local ordering and diagonal-map convention.

For an R2 move, in the ordered switch coordinates used for the
local extension, the mixed switch is
\(H_{g,h}(x,y)=(uy,hx+(u-h)y)\).
Its inverse is
\(H_{g,h}^{-1}(a,b)
=(h^{-1}b+(u^{-1}-h^{-1})a,u^{-1}a)\).
The local values of \(\mu_r\) are obtained by applying these
formulas to the universal boundary colors, using the prescribed
order of the strands.
For an R3 move, the local values are obtained by composing the
affine operations along the two sides of the move.
The mixed exchange laws verified in
Theorem~\ref{thm-multibqdl-mdl} show that the two compositions give
the same boundary values for each choice of the three colors used
in the proof of Theorem~\ref{thm-n-mbqd-cinv}.
The quotient calculation above verifies that the assignment descends
to an \(A\)-linear map \(M_{D'}\to M_D\), and these local formulas
describe its values explicitly.
They therefore provide a presentation-level realization of the
unique compatible isomorphism obtained from representability.

The compatible maps also compose along chosen sequences of moves.
For a colored planar isotopy \(r\), let
\(\Phi_N^{r,\mathrm{aff}}\) transport colors along corresponding
semiarcs, and define \(\mu_r\) by
\(\overline{x'}\mapsto\overline x\).
Corresponding crossings have identical relators after relabeling,
so this is well defined; the inverse correspondence gives its inverse.
These maps satisfy \eqref{eq:mu-r-compatibility}, and the
presentation matrices differ only by row and column permutations.
We may therefore include planar isotopies among the steps below.

\begin{proposition}[Compatibility along chosen move sequences]
\label{prop:move-sequence-compatibility}
Fix \(R,G,A,u\) as above, and let
\(\gamma=(r_1,\ldots,r_m)\) be a chosen sequence of colored
Reidemeister moves and planar isotopies, where
\(r_j\colon(D_{j-1},\sigma_{j-1})\to(D_j,\sigma_j)\).
Put \(M_j:=M_u(D_j,\sigma_j;A)\), and define
\[
\Phi_N^{\gamma,\mathrm{aff}}
:=\Phi_N^{r_m,\mathrm{aff}}\circ\cdots\circ
\Phi_N^{r_1,\mathrm{aff}},
\qquad
\mu_\gamma:=\mu_{r_1}\circ\cdots\circ\mu_{r_m}
\colon M_m\longrightarrow M_0.
\]
For the empty sequence, both composites are understood to be
identity maps.
Then \(\mu_\gamma\) is the unique \(A\)-linear map satisfying
\begin{equation}
\label{eq:move-sequence-compatibility}
\Theta_{D_m,\sigma_m,N}(f\circ\mu_\gamma)
=\Phi_N^{\gamma,\mathrm{aff}}
\bigl(\Theta_{D_0,\sigma_0,N}(f)\bigr)
\end{equation}
for every \(A\)-module \(N\) and every \(f\colon M_0\to N\).
It is an isomorphism and satisfies
\begin{equation}
\label{eq:move-sequence-universal-coloring}
\mu_\gamma\circ U_{D_m,\sigma_m}
=\Phi_{M_0}^{\gamma,\mathrm{aff}}(U_{D_0,\sigma_0}).
\end{equation}
If \(\delta\) is a chosen sequence starting at \((D_m,\sigma_m)\),
write \(\delta*\gamma\) for the sequence obtained by following
\(\gamma\) and then \(\delta\). Then
\begin{equation}
\label{eq:move-sequence-composition}
\mu_{\delta*\gamma}=\mu_\gamma\circ\mu_\delta,
\qquad
\mu_{\varnothing}=\operatorname{id},
\qquad
\mu_{\gamma^{-1}}=\mu_\gamma^{-1},
\end{equation}
where \(\varnothing\) is the empty sequence at a diagram and
\(\gamma^{-1}\) is the reversed sequence of inverse steps.
\end{proposition}

\begin{proof}
Every factor in \(\Phi_N^{\gamma,\mathrm{aff}}\) is a bijection
natural in \(N\), so their composite has the same properties.
Every \(\mu_{r_j}\) is an isomorphism, and hence so is
\(\mu_\gamma\).
We prove \eqref{eq:move-sequence-compatibility} by induction on
the length of the sequence.
For the empty sequence both sides equal
\(\Theta_{D_0,\sigma_0,N}(f)\).
For \(j\ge1\), write \(\gamma_j=(r_1,\ldots,r_j)\).
Applying \eqref{eq:mu-r-compatibility} to \(r_j\) and the map
\(f\circ\mu_{\gamma_{j-1}}\colon M_{j-1}\to N\), and then
using the induction hypothesis, gives
\begin{align*}
\Theta_{D_j,\sigma_j,N}(f\circ\mu_{\gamma_j})
&=\Phi_N^{r_j,\mathrm{aff}}
  \bigl(\Theta_{D_{j-1},\sigma_{j-1},N}
  (f\circ\mu_{\gamma_{j-1}})\bigr)\\
&=\Phi_N^{r_j,\mathrm{aff}}
  \left(\Phi_N^{\gamma_{j-1},\mathrm{aff}}
  \bigl(\Theta_{D_0,\sigma_0,N}(f)\bigr)\right)\\
&=\Phi_N^{\gamma_j,\mathrm{aff}}
  \bigl(\Theta_{D_0,\sigma_0,N}(f)\bigr).
\end{align*}
This proves \eqref{eq:move-sequence-compatibility}.
Taking \(N=M_0\) and \(f=\operatorname{id}_{M_0}\) gives
\eqref{eq:move-sequence-universal-coloring}.
If another \(A\)-linear map \(\nu\colon M_m\to M_0\)
satisfies \eqref{eq:move-sequence-compatibility}, the same
specialization gives
\(\nu\circ U_{D_m,\sigma_m}
=\Phi_{M_0}^{\gamma,\mathrm{aff}}(U_{D_0,\sigma_0})
=\mu_\gamma\circ U_{D_m,\sigma_m}\).
The uniqueness assertion in
Theorem~\ref{thm:universal-affine-coloring}, applied to
\((D_m,\sigma_m)\), yields \(\nu=\mu_\gamma\).

For concatenation, the coloring maps satisfy
\(\Phi_N^{\delta*\gamma,\mathrm{aff}}
=\Phi_N^{\delta,\mathrm{aff}}\circ
\Phi_N^{\gamma,\mathrm{aff}}\).
The defining module-map compositions occur in the reverse order,
so \(\mu_{\delta*\gamma}=\mu_\gamma\circ\mu_\delta\).
The empty composition is the identity.
Finally, if \(\gamma=(r_1,\ldots,r_m)\), then
\(\gamma^{-1}=(r_m^{-1},\ldots,r_1^{-1})\), and
\(\mu_{\gamma^{-1}}
=\mu_{r_m}^{-1}\circ\cdots\circ\mu_{r_1}^{-1}
=(\mu_{r_1}\circ\cdots\circ\mu_{r_m})^{-1}\).
This proves \eqref{eq:move-sequence-composition}.
\end{proof}

The preceding proposition concerns prescribed sequences and their
concatenations. It does not assert that two different sequences
with the same endpoints induce the same module isomorphism.
For invariance of the module isomorphism class, the existence of
one such sequence is sufficient.

\begin{theorem}
\label{thm:Gu-invariance}
Let \(R\) be a commutative ring with identity, let \(G\) be an abelian
group, put \(A=R[G]\), and fix \(u\in A^\times\).
If \((D,\sigma)\) and \((D',\sigma')\) represent the same
\(G\)-colored oriented link, then
\(M_u(D,\sigma;A)\cong M_u(D',\sigma';A)\) and
\(\mathcal A_u(D,\sigma)\sim\mathcal A_u(D',\sigma')\).
\end{theorem}

\begin{proof}
By Theorem~\ref{thm-equi-nchr}, choose a finite sequence
\(\gamma=(r_1,\ldots,r_m)\) of colored Reidemeister moves and
planar isotopies from \((D_0,\sigma_0)=(D,\sigma)\) to
\((D_m,\sigma_m)=(D',\sigma')\).
For every Reidemeister step
\(r_j\colon(D_{j-1},\sigma_{j-1})\to(D_j,\sigma_j)\),
Theorem~\ref{thm:reidemeister-representing-module} supplies an
\(A\)-module isomorphism
\(\mu_{r_j}\colon M_u(D_j,\sigma_j;A)
\to M_u(D_{j-1},\sigma_{j-1};A)\).
For a planar-isotopy step, corresponding semiarcs define the map
\(\overline{x'}\mapsto\overline x\).
The defining relators correspond under this assignment, and the
reverse correspondence gives its inverse.
The presentation matrices differ only by row and column permutations.
Thus \(\mu_\gamma=\mu_{r_1}\circ\cdots\circ\mu_{r_m}\)
is an \(A\)-module isomorphism from \(M_u(D',\sigma';A)\) to
\(M_u(D,\sigma;A)\), with the compatibility established in
Proposition~\ref{prop:move-sequence-compatibility}.
Each Reidemeister step gives presentation equivalence by
Theorem~\ref{thm:reidemeister-representing-module}, and each
planar-isotopy step gives it by the row and column permutations
just described. Transitivity proves equivalence of the endpoint
presentation matrices.
\end{proof}

\begin{corollary}
\label{cor-invariant-module}
For a \(G\)-colored oriented link \((L,\sigma)\), the isomorphism
class of \(M_u(D,\sigma;A)\), and equivalently the
presentation-equivalence class of \(\mathcal A_u(D,\sigma)\),
is independent of the chosen diagram \(D\).
\end{corollary}

\begin{proof}
Any two diagrams of \((L,\sigma)\) satisfy the hypotheses of
Theorem~\ref{thm:Gu-invariance}. That theorem identifies their
module isomorphism classes and their presentation-equivalence classes.
The equivalence of these two descriptions follows from
Theorem~\ref{thm-mod-pres} and the fact that each presentation
move preserves the cokernel up to isomorphism.
\end{proof}

\begin{definition}[\((G,u)\)-deformed Alexander module]
\label{def:Gu-module}
Let \(G\) be an abelian group, let \(R\) be a commutative ring with
identity, put \(A=R[G]\), and fix \(u\in A^\times\).
For a \(G\)-colored oriented link \((L,\sigma)\), choose a diagram
\((D,\sigma)\). The module \(M_u(D,\sigma;A)\), considered up to
\(A\)-module isomorphism, is denoted by \(M_u(L,\sigma;A)\) and
called the \emph{\((G,u)\)-deformed Alexander module} of
\((L,\sigma)\).
Any presentation matrix \(\mathcal A_u(D,\sigma)\) is called a
\emph{\((G,u)\)-deformed Alexander matrix} of \((L,\sigma)\).
\end{definition}

When applying \(\Hom_A(-,N)\), scalar extension, or other module
constructions to \(M_u(L,\sigma;A)\), we use a representative
of this isomorphism class. For a specified diagram we retain the
actual presentation module \(M_D\) and its universal coloring.
For a specified move \(r\), compatibility with the local extension
\(\Phi^{r,\mathrm{aff}}\) uniquely determines \(\mu_r\).
More generally, a specified sequence \(\gamma\) determines
\(\Phi^{\gamma,\mathrm{aff}}\) and its unique compatible map
\(\mu_\gamma\) by Proposition~\ref{prop:move-sequence-compatibility}.
The notation \(M_u(L,\sigma;A)\) records the invariant isomorphism
class; it does not select identifications independent of the chosen
sequence between all diagram presentations.

The construction so far concerns closed links.
To encode tangle composition one must also retain the boundary
semiarcs and their maps into the tangle module.
For a fixed diagram of a \(G\)-colored oriented tangle
\(T\colon w\to w'\), define $M_u(T;A)$ by the same semiarc generators
and crossing relations, retaining all boundary semiarcs. This gives
a module cospan
\[A^w\longrightarrow M_u(T;A)\longleftarrow A^{w'}.\]
Here \(A^w\) and \(A^{w'}\) are the free \(A\)-modules on the
incoming and outgoing boundary semiarcs, and the two maps record
their images in the tangle module.
For vertically composable tangles, the module of the composite is
the pushout of the corresponding cospans.
Applying \(\Hom_A(-,N)\) converts this pushout into the pullback
gluing law for affine colorings.
Horizontal juxtaposition corresponds to direct sum of module
cospans, while closure identifies the appropriate boundary legs.
These diagram-level gluing statements follow directly from the
presentations. If $b_1,b_2:A^{w'}\to M_1,M_2$ are the two maps on
a common boundary, the glued module is
$$(M_1\oplus M_2)/\langle(b_1(e),-b_2(e)):e\in A^{w'}\rangle.$$
A map from this quotient to $N$ is exactly a pair of maps
$f_i:M_i\to N$ satisfying $f_1b_1=f_2b_2$, which is the stated
pullback condition. The same presentation argument gives direct sum
for juxtaposition and identification of boundary generators for closure.

The natural isomorphism
\(\Hom_A(M_u(D,\sigma;A),N)\cong{\rm Col}_N^G(D,\sigma)\)
is a degree-zero representability statement.
The Yoneda lemma explains how natural Reidemeister extension
determines a compatible isomorphism of representing modules in the
reverse direction.
Its application rests on the explicit module construction, the
verification of the affine multibiquandle axioms, and the local
extension bijections proved above.
These supply the algebraic and diagrammatic inputs to the argument.
As explained in the Introduction, the universal-target formulation
also differs from fixing a \(G\)-Alexander biquandle and then
forming a cocycle enhancement~\cite{GAB}.

The degree-zero statement also has an elementary lifting consequence.
For an exact sequence $$0\to K\xrightarrow{\iota}E\xrightarrow{p}N\to0$$ of $A$-modules
and a coloring represented by $f_C:M_D\to N$, form the pullback
$$P_C=\{(m,e)\in M_D\oplus E:f_C(m)=p(e)\}.$$
The exact sequence $$0\to K\to P_C\to M_D\to0,$$ with maps
$k\mapsto(0,\iota(k))$ and $(m,e)\mapsto m$, splits if and only if
$f_C$ lifts to $\widetilde f_C:M_D\to E$: a lift gives the section
$m\mapsto(m,\widetilde f_C(m))$, and a section recovers the lift
by projection to $E$. Equivalently, its class in
$\Ext_A^1(M_D,K)$ vanishes precisely when $C$ lifts to an affine
$E$-coloring. The higher target-derived groups are
$\Ext_A^i(M_D,N)$ for $i\geq0$; no computation of those groups,
relation syzygies, factor-set defects, or Bockstein filtrations is
needed for the invariance and profile results proved here.
This target-variable lifting observation is distinct from the
extension associated with the incidence assembly sequence in Part~II.

The preceding construction also accommodates an arbitrary palette
\(\Omega\) equipped with a map to the abelian group \(G\).
No injectivity of this map is required, since distinct palette labels
are allowed to specify the same coefficient action.

\begin{theorem}
\label{thm-adbq-mod-2}
Let \(\Omega\) be a set, and let \((D,\sigma)\) and
\((D',\sigma')\) be diagrams of the same \(\Omega\)-colored
oriented link. Let \(G\) be an abelian group, let \(R\) be a
commutative ring with identity, and let \(\tau\colon\Omega\to G\)
be a map. Put \(A=R[G]\) and fix \(u\in A^\times\).
Then \(M_u(D,\tau\circ\sigma;A)
\cong M_u(D',\tau\circ\sigma';A)\) and
\(\mathcal A_u(D,\tau\circ\sigma)
\sim\mathcal A_u(D',\tau\circ\sigma')\).
\end{theorem}

\begin{proof}
By Theorem~\ref{thm-equi-nchr}, the two diagrams are related by a
finite sequence of Reidemeister moves and planar isotopies preserving
the \(\Omega\)-label of every component.
Applying \(\tau\) to all component labels turns the same sequence
into a sequence of \(G\)-colored moves and planar isotopies.
Every step still preserves the resulting component labels, even
when \(\tau\) identifies distinct elements of \(\Omega\).
Theorem~\ref{thm:Gu-invariance} therefore gives the module
isomorphism and the presentation equivalence.
\end{proof}

\begin{corollary}
\label{cor-invariant-omega}
Under the hypotheses of Theorem~\ref{thm-adbq-mod-2}, the module
\(M_u(L,\tau\circ\sigma;A)\), considered up to isomorphism, is an
invariant of the \(\Omega\)-colored oriented link \((L,\sigma)\).
Every matrix \(\mathcal A_u(D,\tau\circ\sigma)\) obtained from a
diagram \(D\) presents a representative of this invariant module.
\end{corollary}

\begin{proof}
Theorem~\ref{thm-adbq-mod-2} identifies the modules and presentation
matrices obtained from any two diagrams after applying \(\tau\).
Definition~\ref{def:Gu-module} then gives the invariant module and
identifies each displayed matrix as one of its presentations.
\end{proof}


\raggedbottom
\section{Braid presentations, scalar extension, and normalization}
\label{sec:parameters}

The preceding sections construct the invariant module from local affine
relations and identify the maps induced by Reidemeister moves. We now
explain how this construction behaves under changes of coefficients.
We first record an explicit closed-braid presentation, then establish
scalar extension and the universal Laurent parameter, and finally remove
the common deformation unit. These results supply the coefficient
framework for the Alexander comparisons in Part~II.

Throughout this section, $A=R[G]$, where $R$ is commutative and $G$ is
abelian, and $L$ is nonempty. All ring homomorphisms preserve the identity.
If $\varphi:A\to S$ is a homomorphism to a commutative ring, every
color $g\in G$ acts on an $S$-module through the unit $\varphi(g)$.
For a $G$-labeling $\sigma$, write
$\sigma_S=\varphi\circ\sigma:\pi_0(D)\to S^\times$.
More generally, for a unit-valued component labeling
$\lambda:\pi_0(D)\to S^\times$ and $v\in S^\times$, let
$M_v(D,\lambda;S)$ be the cokernel of the presentation assembled from
the local blocks $B_{\pm}(\lambda(K_{\mathrm{over}});v)$, where
$K_{\mathrm{over}}$ denotes the over-component at the crossing.
As before, each crossingless component contributes one free generator.
The diagram independence of this notation is justified below.

We also fix two algebraic conventions used later in this section.
For a commutative ring $B$, $\Spec B$ denotes its set of prime ideals,
$V(I)=\{\mathfrak p\in\Spec B:I\subseteq\mathfrak p\}$, and
$D(f)=\{\mathfrak p\in\Spec B:f\notin\mathfrak p\}$.
If $B^q\xrightarrow{C}B^p\to N\to0$ is a finite presentation and
$k\geq0$, the Fitting ideal $\Fitt_k^B(N)$ is generated by the
$(p-k)$-minors of $C$. It is $B$ when $p-k\leq0$ and is the zero
ideal when $p-k>\min\{p,q\}$; the determinant of the empty matrix
is $1$. These ideals are independent of the chosen finite presentation.
The notation will be used for the zeroth Fitting ideal here and for
the exact comparison locus in Part~II. Part~II develops the base-change and fiber criteria for these ideals
and computes the associated rank-jump subschemes.

\subsection{Local slots and braid closure}
For the braid and matrix calculations we use the following explicit
local slots. The four symbols at each crossing are distinct formal
occurrences; identifying occurrences belonging to the same semiarc
recovers the relators in Section~\ref{sect-abq}.
\begin{figure}[ht!]
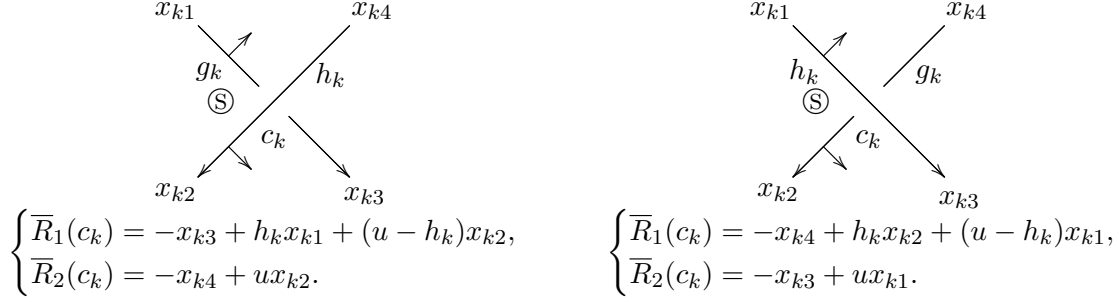

\centerline{\xy 
(-10,10);(-2,2)**@{-} **@{-},(10,-10);(2,-2)**@{-} **@{-} ?<*\dir{<}, 
(10,10);(-10,-10) **@{-} **@{-} ?>*\dir{>},(12,-12.5) *{x_{k3}},
(-13,-12) *{x_{k2}},(-13,12) *{x_{k1}}, (13,12) *{x_{k4}},(-8.5,4) *{g_k},(7.5,3.5) *{h_k},
(0,-5) *{c_k},(-3,9);(-6,6)**@{-} ?<*\dir{<},(11,-3.5) *{},(-8.5,-4) *{},
(-3,-9);(-6,-6)**@{-} ?<*\dir{<},(-7,0)*{\circledS},
(0.5,-20)*{\begin{cases}
\overline R_1(c_k)=-x_{k3} + h_kx_{k1} + (u-h_k)x_{k2}, \\
\overline R_2(c_k)=-x_{k4} + ux_{k2}.
\end{cases}},
\endxy\qquad\xy 
(10,-10);(-10,10) **@{-} **@{-} ?<*\dir{<},(10,10);(2,2)**@{-} **@{-},(-2,-2);(-10,-10) **@{-} **@{-} ?>*\dir{>},(12,-13) *{x_{k3}}, (-13,12) *{x_{k1}},(-12,-12) *{x_{k2}},(13,12) *{x_{k4}},(-8.5,4) *{h_k},(8,3.5) *{g_k},(9,-4) *{},
(0,-5) *{c_k},(-3,9);(-6,6)**@{-} ?<*\dir{<}, (-8.5,-4) *{},
(-3,-9);(-6,-6)**@{-} ?<*\dir{<},(-7,0)*{\circledS},
(1,-20)*{\begin{cases}
\overline R_1(c_k)=-x_{k4} + h_kx_{k2} + (u-h_k)x_{k1},\\
\overline R_2(c_k)=-x_{k3} + ux_{k1}.
\end{cases}},
\endxy}
\vskip.1cm
\caption{Linear local relators $\overline R_1(c_k)$ and $\overline R_2(c_k)$ associated with $c_k$}\label{fig-qd-rels-1}
\end{figure}

The local coefficient matrices are
\begin{equation}\label{loc-blk-mat1}
B_+(h_k;u)=
\begin{bmatrix}
h_k&0\\
u-h_k&u\\
-1&0\\
0&-1
\end{bmatrix},
\qquad
B_-(h_k;u)=
\begin{bmatrix}
u-h_k&u\\
h_k&0\\
0&-1\\
-1&0
\end{bmatrix}.
\end{equation}
Their rows are ordered as $(x_{k1},x_{k2},x_{k3},x_{k4})$ and
their columns as $(\overline R_1(c_k),\overline R_2(c_k))$.

\begin{lemma}[Braid transitions and the closure presentation]
\label{lem:braid-closure}
Let $A$ be a commutative ring, let $u\in A^\times$, and let
$\beta=\sigma_{i_1}^{\epsilon_1}\cdots
\sigma_{i_m}^{\epsilon_m}$ be a downward $n$-strand braid,
read from top to bottom. Label each component of its closure by a
unit of $A$, and transport these labels along the strands. Use the
crossing convention of Figure~\ref{fig-qd-rels-1}: at a positive
crossing the upper-right strand passes over, whereas at a negative
crossing the upper-left strand passes over. The adjacent transition
blocks, with upper and lower colors both ordered from left to right,
are
\begin{equation}\label{eq:braid-transitions}
T_h^+(u)=
\begin{bmatrix}0&u^{-1}\\ h&1-hu^{-1}\end{bmatrix},
\qquad
T_h^-(u)=
\begin{bmatrix}1-uh^{-1}&h^{-1}\\ u&0\end{bmatrix}
=\bigl(T_h^+(u)\bigr)^{-1},
\end{equation}
where $h$ is the current label of the overstrand. Embed each block
in positions $i_r,i_r+1$ of $A^n$, acting as the identity on the
other coordinates, and denote the resulting matrix by $T_r$.
Put $T_\beta=T_m\cdots T_1$; the rightmost factor acts first.
Then the closed-braid module has the presentation
\begin{equation}\label{eq:braid-closure-presentation}
M_u(\widehat\beta,\sigma;A)
\cong\coker\bigl((T_\beta-I_n)^{\mathsf T}\bigr).
\end{equation}
Here the module is defined by the same local relations as above,
with the component labels interpreted as units of $A$; in
particular, this applies to $A=R[G]$ and $G$-colored links.
For any $A$-module $N$, restriction to the top strands identifies
the affine coloring set with
$\ker(T_\beta-I_n:N^n\to N^n)$. In particular, over a field $k$
the module dimension equals the nullity of $T_\beta-I_n$, and over
$k=\mathbb F_q$ the number of colorings is $q$ raised to this
dimension.
\end{lemma}

\begin{proof}
At one crossing, write $(X,Y)$ for the upper generators and
$(P,Q)$ for the lower generators, in left-to-right order. The
positive local relators are
$r_1=-Q+hX+(u-h)P$ and $r_2=-Y+uP$.
Put $c_2=-u^{-1}r_2=-P+u^{-1}Y$ and
$c_1=r_1+(u-h)c_2=-Q+hX+(1-hu^{-1})Y$.
The resulting two relations are
\(P=u^{-1}Y\) and \(Q=hX+(1-hu^{-1})Y\).
They give $T_h^+(u)$. At a negative crossing, the local relators
are $r_1=-Y+hP+(u-h)X$ and $r_2=-Q+uX$.
Multiplying $r_1$ by $-h^{-1}$ gives
$-P+(1-uh^{-1})X+h^{-1}Y$. Thus the relations become
\(P=(1-uh^{-1})X+h^{-1}Y\) and \(Q=uX\),
which give $T_h^-(u)$. These replacements are invertible column
operations and use only the units $u$ and $h$. Explicitly,
\[
\begin{aligned}
T_h^+(u)T_h^-(u)
&=\begin{bmatrix}
1&0\\
h-u+(1-hu^{-1})u&1
\end{bmatrix}=I_2,\\
T_h^-(u)T_h^+(u)
&=\begin{bmatrix}
1&(1-uh^{-1})u^{-1}+h^{-1}(1-hu^{-1})\\
0&1
\end{bmatrix}=I_2.
\end{aligned}
\]

At either crossing the ordered component labels are interchanged:
if the labels above it are $(p,q)$, those below it are $(q,p)$.
Consequently $h=q$ at a positive crossing and $h=p$ at a negative
crossing. Updating this ordered label list after each crossing
determines every factor $T_r$. Because the labeling comes from the
closed link, its top and bottom lists agree under the usual
positionwise closure, even when the braid permutation is nontrivial.
The permutation is already included in the transition blocks, so
no additional permutation matrix occurs in
\eqref{eq:braid-closure-presentation}.

For completeness, subdivide the strands at horizontal levels
between consecutive crossings. This only adds a generator and a
relation equating it to the preceding generator at each subdivision;
the coefficient of the new generator is a unit. Let
$\mathbf x^{(r)}=(x_1^{(r)},\ldots,x_n^{(r)})^{\mathsf T}$
be the generators at the $r$th level, with $r=0$ at the top.
The normalized crossing and straight-strand relations are the
entries of
$\mathbf r^{(r)}=\mathbf x^{(r)}-T_r\mathbf x^{(r-1)}$.
Put $\mathbf s^{(1)}=\mathbf r^{(1)}$ and, for $r\geq2$,
replace this block of relators by
\(\mathbf s^{(r)} =\mathbf r^{(r)}+T_r\mathbf s^{(r-1)} =\mathbf x^{(r)}-T_r\cdots T_1\mathbf x^{(0)}\).
This is an invertible block-triangular change of relation
generators: each new block equals the old block plus a combination
of earlier blocks. The coefficient of $x_j^{(r)}$ in its own
relation is $1$, so it can be eliminated together with that
relation. After eliminating all intermediate generators, the
closure equations $\mathbf x^{(m)}=\mathbf x^{(0)}$ leave exactly
$(T_\beta-I_n)\mathbf x^{(0)}=0$, up to a simultaneous change of
relation signs. The coefficients of each scalar relator form a
column of the presentation matrix, which explains the transpose
in \eqref{eq:braid-closure-presentation}. If $m=0$, take
$T_\beta=I_n$; the closure is the $n$-component unlink and its
module is $A^n$, as the same formula states.

The same calculation in $N$ assigns all intermediate colors
uniquely to any top vector $v\in N^n$, and the assignment closes
if and only if $(T_\beta-I_n)v=0$. Over a field, therefore,
\[
\begin{aligned}
\dim_k M_u(\widehat\beta,\sigma;k)
&=n-\operatorname{rank}_k\bigl((T_\beta-I_n)^{\mathsf T}\bigr)\\
&=n-\operatorname{rank}_k(T_\beta-I_n)
=\dim_k\ker(T_\beta-I_n).
\end{aligned}
\]
If $k=\mathbb F_q$, a vector space of this dimension has precisely
$q^{\dim_k\ker(T_\beta-I_n)}$ elements, giving the asserted
coloring count.
\end{proof}

\subsection{Scalar extension and the universal parameter}
\label{subsec:base-change}

\begin{theorem}[Base change]
\label{thm:base-change}
Let $(D,\sigma)$ be a $G$-colored diagram, let $u\in A^\times$, and
let $\varphi:A\to S$ be a homomorphism to a commutative ring. There
are canonical isomorphisms
\begin{equation}
\label{eq:base-change-presentation}
M_u(D,\sigma;A)\otimes_A S
\cong
\coker\bigl(\varphi(\mathcal A_u(D,\sigma))\bigr)
\cong
M_{\varphi(u)}(D,\sigma_S;S).
\end{equation}
Their composite $b_D^\varphi$ is characterized by
$b_D^\varphi([x]\otimes s)=s[x]_S$ for each semiarc $x$, where
$[x]_S$ is its class in the specialized module. In particular, it
carries the coloring $x\mapsto[x]\otimes1$ to the universal affine
coloring over $S$. No flatness assumption on $S$ over $A$ is required.
Consequently, scalar extension of the invariant module is again a
colored-link invariant over $S$.
\end{theorem}

\begin{proof}
Let $\ell$ be the number of crossings and let $m$ be the number of
semiarc generators, including one generator for each crossingless
component. Tensor the right-exact presentation
$$A^{2\ell}\xrightarrow{\mathcal A_u(D,\sigma)}A^m
\longrightarrow M_u(D,\sigma;A)\longrightarrow 0$$ with $S$.
The resulting sequence
$$S^{2\ell}\xrightarrow{\varphi(\mathcal A_u(D,\sigma))}S^m
\longrightarrow M_u(D,\sigma;A)\otimes_A S\longrightarrow 0$$
is right exact. This proves the first isomorphism in
\eqref{eq:base-change-presentation}, without requiring tensor product
to preserve an injection.

Applying $\varphi$ entrywise to the positive and negative local blocks
gives, respectively,
\[
\begin{bmatrix}
\varphi(h)&0\\
\varphi(u)-\varphi(h)&\varphi(u)\\
-1&0\\
0&-1
\end{bmatrix},
\qquad
\begin{bmatrix}
\varphi(u)-\varphi(h)&\varphi(u)\\
\varphi(h)&0\\
0&-1\\
-1&0
\end{bmatrix}.
\]
The incidence entries $0,1$ are unchanged. Hence the specialized
blocks assemble to $\mathcal A_{\varphi(u)}(D,\sigma_S)$, proving
the second isomorphism. The induced map on the cokernels is precisely
$b_D^\varphi([x]\otimes s)=s[x]_S$. Since the semiarc classes define
the universal affine colorings, this also proves the asserted
compatibility with those colorings.

For an $S$-module $N$, regarded as an $A$-module through $\varphi$,
the original affine relations act by the same coefficients as the
specialized relations. Thus the tensor--Hom adjunction
$\Hom_S(M_u(D,\sigma;A)\otimes_A S,N)
\cong\Hom_A(M_u(D,\sigma;A),N)$ identifies both modules in
\eqref{eq:base-change-presentation} as representing the same affine
coloring functor on $S$-modules. Explicitly, an $S$-linear map $f$
on the tensor product corresponds to the $A$-linear map
$m\mapsto f(m\otimes1)$, and both give the same coloring of $D$.
This identification is natural in $N$.

Finally, the isomorphisms of Theorem~\ref{thm:Gu-invariance} remain
isomorphisms after tensoring with $S$: the tensor of an inverse is an
inverse of the tensor map. Combining them with $b_D^\varphi$ proves
diagram independence after scalar extension.
\end{proof}

The same argument establishes diagram independence for arbitrary
unit-valued labelings. Indeed, put $H=S^\times$ and
$\widetilde A=\mathbb Z[H]$, and denote by $[s]$ the group-ring basis
element corresponding to $s\in S^\times$. The homomorphism
$\epsilon_S:\widetilde A\to S$ determined by
$\epsilon_S([s])=s$ sends the unit $[v]$ to $v$ and the component
label $[\lambda(K)]$ to $\lambda(K)$. Theorem~\ref{thm:base-change}
therefore identifies $M_v(D,\lambda;S)$ with the scalar extension
of the group-ring module with these labels and parameter $[v]$.
Its diagram independence follows from Theorem~\ref{thm:Gu-invariance},
so we may write $M_v(L,\lambda;S)$ for the resulting module,
understood up to isomorphism.
For every $S$-module $N$, the same quotient presentation identifies
$S$-linear maps $M_v(D,\lambda;S)\to N$ with affine colorings of
$(D,\lambda)$ with values in $N$.

The base-change isomorphisms also respect successive coefficient
homomorphisms. If $\psi:S\to T$, the canonical identification
$(M\otimes_A S)\otimes_S T\cong M\otimes_A T$ sends
$(m\otimes s)\otimes t$ to $m\otimes\psi(s)t$.
Under this identification, specialization first along $\varphi$ and
then along $\psi$ agrees with specialization along $\psi\circ\varphi$:
both send $([x]\otimes s)\otimes t$ to $\psi(s)t[x]_T$.

\begin{proposition}[Compatibility with Reidemeister transport]
\label{prop:base-change-reidemeister}
Let $r:(D,\sigma)\to(D',\sigma')$ be a generalized Reidemeister
move. Write $M_D=M_u(D,\sigma;A)$ and
$M_D^S=M_{\varphi(u)}(D,\sigma_S;S)$, and use analogous notation
for $D'$. Let $\mu_r:M_{D'}\to M_D$ and
$\mu_r^S:M_{D'}^S\to M_D^S$ be the isomorphisms determined by the
affine Reidemeister extension maps over $A$ and $S$, respectively.
Then
\begin{equation}
\label{eq:base-change-reidemeister}
b_D^\varphi\circ(\mu_r\otimes_A\id_S)
=\mu_r^S\circ b_{D'}^\varphi.
\end{equation}
The same identity holds for the transport maps associated with any
fixed finite sequence of generalized Reidemeister moves.
\end{proposition}

\begin{proof}
Write $U_D=U_{D,\sigma}$, and let $U_D^S$ be the universal affine
coloring with values in $M_D^S$. The $A$-linear map
$j_D:M_D\to M_D^S$ defined by $j_D(m)=b_D^\varphi(m\otimes1)$
satisfies $j_D\circ U_D=U_D^S$, where $M_D^S$ is regarded as an
$A$-module through $\varphi$.

The local Reidemeister extension formulas over $S$ are the
specializations of the formulas over $A$: all coefficients are
formed by addition, multiplication, and inverses of the relevant
units, and $\varphi(a^{-1})=\varphi(a)^{-1}$ for every
$a\in A^\times$. Equivalently, the extension problem for an
$S$-module is unchanged when that module is regarded as an
$A$-module. These extension maps are natural in the target and,
together with the quotient representability just established,
determine $\mu_r^S$ by the argument of
Theorem~\ref{thm:reidemeister-representing-module}.
Denote the specialized extension map by
$\Phi^{r,\mathrm{aff},S}$. Naturality and
\eqref{eq:mu-r-universal-coloring} give
\[
\begin{aligned}
j_D\circ\mu_r\circ U_{D',\sigma'}
&=j_D\circ\Phi_{M_D}^{r,\mathrm{aff}}(U_D)
 =\Phi_{M_D^S}^{r,\mathrm{aff},S}(j_D\circ U_D)\\
&=\Phi_{M_D^S}^{r,\mathrm{aff},S}(U_D^S)
 =\mu_r^S\circ U_{D'}^S.
\end{aligned}
\]
Since $U_{D'}^S(x')=b_{D'}^\varphi([x']\otimes1)$, the two sides
of \eqref{eq:base-change-reidemeister} agree on every generator
$[x']\otimes1$. They are $S$-linear, so they agree on all of
$M_{D'}\otimes_A S$. The assertion for a fixed sequence follows by
composition, using Proposition~\ref{prop:move-sequence-compatibility}.
\end{proof}

\begin{definition}[Universal Laurent deformation]
\label{def:universal-deformation}
Let $z$ be an indeterminate and set $A_z=A[z^{\pm1}]$. Replacing
$u$ by $z$ in every local block defines
\[
\mathscr M_z(D,\sigma)
:=\coker\bigl(\mathcal A_z(D,\sigma)\bigr),
\]
an $A_z$-module called the \emph{universal Laurent deformation} of
the $(G,u)$-deformed Alexander module.
\end{definition}

Since $A_z\cong R[z^{\pm1}][G]$ and $z\in A_z^\times$,
Theorem~\ref{thm:Gu-invariance} applies directly to this presentation.
We therefore write $\mathscr M_z(L,\sigma)$ for the resulting
module, understood up to isomorphism. For each diagram, its universal
affine coloring is again given by the semiarc classes. The following
evaluation property makes precise the role of the universal parameter.

\begin{corollary}[Evaluation of the universal parameter]
\label{cor:universal-evaluation}
For every $u\in A^\times$, evaluation
$\operatorname{ev}_u:A_z\to A$, defined by $z\mapsto u$ and
$a\mapsto a$ for $a\in A$, gives
\begin{equation}
\label{eq:universal-evaluation}
\mathscr M_z(L,\sigma)\otimes_{A_z,\operatorname{ev}_u}A
\cong M_u(L,\sigma;A).
\end{equation}
More generally, if $\varphi:A\to S$ is a homomorphism to a
commutative ring and $v\in S^\times$, there is a unique
homomorphism $\widetilde\varphi_v:A_z\to S$ extending $\varphi$
and sending $z$ to $v$, and it induces an isomorphism
\begin{equation}
\label{eq:universal-evaluation-general}
\mathscr M_z(L,\sigma)
\otimes_{A_z,\widetilde\varphi_v}S
\cong M_v(L,\sigma_S;S).
\end{equation}
For a fixed diagram, these isomorphisms preserve the universal affine
colorings and commute with Reidemeister transport. Thus all unit
choices of the deformation parameter, together with changes of
coefficients, arise from this single Laurent presentation.
\end{corollary}

\begin{proof}
Because $v$ is a unit, the formula
$\widetilde\varphi_v(\sum_n a_nz^n)=\sum_n\varphi(a_n)v^n$
defines a homomorphism $A_z\to S$, including for negative exponents.
Any extension of $\varphi$ sending $z$ to $v$ must send $z^{-1}$
to $v^{-1}$, so this extension is unique. Applying it to the local
blocks replaces $h$ by $\varphi(h)$ and $z$ by $v$.
Theorem~\ref{thm:base-change}, applied over
$A_z\cong R[z^{\pm1}][G]$, gives
\eqref{eq:universal-evaluation-general}. On each diagram the map is
$[x]\otimes s\mapsto s[x]_S$, so it preserves the universal affine
coloring. Compatibility with Reidemeister transport follows from
Proposition~\ref{prop:base-change-reidemeister}. Taking
$S=A$, $\varphi=\id_A$, and $v=u$ gives
\eqref{eq:universal-evaluation}.
\end{proof}

\subsection{Removal of the common unit}

We use the adjacent braid transition blocks in
\eqref{eq:braid-transitions}, keeping track of the current component
label on the over-strand.

\begin{theorem}[Removal of the common deformation unit]
\label{thm:all-color-reduction}
Let $S$ be a commutative ring, $u\in S^\times$, and
$\lambda:\pi_0(L)\to S^\times$ a labeling of a classical oriented link.
Then
\begin{equation}\label{eq:all-color-reduction}
M_u(L,\lambda;S)\cong M_1(L,u^{-1}\lambda;S).
\end{equation}
For independent component variables $a_i$ and $t_i=a_i/z$, the
universal deformation is therefore the scalar extension of
$M_1(L,(t_1,\ldots,t_\mu);R[t_1^{\pm1},\ldots,t_\mu^{\pm1}])$
to $R[t_1^{\pm1},\ldots,t_\mu^{\pm1},z^{\pm1}]$.
If $L$ is nonempty, $M_u(L,\lambda;S)$ has a free rank-one direct
summand and $\Fitt_0^S(M_u(L,\lambda;S))=0$.
\end{theorem}

\begin{proof}
Choose a colored $n$-braid whose closure is $L$, and put
$D_n=\operatorname{diag}(1,u,\ldots,u^{n-1})$.
On adjacent positions $i,i+1$, the restriction of $D_n$ is
$u^{i-1}\operatorname{diag}(1,u)$. The common scalar cancels in
conjugation. Thus the positive block satisfies
\[
\begin{bmatrix}1&0\\0&u^{-1}\end{bmatrix}
\begin{bmatrix}0&u^{-1}\\h&1-hu^{-1}\end{bmatrix}
\begin{bmatrix}1&0\\0&u\end{bmatrix}
=\begin{bmatrix}0&1\\h/u&1-h/u\end{bmatrix}
=T_{h/u}^+(1).
\]
For the negative block, the corresponding multiplication is
\[
\begin{bmatrix}1&0\\0&u^{-1}\end{bmatrix}
\begin{bmatrix}1-uh^{-1}&h^{-1}\\u&0\end{bmatrix}
\begin{bmatrix}1&0\\0&u\end{bmatrix}
=\begin{bmatrix}1-u/h&u/h\\1&0\end{bmatrix}
=T_{h/u}^-(1).
\]
Every other coordinate is fixed by the embedded crossing block.
The same diagonal matrix therefore works at every crossing, for
either sign and every current over-color $h$. Rescaling the labels
does not change their permutation along the braid. If the successive
transition matrices are $T_1,\ldots,T_r$, their product satisfies
\[D_n^{-1}(T_r\cdots T_1)D_n =(D_n^{-1}T_rD_n)\cdots(D_n^{-1}T_1D_n).\]
Consequently
$D_n^{-1}T_\beta(u,\lambda)D_n
=T_\beta(1,u^{-1}\lambda)$.

Write $x_1,\ldots,x_n$ for the top semiarc generators. Each local
relation has a unit coefficient at the generator being eliminated,
so the internal braid generators can be expressed successively in
terms of these top generators. The resulting bottom color vector
is $T_\beta$ times the top color vector. Closing the braid imposes
$T_\beta\mathbf x=\mathbf x$, and the matrix with these relations
as columns is $(T_\beta-I_n)^{\mathsf T}$.
The same diagonal $D_n$ acts at the top and bottom, so this change
of coordinates respects those closure identifications, even when
the braid permutation is nontrivial.

Put $T=T_\beta(u,\lambda)$ and
$T'=T_\beta(1,u^{-1}\lambda)=D_n^{-1}TD_n$.
The two closure matrices satisfy
\[(T'-I_n)^{\mathsf T} =D_n^{\mathsf T}(T-I_n)^{\mathsf T}(D_n^{-1})^{\mathsf T}.\]
Both outside factors are invertible. More explicitly, if
$[y_1],\ldots,[y_n]$ are the top generator classes for the
normalized presentation, the map $[x_i]\mapsto u^{i-1}[y_i]$
descends to an isomorphism of cokernels, with inverse
$[y_i]\mapsto u^{1-i}[x_i]$. This proves
\eqref{eq:all-color-reduction}.

For the universal statement, the Laurent substitutions $t_i=a_i/z$
and $a_i=zt_i$ identify
$R[a_1^{\pm1},\ldots,a_\mu^{\pm1},z^{\pm1}]$ with
$R[t_1^{\pm1},\ldots,t_\mu^{\pm1},z^{\pm1}]$.
After the preceding normalization, every local entry belongs to
$R[t_1^{\pm1},\ldots,t_\mu^{\pm1}]$; adjoining $z^{\pm1}$ is
therefore precisely the scalar extension of
Theorem~\ref{thm:base-change}.

At parameter $1$, every relator has coefficient sum zero, so
$[x_j]\mapsto1$ defines an $S$-linear map
$\varepsilon:M_1\to S$. Since $L$ is nonempty, choose a semiarc
$x_1$. The map $a\mapsto a[x_1]$ is a section, and the splitting
and its inverse are
$m\mapsto(\varepsilon(m),m-\varepsilon(m)[x_1])$ and
$(a,k)\mapsto a[x_1]+k$, respectively. Thus
$M_1\cong S\oplus\ker\varepsilon$.
In the generator basis $x_1,x_2-x_1,\ldots,x_m-x_1$, the first
coordinate of a relator is its original coefficient sum and hence
is zero. The presentation has a zero row, so every maximal minor
vanishes. If there are fewer columns than rows, the same conclusion
holds by the convention defining the zeroth Fitting ideal. Therefore
$\Fitt_0^S(M_1)=0$. Transport the splitting and this ideal through
\eqref{eq:all-color-reduction} to obtain the last assertion.
\end{proof}

This is an explicit linear realization of the classical-link
biquandle-to-quandle reduction \cite{IT,Murao}. The rescaled labels are
units of $S$; they need not belong to the originally prescribed subgroup
$G$. In universal free colors, the augmentation divisors $t_i=1$
correspond exactly to the collision divisors $a_i=z$.

\section{Associated biquandles and color-profile polynomials}\label{sec:profiles}
\label{subsec:associated-biquandle}

The associated di(bi)quandle construction for two component colors was
studied in~\cite{KLS2}; related associated and multiple-conjugation
constructions for $G$-families of Alexander biquandles appear in
\cite{Murao,GAB}.  We record the general construction because its
precise fiber statement connects ordinary biquandle colorings to the
modules above. Unlike the image of an individual coloring under a
biquandle projection~\cite{LeeProjection}, the profile below records
the sizes of entire fibers over component labelings.  In particular, the coloring set for one fixed component
labeling is a \emph{fiber}; it is not the entire associated-biquandle
coloring set.
The fiber and profile statements allow crossingless components;
their contribution to the affine module is the free summand described
in Remark~\ref{rem-square-presentation}.

Let
$X=(X,\{\underline{\triangleright}_i,
\overline{\triangleright}_i\}_{i\in\Omega})$
be an $\Omega$-multibiquandle.  Define two operations on
$X\times\Omega$ by
\begin{equation}\label{eq:associated-biquandle-operations}
\begin{aligned}
 (x,i)\mathbin{\underline{\triangleright}}(y,j)
   &=\bigl(x\mathbin{\underline{\triangleright}_j}y,i\bigr),\\
 (x,i)\mathbin{\overline{\triangleright}}(y,j)
   &=\bigl(x\mathbin{\overline{\triangleright}_j}y,i\bigr).
\end{aligned}
\end{equation}

\begin{theorem}[Associated biquandle and its coloring fibers]
\label{thm:associated-biquandle-fibers}
The operations~\eqref{eq:associated-biquandle-operations} make
$\operatorname{As}_{\Omega}(X):=X\times\Omega$ the
\emph{associated biquandle} of $X$. For an uncolored oriented link
diagram $D$, the projection
$\pi:\operatorname{Col}_{\operatorname{As}_{\Omega}(X)}(D)
\to\Omega^{\pi_0(D)}$ to component labelings induces a bijection,
natural in $X$ with respect to $\Omega$-multibiquandle homomorphisms,
\begin{equation}\label{eq:associated-coloring-decomposition}
 \operatorname{Col}_{\operatorname{As}_{\Omega}(X)}(D)
 \;\cong\;
 \coprod_{\sigma:\pi_0(D)\to\Omega}
 \operatorname{Col}^{\Omega}_X(D,\sigma).
\end{equation}
Equivalently, the fiber over a fixed component labeling $\sigma$ is
canonically $\operatorname{Col}^{\Omega}_X(D,\sigma)$.
\end{theorem}

\begin{proof}
We verify \textbf{BQ1}--\textbf{BQ4} of Definition~\ref{def-biqdle}.
For \textbf{BQ1}, both operations on $((x,i),(x,i))$ give the same
result by \textbf{MBQ1}. For \textbf{BQ2}, the right-operation inverses
with second input $(y,j)$ send $(x,i)$ to
$(x\mathbin{\underline{\triangleright}_j^{-1}}y,i)$ and
$(x\mathbin{\overline{\triangleright}_j^{-1}}y,i)$, respectively;
the two inverse identities are those of the biquandle with index $j$.
For \textbf{BQ3}, the switch is
\[
 S\bigl((x,i),(y,j)\bigr)
 =\bigl((y\mathbin{\overline{\triangleright}_i}x,j),
        (x\mathbin{\underline{\triangleright}_j}y,i)\bigr).
\]
If $H_{ij}^{-1}(r,s)=(x,y)$, then
$S^{-1}((r,j),(s,i))=((x,i),(y,j))$.
The existence of $H_{ij}^{-1}$ in \textbf{MBQ2} gives this inverse on
every ordered label fiber, proving \textbf{BQ3}.
For \textbf{BQ4}, take $(x,i),(y,j),(z,k)$. Both sides of each of its
three identities have second coordinate $i$. The first law reduces on first coordinates to
\[
(x\mathbin{\underline{\triangleright}_j}y)
 \mathbin{\underline{\triangleright}_k}
 (z\mathbin{\underline{\triangleright}_j}y)
=
(x\mathbin{\underline{\triangleright}_k}z)
 \mathbin{\underline{\triangleright}_j}
 (y\mathbin{\overline{\triangleright}_k}z).
\]
This is the first identity in \textbf{MBQ3} with indices $(j,k)$.
For the second and third laws, the first-coordinate equations are
\begin{align*}
&(x\mathbin{\underline{\triangleright}_j}y)
 \mathbin{\overline{\triangleright}_k}
 (z\mathbin{\underline{\triangleright}_j}y)
=(x\mathbin{\overline{\triangleright}_k}z)
 \mathbin{\underline{\triangleright}_j}
 (y\mathbin{\overline{\triangleright}_k}z),\\
&(x\mathbin{\overline{\triangleright}_j}y)
 \mathbin{\overline{\triangleright}_k}
 (z\mathbin{\overline{\triangleright}_j}y)
=(x\mathbin{\overline{\triangleright}_k}z)
 \mathbin{\overline{\triangleright}_j}
 (y\mathbin{\underline{\triangleright}_k}z).
\end{align*}
These are the second and third \textbf{MBQ3} identities with indices
$(j,k)$. Thus \textbf{BQ4} holds, and
$\operatorname{As}_{\Omega}(X)$ is a biquandle.

Now let $\mathcal C$ be an
$\operatorname{As}_{\Omega}(X)$-coloring.  Formula
\eqref{eq:associated-biquandle-operations} shows that the second
coordinate is unchanged as a semiarc passes through either the under- or
the over-strand of a crossing.  It is therefore constant on each
component and determines a unique map
$\sigma_{\mathcal C}:\pi_0(D)\to\Omega$.  The first coordinates then
satisfy precisely the $\Omega$-multibiquandle coloring rules for
$(D,\sigma_{\mathcal C})$.  Conversely, adjoining the component label
$\sigma(K)$ to the first coordinate on every semiarc of $K$ turns an
$X$-coloring of $(D,\sigma)$ into an
$\operatorname{As}_{\Omega}(X)$-coloring.  The two constructions are inverse. An $\Omega$-multibiquandle
homomorphism $f:X\to Y$ sends $(x,i)$ to $(f(x),i)$; taking first
coordinates and adjoining component labels both commute with this map.
This proves naturality and~\eqref{eq:associated-coloring-decomposition}.
\end{proof}

\begin{definition}[Color-profile polynomial]
\label{def:color-profile-polynomial}
Assume that $X$ and $\Omega$ are finite, and let $L$ have $\mu$
components. For a labeling $\sigma:\pi_0(L)\to\Omega$, write
$|\operatorname{Col}^{\Omega}_X(L,\sigma)|$ for the coloring count
of any diagram of $(L,\sigma)$, which is independent of the diagram
by Theorem~\ref{thm-n-mbqd-cinv}. We use the analogous notation for
the associated-biquandle coloring count. These notations concern
cardinalities and do not require preferred identifications between
coloring sets of different diagrams.
The \emph{color-profile polynomial} is
\begin{equation}\label{eq:color-profile-polynomial}
 \Phi_X^{\Omega}(L;\alpha,\xi)
 :=\sum_{\sigma:\pi_0(L)\to\Omega}
 \alpha^{|\operatorname{im}\sigma|}
 \xi^{|\operatorname{Col}^{\Omega}_X(L,\sigma)|}
 \in\mathbb Z[\alpha,\xi].
\end{equation}
\end{definition}

\begin{corollary}[Profile invariance and decategorifications]
\label{cor:color-profile-invariance}
The polynomial $\Phi_X^{\Omega}(L;\alpha,\xi)$ is an invariant of the
uncolored oriented link $L$.  It satisfies
\begin{align}
 \Phi_X^{\Omega}(L;1,1)&=|\Omega|^{\mu},
 \label{eq:profile-number-labelings}\\
 \left.\frac{\partial}{\partial\xi}
 \Phi_X^{\Omega}(L;1,\xi)\right|_{\xi=1}
 &=|\operatorname{Col}_{\operatorname{As}_{\Omega}(X)}(L)|.
 \label{eq:profile-associated-count}
\end{align}
If $\Omega=G$ is a finite abelian group and $D$ is a chosen diagram
of $L$, projection to component labels gives a canonical map
$\Pi_D:\operatorname{Col}_{\operatorname{As}_{G}(X)}(D)
\to\operatorname{Hom}(\pi_1(S^3\setminus\nu L),G)$,
and
\begin{equation}\label{eq:representation-profile}
 \Phi_X^G(L;1,\xi)
 =\sum_{\rho\in\operatorname{Hom}(\pi_1(S^3\setminus\nu L),G)}
   \xi^{|\Pi_D^{-1}(\rho)|}.
\end{equation}
On the common domain of finite abelian $G$-families of biquandles,
\eqref{eq:representation-profile} agrees with the classical-link
restriction of the representation-fiber enhancement of Ishii--Nelson
\cite[Definition~6.2 and Proposition~6.3]{IN}. Indeed, the conjugation
of the group coordinate in their associated construction is trivial
when $G$ is abelian; see \cite[Proposition~5.4]{IN}.
The present formulation also allows arbitrary finite palettes and
records the number of palette labels used.
\end{corollary}

\begin{proof}
An oriented isotopy matches the components of two diagrams.  For each
fixed $\sigma$, Theorem~\ref{thm-n-mbqd-cinv} gives a bijection of the
corresponding coloring sets, and it preserves
$|\operatorname{im}\sigma|$.  This proves invariance.
Equation~\eqref{eq:profile-number-labelings} counts all maps from a
$\mu$-element set to $\Omega$.  Differentiating at $\xi=1$ sums the
cardinalities of all fibers, which equals the associated-biquandle count
by Theorem~\ref{thm:associated-biquandle-fibers}.

When $G$ is abelian, every homomorphism from the link group to $G$
factors uniquely through
$H_1(S^3\setminus\nu L;\mathbb Z)\cong\mathbb Z^{\mu}$.  Consequently,
component labelings $\sigma:\pi_0(L)\to G$ are in canonical bijection
with such homomorphisms by sending a positively oriented meridian of a
component $K$ to $\sigma(K)$.  Under this bijection,
Theorem~\ref{thm:associated-biquandle-fibers} identifies the fiber of
$\Pi_D$ with $\operatorname{Col}^G_X(D,\sigma)$, after transporting
$\sigma$ to the components of $D$. Taking cardinalities proves
\eqref{eq:representation-profile}.
\end{proof}

The affine construction makes the preceding profile computable directly
from the modules of this paper.

\begin{proposition}[Affine representability of the coloring fibers]
\label{prop:associated-affine-fibers}
Let $A=R[G]$, let $u\in A^\times$, let $\Omega$ be a finite
nonempty palette, and fix a map
$\tau:\Omega\to G$.  Let $N$ be a finite $A$-module, equipped with the
$\Omega$-multibiquandle operations
$x\mathbin{\underline{\triangleright}_i}y
=\tau(i)x+(u-\tau(i))y$ and
$x\mathbin{\overline{\triangleright}_i}y=ux$.
Then
\begin{equation}\label{eq:affine-profile-hom}
 \Phi_N^{\Omega}(L;\alpha,\xi)
 =\sum_{\sigma:\pi_0(L)\to\Omega}
 \alpha^{|\operatorname{im}\sigma|}
 \xi^{|\operatorname{Hom}_A(M_u(L,\tau\circ\sigma;A),N)|}.
\end{equation}
In particular, each component-label fiber is the value at $N$ of a
represented affine coloring functor.
The exponent $|\operatorname{im}\sigma|$ is computed in $\Omega$;
distinct palette elements remain distinct labels even when $\tau$
assigns them the same element of $G$.
\end{proposition}

\begin{proof}
Fix $\sigma$ and put $M=M_u(D,\tau\circ\sigma;A)$. For a coloring
$C$, send the class of a semiarc $x$ to $C(x)$. The crossing conditions
say exactly that this assignment kills every defining relator, so it
induces a unique homomorphism $f_C:M\to N$. Conversely, an $A$-linear
map $f$ gives the coloring $x\mapsto f([x])$. These assignments are
inverse by Theorem~\ref{thm:representability}. Taking their common
cardinality for each $\sigma$ and inserting it in
\eqref{eq:color-profile-polynomial} proves \eqref{eq:affine-profile-hom}.
Theorem~\ref{thm:Gu-invariance} makes every summand independent of $D$.
\end{proof}

For a nonempty classical link, Theorem~\ref{thm:all-color-reduction}
gives a splitting
$M_u(L,\tau\circ\sigma;A)\cong A\oplus M_\sigma^{\mathrm{red}}$,
where $M_\sigma^{\mathrm{red}}$ denotes a complementary summand.
Consequently,
$\Hom_A(M_u(L,\tau\circ\sigma;A),N)
\cong N\times\Hom_A(M_\sigma^{\mathrm{red}},N)$.
Every affine coloring fiber therefore has cardinality
$|N|\,|\Hom_A(M_\sigma^{\mathrm{red}},N)|$, a positive multiple
of $|N|$. This numerical conclusion does not require a canonical
choice of the splitting.

\begin{corollary}[Finite-field Betti profile]
\label{cor:finite-field-color-profile}
In the setting of Proposition~\ref{prop:associated-affine-fibers}, let
$\lambda:A\to k$ be a homomorphism to a finite field of order $Q$, and
regard $k$ as an $A$-module.  Define
$d_{\lambda,\sigma}
=\dim_k\bigl(M_u(L,\tau\circ\sigma;A)
\otimes_{A,\lambda}k\bigr)$.
Then
\begin{equation}\label{eq:finite-field-hom-count}
 |\operatorname{Hom}_A(M_u(L,\tau\circ\sigma;A),k)|
 =Q^{d_{\lambda,\sigma}}.
\end{equation}
Consequently the \emph{finite-field Betti profile}
\begin{equation}\label{eq:betti-profile-polynomial}
 \mathcal B_{\lambda}^{\Omega}(L;\alpha,t)
 =\sum_{\sigma:\pi_0(L)\to\Omega}
   \alpha^{|\operatorname{im}\sigma|}t^{d_{\lambda,\sigma}}
\end{equation}
is an uncolored oriented-link invariant and determines the corresponding affine
color-profile polynomial by the monomial transformation
$\alpha^r t^d\mapsto\alpha^r\xi^{Q^d}$, extended additively.
This transformation is applied term by term, rather than by an ordinary
substitution for $t$. For fixed $Q$, it is injective, so the Betti profile
and the corresponding affine color-profile polynomial determine each
other. For nonempty classical links, $d_{\lambda,\sigma}\geq1$.
\end{corollary}

\begin{proof}
Write $M=M_u(L,\tau\circ\sigma;A)$. An $A$-linear map $f:M\to k$
extends uniquely to the $k$-linear map
$\widetilde f:M\otimes_Ak\to k$ defined by
$\widetilde f(m\otimes c)=c f(m)$; the tensor relations hold because
$f(am)=\lambda(a)f(m)$. Conversely, $\widetilde f$ restricts by
$f(m)=\widetilde f(m\otimes1)$. Thus $\Hom_A(M,k)\cong\Hom_k(M\otimes_Ak,k)$.
A linear functional on a $d_{\lambda,\sigma}$-dimensional vector space
is determined freely by its values on a basis, giving
$Q^{d_{\lambda,\sigma}}$ choices. This proves
\eqref{eq:finite-field-hom-count}. The dimension is invariant by base
change of the invariant module. Inserting these counts in the defining sum gives
\eqref{eq:betti-profile-polynomial} and the monomial transformation.
Finally, distinct nonnegative integers $d$ give distinct exponents
$Q^d$, proving that this transformation retains all coefficients.
For a nonempty classical link, the free rank-one summand of
Theorem~\ref{thm:all-color-reduction} remains a copy of $k$ after
scalar extension. Hence $d_{\lambda,\sigma}\geq1$.
\end{proof}

\begin{figure}[ht]
\centering
\begin{tikzpicture}[>=Stealth, every node/.style={inner sep=2pt}]
 \node (fib) at (-5.8,0) {$\pi^{-1}(\sigma)$};
 \node (col) at (-1.8,0) {$\operatorname{Col}^{\Omega}_N(D,\sigma)$};
 \node (hom) at (3.45,0)
 {$\operatorname{Hom}_A(M_u(D,\tau\circ\sigma;A),N)$};
 \node (all) at (-1.8,1.5) {$\operatorname{Col}_{\operatorname{As}_{\Omega}(N)}(D)$};
 \node (lab) at (3.45,1.5) {$\Omega^{\pi_0(D)}$};
 \draw[->] (all)--node[above] {$\pi$}(lab);
 \draw[->] (fib)--node[above left,font=\small] {inclusion}(all);
 \draw[<->] (fib)--node[below,yshift=-2pt,font=\small] {fiber bijection}(col);
 \draw[<->] (col)--node[above,yshift=2pt,font=\small] {represented by}(hom);
\end{tikzpicture}
\caption{The associated-biquandle coloring set is a disjoint union of fibers
indexed by component labelings.
Each affine fiber is represented by its deformed Alexander module.}
\label{fig:associated-coloring-fibers}
\end{figure}

\begin{example}[Information retained by the palette-size weight]
\label{ex:strict-color-profile}
The variable $\alpha$ in the color-profile polynomial retains
information that can be lost even when the entire unweighted
distribution of fiber sizes is known. We demonstrate this using
$A=\mathbb F_3[C_2]$, where $C_2=\langle g\mid g^2=1\rangle$,
with $u=1$ and the homomorphism
$\lambda:A\to\mathbb F_3$ determined by $\lambda(g)=2$.
Take $\Omega=\{0,1\}$, $\tau(0)=1$, $\tau(1)=g$, and regard
$X=\mathbb F_3$ as an $A$-module through $\lambda$.
The induced operations are
$x\underline{\triangleright}_0y=x$,
$x\underline{\triangleright}_1y=2x+2y$, and
$x\overline{\triangleright}_iy=x$ for $i=0,1$.

Let $L_1$ be the closure of $\beta=\sigma_1^4$ on two strands,
and let $L_2$ be the closure of
$\gamma=\sigma_1^3\sigma_2^2$ on three strands, with all strands
oriented downwards and the crossing convention of
\eqref{eq:braid-transitions}. Both closures have two components.
For $L_1$, order the components by the two top strands.
For $L_2$, the braid permutation is $(12)$, so the first two top
strands belong to the first component and the third belongs to
the second. For a component labeling $(i,j)\in\Omega^2$, put
$a=\lambda(\tau(i))$ and $b=\lambda(\tau(j))$.

At a positive crossing, the specialized transition matrix is
$T(h)=\left[\begin{smallmatrix}0&1\\h&1-h\end{smallmatrix}\right]$,
where $h$ is the coefficient of the overstrand, which enters at
the upper right. Let $E_{12}(h)$ and $E_{23}(h)$ be its
embeddings in the indicated adjacent coordinates of
$\mathbb F_3^3$. The overstrand coefficients for $\beta$, read
from top to bottom, are $(b,a,b,a)$.
For $\gamma$, the top coefficients are $(a,a,b)$, and the
overstrand coefficients are $(a,a,a,b,a)$.
Thus the complete transition matrices are
\[
T_\beta(a,b)=\bigl(T(a)T(b)\bigr)^2,\qquad
T_\gamma(a,b)=E_{23}(a)E_{23}(b)E_{12}(a)^3.
\]
The rightmost factor acts first in each product.
Here the required products can be evaluated using only two
$2\times2$ matrices. Put
\[
 P=T(1)=\begin{bmatrix}0&1\\1&0\end{bmatrix},\qquad
 Q=T(2)=\begin{bmatrix}0&1\\2&2\end{bmatrix}.
\]
All the following entries are in $\mathbb F_3$. Multiplication gives
\[
 \begin{gathered}
 P^2=I_2,\qquad
 Q^2=\begin{bmatrix}2&2\\1&0\end{bmatrix},\qquad
 Q^3=I_2,\\
 PQ=\begin{bmatrix}2&2\\0&1\end{bmatrix},\qquad
 QP=\begin{bmatrix}1&0\\2&2\end{bmatrix},\qquad
 (PQ)^2=(QP)^2=I_2.
 \end{gathered}
\]
For instance, the upper-right entry of $(PQ)^2$ is
$2\cdot2+2\cdot1=6=0$, and both diagonal entries are $1$;
the lower-left entry of $(QP)^2$ is likewise $6=0$.
It follows that $T_\beta$ equals $I_2$ for the first three
labelings $00,01,10$, and equals $Q^4=Q$ for $11$.

Write $P_{12}$ for the $3\times3$ embedding of $P$ in coordinates
$1,2$, and use analogous subscripts for the other embeddings.
The three initial crossings of $\gamma$ have product
$E_{12}(a)^3=P_{12}$ when $a=1$ and $I_3$ when $a=2$.
Consequently the four complete products for $\gamma$ are
\[
 \begin{aligned}
 T_\gamma(1,1)&=P_{12}
 =\begin{bmatrix}0&1&0\\1&0&0\\0&0&1\end{bmatrix},\\[2pt]
 T_\gamma(1,2)&=(PQ)_{23}P_{12}
 =\begin{bmatrix}1&0&0\\0&2&2\\0&0&1\end{bmatrix}
  \begin{bmatrix}0&1&0\\1&0&0\\0&0&1\end{bmatrix}
 =\begin{bmatrix}0&1&0\\2&0&2\\0&0&1\end{bmatrix},\\[2pt]
 T_\gamma(2,1)&=(QP)_{23}
 =\begin{bmatrix}1&0&0\\0&1&0\\0&2&2\end{bmatrix},\qquad
 T_\gamma(2,2)=(Q^2)_{23}
 =\begin{bmatrix}1&0&0\\0&2&2\\0&1&0\end{bmatrix}.
 \end{aligned}
\]
Subtracting the relevant identity matrix gives every entry in the
closure-matrix table below.

By Lemma~\ref{lem:braid-closure}, eliminating the internal semiarcs
leaves the transposed closure matrix $(T-I)^{\mathsf T}$.
Consequently the specialized module dimension is the nullity of
$T-I$, and each vector in $\ker(T-I)$ extends uniquely to a
coloring of the closed braid. Direct multiplication in
$\mathbb F_3$ gives the following complete calculation:
\begin{equation}\label{eq:strict-profile-closure-matrices}
\renewcommand{\arraystretch}{1.7}
\begin{array}{c|c|c|cc}
(i,j)&T_\beta-I_2&T_\gamma-I_3&
d_{\lambda,(i,j)}(L_1)&d_{\lambda,(i,j)}(L_2)\\ \hline
(0,0)&0_{2\times2}&
\left[\begin{smallmatrix}2&1&0\\1&2&0\\0&0&0\end{smallmatrix}\right]
&2&2\\
(0,1)&0_{2\times2}&
\left[\begin{smallmatrix}2&1&0\\2&2&2\\0&0&0\end{smallmatrix}\right]
&2&1\\
(1,0)&0_{2\times2}&
\left[\begin{smallmatrix}0&0&0\\0&0&0\\0&2&1\end{smallmatrix}\right]
&2&2\\
(1,1)&
\left[\begin{smallmatrix}2&1\\2&1\end{smallmatrix}\right]&
\left[\begin{smallmatrix}0&0&0\\0&1&2\\0&1&2\end{smallmatrix}\right]
&1&2
\end{array}
\end{equation}
The ranks in the second column are $0,0,0,1$.
Those in the third column are $1,2,1,1$: in the second row of
the table, the two nonzero rows are independent, whereas in
each of the other cases the nonzero rows span a line.
More explicitly, the rank-two case has the nonzero minor
$\det\left[\begin{smallmatrix}2&1\\2&2\end{smallmatrix}\right]
=4-2=2\ne0$ in $\mathbb F_3$.
The coloring vectors can also be read off without using any rank
calculation. If $s,t$ range independently over $\mathbb F_3$,
the full kernels of the closure matrices are
\[
 \renewcommand{\arraystretch}{1.35}
 \begin{array}{c|c|c}
 (i,j)&\ker(T_\beta-I_2)&\ker(T_\gamma-I_3)\\\hline
 (0,0)&\{(s,t)\}&\{(s,s,t)\}\\
 (0,1)&\{(s,t)\}&\{(t,t,t)\}\\
 (1,0)&\{(s,t)\}&\{(s,t,t)\}\\
 (1,1)&\{(t,t)\}&\{(s,t,t)\}
 \end{array}.
\]
For example, in the $(0,1)$ case for $L_2$ the equations are
$2x+y=0$ and $2x+2y+2z=0$. The first says $y=x$,
and substitution in the second gives $z=x$. In the remaining
$L_2$ cases the only independent equations are, respectively,
$y=x$, $z=y$, and $z=y$. For $L_1$, only the $(1,1)$ case
imposes an equation, namely $2x+y=0$, or $y=x$.
Thus the four fiber cardinalities are $9,9,9,3$ for $L_1$ and
$9,3,9,9$ for $L_2$, with all colors on internal semiarcs
recovered uniquely by the displayed crossing transitions.

The labelings $(0,0)$ and $(1,1)$ use one palette element,
whereas $(0,1)$ and $(1,0)$ use two. The Betti profiles are
therefore
\begin{equation}\label{eq:strict-betti-profiles}
\begin{aligned}
\mathcal B_{\lambda}^{\Omega}(L_1;\alpha,t)
 &=\alpha(t^2+t)+2\alpha^2t^2,\qquad
\mathcal B_{\lambda}^{\Omega}(L_2;\alpha,t)
 =2\alpha t^2+\alpha^2(t+t^2).
\end{aligned}
\end{equation}
Since a fiber of dimension $d$ has $3^d$ elements, the
corresponding color-profile polynomials are
\begin{equation}\label{eq:strict-color-profiles}
\begin{aligned}
\Phi_X^\Omega(L_1;\alpha,\xi)
 &=\alpha(\xi^9+\xi^3)+2\alpha^2\xi^9,\qquad
\Phi_X^\Omega(L_2;\alpha,\xi)
 =2\alpha\xi^9+\alpha^2(\xi^3+\xi^9).
\end{aligned}
\end{equation}
Both links have the same unweighted fiber-size polynomial:
$\Phi_X^\Omega(L_1;1,\xi)=
\Phi_X^\Omega(L_2;1,\xi)=\xi^3+3\xi^9$.
In particular, their associated biquandle coloring counts both
equal $3+3\cdot9=30$.
Nevertheless, their $\alpha$-weighted profiles differ.
For $L_1$, the unique fiber of size $3$ comes from a labeling
using one palette element; for $L_2$, it comes from a labeling
using both palette elements. Thus retaining the number of
palette elements used gives a strict refinement of the
unweighted fiber-size distribution for this fixed affine
multibiquandle.
\end{example}

\section{Final remarks}
\label{sec:series-i-conclusion}

The universal coloring clarifies how the affine module encodes
diagrammatic coloring data. For a fixed component labeling, the
finitely presented module represents colorings with values in every
$A$-module, and each coloring is obtained by a unique homomorphism
from the representing module. Keeping the universal coloring explicit
also specifies the compatibility required when either the diagram or
the target module changes. The module and its universal coloring thus
organize the individual coloring sets into a single functorial
construction.

The two proofs of Reidemeister invariance emphasize complementary
aspects of this construction. Naturality of the coloring extension
maps determines the compatible module isomorphism for each specified
move; its uniqueness is relative to that prescribed coloring
transport. The matrix proof in Appendix~\ref{app:matrix-reductions}
describes the same invariance through local row and column operations
that preserve the exterior presentation. Both arguments work over
commutative group rings of abelian groups and use only the designated
coefficient units. In particular, no invertibility assumption on
$u-g$ is needed.

The coefficient and normalization results clarify the role of the
parameters. Scalar extension preserves the module presentation and
its universal coloring without a flatness hypothesis, so the Laurent
parameter organizes the choices of the common unit in a single
presentation. For classical links,
Theorem~\ref{thm:all-color-reduction} identifies the common-unit theory
with its normalized form after simultaneously rescaling the component
labels. The rescaled labels belong to the full unit group of the
coefficient ring and need not remain in the originally chosen subgroup
$G$. In the normalized module of a nonempty link, the augmentation
$[x]\mapsto1$ splits by sending $1$ to any chosen arc generator.
The diagonal normalization isomorphism transports the resulting free
rank-one summand to the original module. These observations distinguish
the choice of coefficient coordinates from the fixed component-label
data.

The profile construction makes a further distinction between retaining
labeling fibers and retaining only numerical summaries of those fibers.
The associated biquandle keeps the component label as a second
coordinate, and its coloring set decomposes accordingly. In the affine
case, each fiber is represented by the corresponding fixed-label
module. For a finite affine target and a finite palette, the weighted
profile records the joint distribution of fiber sizes and numbers of
palette labels used. Specializing $\alpha=1$ forgets the latter information,
while summing the fiber cardinalities gives the associated-biquandle
coloring count. Example~\ref{ex:strict-color-profile} shows explicitly
that forgetting the palette-size weight can lose information even
when the entire unweighted fiber-size distribution is retained:
the unique fiber of size $3$ uses one palette label for $L_1$ and two
for $L_2$. Thus the strict refinement established there concerns the
weighted and unweighted profiles for the stated affine target.

The representing module, its fixed-label coloring fibers, and their
finite profiles therefore serve different purposes. The module
retains the coloring functor for arbitrary module targets, whereas
the profiles give finite numerical invariants after choosing a target
and collecting the component labelings. The example isolates the loss caused by forgetting the palette-size
weight. The relation with
classical Alexander modules requires a separate comparison of module
presentations; Part~II of this series develops that comparison through
incidence assembly over the original coefficient ring.

\appendix
\section{A direct matrix proof of Reidemeister invariance}
\label{sec:matrix-reductions}
\label{app:matrix-reductions}

This appendix proves Theorem~\ref{thm:Gu-invariance} directly from the
crossing presentation, by elementary matrix operations. Throughout,
$A=R[G]$, where $R$ is a commutative ring with identity and $G$ is
abelian, and $u\in A^\times$. Rows index semiarc generators and columns
index relations. The equivalence $\sim$ is the presentation equivalence
of Theorem~\ref{thm-mod-pres}. Unit row and column scalings are allowed
by Proposition~\ref{prop-eq-mat}. None of the reductions below requires
$u-g$ to be invertible for any $g\in G$.

\subsection{Boundary splitting and unit elimination}

For clarity, an elementary row operation $U$ induces the cokernel map
$[v]\mapsto[Uv]$, whereas an elementary column operation leaves the
relation submodule unchanged. Thus a row operation fixes a designated
generator class when it fixes its basis vector. This is distinct from
requiring its entire matrix row to remain unchanged. All row additions
below use an internal pivot row as their source; consequently they fix
the exterior generator basis vectors and every exterior relation column.

\begin{lemma}[Relative elimination of internal semiarcs]
\label{lem:relative-local-elimination}
Let $\Delta$ be a move disk. Every occurrence of a semiarc at
$\partial\Delta$ may be replaced by a separate boundary generator and a
relation identifying it with the original semiarc. This operation
preserves the presentation, including all exterior generators and
exterior relation columns.

After this splitting, suppose the local relations have the form
\begin{align*}
r_\nu&=-z_\nu+\ell_\nu(\mathbf x,z_1,\ldots,z_{\nu-1})
&& (1\leq\nu\leq s),\\
q_j&=-y_j+m_j(\mathbf x,z_1,\ldots,z_s)
&& (1\leq j\leq m),
\end{align*}
where $\mathbf x=(x_1,\ldots,x_m)^{\mathsf T}$ and
$\mathbf y=(y_1,\ldots,y_m)^{\mathsf T}$ are boundary generators,
$z_1,\ldots,z_s$ are internal generators, and all expressions are
$A$-linear. Successive substitution of the $r_\nu$ gives a boundary map
$\mathbf y=T\mathbf x$. Then the local block reduces to
\begin{equation}
\label{eq:appendix-graph-normal-form}
\operatorname{Gr}(T)\oplus I_s,
\qquad
\operatorname{Gr}(T)=
\begin{bmatrix}-T^{\mathsf T}\\ I_m\end{bmatrix}.
\end{equation}
The reduction fixes the exterior generator classes, all boundary
identification relations, and every exterior relation column.
\end{lemma}

\begin{proof}
For a boundary occurrence of $x$, adjoin a new generator
$x^{\partial}$ and the relation $x^{\partial}-x$. This follows from
identity stabilization followed by a row addition. In each local
relation containing $c x$, add $c(x^{\partial}-x)$ to replace that
occurrence by $c x^{\partial}$. No exterior relation changes. Use a
separate copy for each boundary occurrence, even when several belong to
the same global semiarc. A crossingless local strand with boundary
copies $x^{\partial},y^{\partial}$ is recorded by the continuity
relation $y^{\partial}-x^{\partial}$. If both copies are identified
with one global generator $x$, this relation is
$(y^{\partial}-x)-(x^{\partial}-x)$. It may therefore be adjoined by
adding a zero column and then the indicated identification columns;
the inverse operations remove it. Thus the graph presentation of a
straight strand also follows from the allowed presentation moves.

Suppose first that the coefficient of an internal generator $z$ in a
relation $r$ is a unit $v$. In every other local column $c$ having
$z$-coefficient $a$, apply $c\leftarrow c-av^{-1}r$. The $z$-row now
vanishes outside $r$. If the coefficient of another generator $w$ in
$r$ is $b$, apply $R_w\leftarrow R_w-bv^{-1}R_z$. This isolates the
pivot $v$. Scale it to $1$ and delete its row and column by inverse
identity stabilization. Since $z$ is internal, its row vanishes in the
boundary identification relations and in every exterior relation.
These row additions therefore leave all those columns unchanged and
fix every exterior and boundary basis vector.

Apply this procedure successively to $(z_\nu,r_\nu)$, whose pivot is
$-1$. Equivalently, add $a r_\nu$ to each later column with
$z_\nu$-coefficient $a$. This is exactly the substitution
$z_\nu=\ell_\nu$. The remaining columns are
$-y_j+\sum_iT_{ji}x_i$. Multiplying them by $-1$ gives
\eqref{eq:appendix-graph-normal-form}; keeping the isolated pivots gives
the identity summand. Undoing the boundary identifications glues the
reduced local presentation back to the unchanged exterior. A zero
relation column, if produced, may be deleted by move~(3) of
Theorem~\ref{thm-mod-pres}.
\end{proof}

\subsection{The Reidemeister-I and Reidemeister-II reductions}

\begin{lemma}[Matrix reduction for R1]
\label{lem-tt-pm-1}
An oriented R1 move preserves the presentation-equivalence class of
$\mathcal A_u(D,\sigma)$. The reduction fixes the exterior generator
classes and every exterior relation column.
\end{lemma}

\begin{proof}
Let $g$ be the component label, let $a,a'$ be the two boundary
generators, and let $b$ be the loop generator. Retain $a$ and $a'$ as
distinct generators until the boundary identifications are restored.
Since both diagonal affine operations send $b$ to $ub$, the two
source-loop readings give the relations $a-ub$ and $a'-ub$.
Their difference is $a'-a$. For the opposite-loop readings the relations
are $-b+ga+(u-g)a'$ and $-b+ua'$, up to their order and signs.
Their difference is $g(a-a')$. Scaling by $-g^{-1}$ gives $a'-a$,
and subtracting $u(a'-a)$ from the second relation gives $ua-b$.
The negative-crossing readings interchange the source slots and give
these same two cases.

Thus, in row order $(a,a',b)$, the two possible local normal forms are
\[
K_{\mathrm{out}}^{\partial}=
\begin{bmatrix}1&-1\\0&1\\-u&0\end{bmatrix},
\qquad
K_{\mathrm{loop}}^{\partial}=
\begin{bmatrix}u&-1\\0&1\\-1&0\end{bmatrix}.
\]
Apply $R_a\leftarrow R_a+u^{-1}R_b$ in the first matrix and
$R_a\leftarrow R_a+uR_b$ in the second. Their first columns become
$-u b$ and $-b$, respectively, while their second columns remain
$a'-a$. Scaling and deleting the loop pivot leaves
$\operatorname{Gr}(I_1)$, exactly the boundary-split straight strand.

If the two exterior occurrences already represent the same global
generator $a$, the corresponding two-by-two blocks are
\[
K_{\mathrm{out}}=
\begin{bmatrix}1&1\\-u&-u\end{bmatrix},
\qquad
K_{\mathrm{loop}}=
\begin{bmatrix}u&u\\-1&-1\end{bmatrix}.
\]
Here $C_2\leftarrow C_2-C_1$ produces a zero relation column. The
same row addition isolates the loop pivot; delete that pivot and the
zero column. The remaining generator is the straight-strand generator.
All pivot rows are internal, so the exterior presentation is fixed.
Reversing these operations handles the inverse R1 moves.
\end{proof}

\begin{lemma}[Matrix reduction for R2]
\label{lem-tt-pm-2}
Each of the two R2 generators in Figure~\ref{fig-dibiqdl-rm} preserves
the presentation-equivalence class of $\mathcal A_u(D,\sigma)$,
relative to the exterior generator classes and relation columns.
\end{lemma}

\begin{proof}
Write $h$ for the over-component label. In source-slot coordinates the
two crossing relations say
\((p,q)=H_h(x,y)\) and \(H_h(x,y)=(uy,hx+(u-h)y)\).
This coordinate order is the source-slot order, not the top-to-bottom
order of a downward braid. The inverse is
$H_h^{-1}(p,q)=((u^{-1}-h^{-1})p+h^{-1}q,u^{-1}p)$; only $u$ and $h$
are inverted.

For the first R2 generator, let $(x,y)$ and $(x',y')$ be the two
exterior source-slot pairs and $(p,q)$ the internal pair. Reading the
two crossings gives the columns
\begin{align*}
c_1&=-p+uy,&c_2&=-q+hx+(u-h)y,\\
c_3&=-p+uy',&c_4&=-q+hx'+(u-h)y'.
\end{align*}
In row order $(x,y,x',y',p,q)$ their matrix is
\[
P_{2a}=
\begin{bmatrix}
0&h&0&0\\
u&u-h&0&0\\
0&0&0&h\\
0&0&u&u-h\\
-1&0&-1&0\\
0&-1&0&-1
\end{bmatrix}.
\]
Perform, in the indicated order,
\[
C_3\leftarrow u^{-1}(C_3-C_1),\qquad
C_4\leftarrow h^{-1}\bigl(C_4-C_2-(u-h)C_3\bigr).
\]
The new third and fourth columns are $y'-y$ and $x'-x$. Now apply
\[
R_y\leftarrow R_y+uR_p,\qquad
R_x\leftarrow R_x+hR_q,\qquad
R_y\leftarrow R_y+(u-h)R_q.
\]
The first two columns become the isolated pivots $-p$ and $-q$.
Permuting the last two columns, scaling the pivots, and deleting them
leaves $\operatorname{Gr}(I_2)$.

For the second R2 generator, the internal pair occupies the source
slots at both crossings. Denote it by $(v,w)$ and denote the two
exterior pairs by $(p,q)$ and $(p',q')$. The columns are
\begin{align*}
d_1&=-p+uw,&d_2&=-q+hv+(u-h)w,\\
d_3&=-p'+uw,&d_4&=-q'+hv+(u-h)w.
\end{align*}
With row order $(p,q,p',q',v,w)$ the local matrix is
\[
P_{2b}=
\begin{bmatrix}
-1&0&0&0\\
0&-1&0&0\\
0&0&-1&0\\
0&0&0&-1\\
0&h&0&h\\
u&u-h&u&u-h
\end{bmatrix}.
\]
First apply $D_3\leftarrow D_3-D_1$ and
$D_4\leftarrow D_4-D_2$. These columns become $p-p'$ and $q-q'$.
Next apply $D_2\leftarrow D_2-(u-h)u^{-1}D_1$, giving
$d_2=-q+(u-h)u^{-1}p+hv$. The row additions
\[
R_p\leftarrow R_p+u^{-1}R_w,\qquad
R_p\leftarrow R_p-(u-h)u^{-1}h^{-1}R_v,\qquad
R_q\leftarrow R_q+h^{-1}R_v
\]
isolate the pivots $u$ in the $w$-row and $h$ in the $v$-row.
After scaling them and changing the signs of the two boundary columns,
this matrix also becomes $\operatorname{Gr}(I_2)\oplus I_2$.
Removing the pivots leaves precisely the two straight-strand relations.

The rows used as sources are internal, and hence are zero in all
boundary identification and exterior relation columns. The reductions
therefore glue to the fixed exterior by
Lemma~\ref{lem:relative-local-elimination}. They also apply if the two
strands belong to the same component or their labels coincide.
\end{proof}

\subsection{The Reidemeister-III reduction}

\begin{lemma}[Matrix reduction for R3]
\label{lem-tt-pm-3}
The oriented R3 generator $\Omega_{3b}$ preserves the
presentation-equivalence class of $\mathcal A_u(D,\sigma)$, relative
to the exterior generator classes and relation columns.
\end{lemma}

\begin{proof}
Use the downward braid form of $\Omega_{3b}$, whose three crossings
are positive. The two sides are the words $121$ and $212$, read from
top to bottom. At one crossing let $(X,Y)$ be the actual upper
generators and $(P,Q)$ the actual lower generators, ordered from left
to right. The source-slot relations are
$-Q+hX+(u-h)P$ and $-Y+uP$. Multiplying the second by $-u^{-1}$ and
adding $(u-h)$ times this new column to the first gives
\(-P+u^{-1}Y\) and \(-Q+hX+(1-hu^{-1})Y\).
Thus the downward transition is exactly the block of
\eqref{eq:braid-transitions}:
$T_h^+(u)=\left[\begin{smallmatrix}0&u^{-1}\\h&1-hu^{-1}\end{smallmatrix}\right]$.
This preliminary reduction uses only a unit column scaling and a
column addition.

Put $a=u^{-1}$ and $b_h=1-ah$. Let $(x,y,z)$ be the upper boundary
generators with component labels $(i,j,k)$, and let $(x',y',z')$
denote the lower boundary generators on those same components. Thus
the spatial order of the lower generators is $(z',y',x')$.
The label orders along $121$ are
$(i,j,k)\to(j,i,k)\to(j,k,i)\to(k,j,i)$, whereas along $212$
they are $(i,j,k)\to(i,k,j)\to(k,i,j)\to(k,j,i)$.
Consequently the successive over-labels are $(j,k,k)$ on the left
and $(k,k,j)$ on the right. Introduce three internal generators
$p,q,r$ separately on each side.

On the $121$ side, the six columns, reordered to put the internal
pivots first, are
\begin{align*}
c_1&=-p+ay,&c_2&=-q+jx+b_jy,&c_3&=-r+az,\\
c_4&=-x'+kq+b_kz,&c_5&=-y'+kp+b_kr,&c_6&=-z'+ar.
\end{align*}
On the $212$ side, they are
\begin{align*}
d_1&=-p+az,&d_2&=-q+ky+b_kz,&d_3&=-r+kx+b_kp,\\
d_4&=-x'+jr+b_jq,&d_5&=-y'+aq,&d_6&=-z'+ap.
\end{align*}
In the common row order $(x,y,z,x',y',z',p,q,r)$, the two matrices are
\[
P_{3L}=\begin{bmatrix}
0&j&0&0&0&0\\
a&b_j&0&0&0&0\\
0&0&a&b_k&0&0\\
0&0&0&-1&0&0\\
0&0&0&0&-1&0\\
0&0&0&0&0&-1\\
-1&0&0&0&k&0\\
0&-1&0&k&0&0\\
0&0&-1&0&b_k&a
\end{bmatrix},\qquad
P_{3R}=\begin{bmatrix}
0&0&k&0&0&0\\
0&k&0&0&0&0\\
a&b_k&0&0&0&0\\
0&0&0&-1&0&0\\
0&0&0&0&-1&0\\
0&0&0&0&0&-1\\
-1&0&b_k&0&0&a\\
0&-1&0&b_j&a&0\\
0&0&-1&j&0&0
\end{bmatrix}.
\]
For $P_{3L}$ apply
\[
C_4\leftarrow C_4+kC_2,\qquad
C_5\leftarrow C_5+kC_1+b_kC_3,\qquad
C_6\leftarrow C_6+aC_3.
\]
The last three columns become
\begin{equation}
\label{eq:appendix-r3-boundary-columns}
\begin{aligned}
q_i&=-x'+jkx+kb_jy+b_kz,\\
q_j&=-y'+aky+ab_kz,\qquad q_k=-z'+a^2z.
\end{aligned}
\end{equation}
For $P_{3R}$ first apply $D_3\leftarrow D_3+b_kD_1$, so that
$d_3=-r+kx+ab_kz$. With this updated column apply
\[
D_4\leftarrow D_4+jD_3+b_jD_2,\qquad
D_5\leftarrow D_5+aD_2,\qquad
D_6\leftarrow D_6+aD_1.
\]
The fourth column becomes
$-x'+jkx+kb_jy+b_k(aj+b_j)z=q_i$, since $aj+b_j=1$.
The fifth and sixth become $q_j$ and $q_k$.

It remains to isolate the first three pivots. On the left apply
\[
R_y\leftarrow R_y+aR_p,\quad
R_x\leftarrow R_x+jR_q,\quad
R_y\leftarrow R_y+b_jR_q,\quad
R_z\leftarrow R_z+aR_r.
\]
On the right apply
\[
R_z\leftarrow R_z+aR_p,\quad
R_y\leftarrow R_y+kR_q,\quad
R_z\leftarrow R_z+b_kR_q,\quad
R_x\leftarrow R_x+kR_r,\quad
R_z\leftarrow R_z+ab_kR_r.
\]
In both matrices the first three columns are now $-p,-q,-r$, and
the last three are unchanged from
\eqref{eq:appendix-r3-boundary-columns}. Scale all six columns by
$-1$ and move the boundary columns first. The common matrix is
\begin{equation}
\label{eq:appendix-r3-common-matrix}
\left[
\begin{array}{ccc|ccc}
-jk&0&0&0&0&0\\
-kb_j&-ak&0&0&0&0\\
-b_k&-ab_k&-a^2&0&0&0\\
1&0&0&0&0&0\\
0&1&0&0&0&0\\
0&0&1&0&0&0\\ \hline
0&0&0&1&0&0\\
0&0&0&0&1&0\\
0&0&0&0&0&1
\end{array}
\right]
=\operatorname{Gr}(T)\oplus I_3,
\qquad
T=\begin{bmatrix}
jk&k(1-j/u)&1-k/u\\
0&k/u&(1-k/u)/u\\
0&0&u^{-2}
\end{bmatrix}.
\end{equation}
All three pivots are units $-1$, and all row additions use internal
rows. The boundary identifications and exterior presentation therefore
remain fixed. The equality $jk=kj$ uses the commutativity of $G$;
the label $i$ does not occur because that strand is under both other
strands in this positive braid. The same calculation allows repeated
labels. Reversing the elementary operations proves the inverse move.
\end{proof}

\subsection{Completion of the direct proof}

\begin{theorem}[Matrix form of Reidemeister invariance]
\label{thm:matrix-reidemeister-invariance}
\label{thm-adbq-mod-l}
If $(D,\sigma)$ and $(D',\sigma')$ differ by a colored oriented
Reidemeister move, then
$\mathcal A_u(D,\sigma)\sim\mathcal A_u(D',\sigma')$.
The equivalence fixes the exterior generator classes and all relation
columns supported outside the move disk.
\end{theorem}

\begin{proof}
Lemmas~\ref{lem-tt-pm-1}--\ref{lem-tt-pm-3} give elementary
reductions for the oriented generators
$\Omega_{1a},\Omega_{1c},\Omega_{2c},\Omega_{2d},\Omega_{3b}$,
with the conventions of Figure~\ref{fig-dibiqdl-rm}.
The last is the all-positive braid move just reduced explicitly.
These five moves generate all oriented Reidemeister moves by
\cite[Theorem~1.2]{Polyak}. The generating sequences are performed
within the original move disk. Assigning each intermediate strand its
component label makes every step color-preserving. Planar isotopies
only permute the semiarc rows and crossing columns. Apply the displayed
reductions at each step, reverse them for inverse moves, and glue each
local reduction to the exterior using
Lemma~\ref{lem:relative-local-elimination}. Transitivity of
presentation equivalence proves the assertion.
\end{proof}

\begin{corollary}[Diagram independence by matrix reduction]
\label{cor:module-reidemeister-invariance}
If $(D,\sigma)$ and $(D',\sigma')$ represent the same $G$-colored
oriented link, then
\[
\mathcal A_u(D,\sigma)\sim\mathcal A_u(D',\sigma'),
\qquad
M_u(D,\sigma;A)\cong M_u(D',\sigma';A).
\]
\end{corollary}

\begin{proof}
Choose a colored Reidemeister sequence by
Theorem~\ref{thm-equi-nchr}. Apply
Theorem~\ref{thm:matrix-reidemeister-invariance} to each move and
compose the resulting presentation equivalences. Taking cokernels
gives the stated module isomorphism. A crossingless component supplies
one free generator and no crossing relation. When an R1 move creates
its first crossing, the coincident-boundary calculation in
Lemma~\ref{lem-tt-pm-1} leaves exactly that generator. Thus the proof
also includes diagrams with crossingless components. This establishes
Theorem~\ref{thm:Gu-invariance} without using representability.
\end{proof}

\section*{Acknowledgements}
This work was supported by the Basic Science Research Program through the National Research Foundation of Korea (NRF) funded by the Ministry of Education, Science and Technology (NRF-2019R1F1A1060205) and by the international cooperation program managed by the National Research Foundation of Korea (NRF-2019K2A9A1A06100201). OpenAI’s ChatGPT was used in the preparation of this work to assist with typesetting, English-language editing, computational scripting and certificate generation, literature searches, and manuscript revision. The author takes sole responsibility for all mathematical claims, proofs, computations, references, and the final content of the manuscript.



\begin{thebibliography}{99}
\setlength{\itemsep}{0pt}

\bibitem{EN}
M. Elhamdadi and S. Nelson,
\textit{Quandles: An Introduction to the Algebra of Knots},
Student Mathematical Library 74, American Mathematical Society,
Providence, RI, 2015.

\bibitem{FJK}
R. Fenn, M. Jordan-Santana, and L. H. Kauffman,
Biquandles and virtual links,
\textit{Topology Appl.} \textbf{145} (2004), 157--175.

\bibitem{IIJO}
A. Ishii, M. Iwakiri, Y. Jang, and K. Oshiro,
A $G$-family of quandles and handlebody-knots,
\textit{Illinois J. Math.} \textbf{57} (2013), 817--838.

\bibitem{GAB}
A. Ishii, M. Iwakiri, S. Kamada, J. Kim, S. Matsuzaki, and K. Oshiro,
Cocycles of $G$-Alexander biquandles and $G$-Alexander multiple
conjugation biquandles,
\textit{Topology Appl.} \textbf{301} (2021), 107512.

\bibitem{IN}
A. Ishii and S. Nelson,
Partially multiplicative biquandles and handlebody-knots,
\textit{Contemp. Math.} \textbf{689} (2017), 159--176.

\bibitem{IT}
K. Ishikawa and K. Tanaka,
Quandle colorings vs. biquandle colorings,
\textit{Topology Appl.} \textbf{345} (2024), 108832.

\bibitem{Jo}
D. Joyce,
A classifying invariant of knots: the knot quandle,
\textit{J. Pure Appl. Algebra} \textbf{23} (1982), 37--65.

\bibitem{KLS1}
J. Kim, S. Y. Lee, and M. I. Sheikh,
Presentations of diquandles and diquandle coloring invariants for solid
torus knots and links,
\textit{J. Knot Theory Ramifications} \textbf{33} (2024), 2350102.

\bibitem{KLS2}
J. Kim, S. Y. Lee, and M. I. Sheikh,
Dibiquandle coloring invariants for dichromatic links and associated
di(bi)quandles,
\textit{Topology Appl.} \textbf{371} (2025), 109358.

\bibitem{LeeProjection}
S. Y. Lee,
Projection image polynomial invariants for dichromatic links with
dibiquandle colorings,
\textit{Kyungpook Math. J.} \textbf{65} (2025), no.~4, 693--721,
\url{https://doi.org/10.5666/KMJ.2025.65.4.693}.

\bibitem{LeeSeriesII}
S.~Y.~Lee,
Affine $G$-multibiquandle modules for colored links II:
Incidence assembly and integral Alexander comparison,
companion manuscript, 2026.

\bibitem{LeeSolid}
S. Y. Lee,
Invariants of $n$-colored links in the solid torus $S^1\times D^2$
associated with $n$-multibiquandles,
\textit{J. Knot Theory Ramifications} (2026), 2640002, 25 pp.,
\url{https://doi.org/10.1142/S021821652640002X}.

\bibitem{LS}
S. Y. Lee and M. I. Sheikh,
Diquandles and invariants for oriented dichromatic links,
\textit{J. Knot Theory Ramifications} \textbf{30} (2021), 2150049.

\bibitem{Lic}
W. B. R. Lickorish,
An Introduction to Knot Theory,
Graduate Texts in Mathematics 175, Springer, New York, 1997.

\bibitem{MacLane}
S. Mac Lane,
\textit{Categories for the Working Mathematician},
2nd ed., Graduate Texts in Mathematics 5,
Springer, New York, 1998.

\bibitem{Ma}
S. V. Matveev,
Distributive groupoids in knot theory,
\textit{Mat. Sb. (N.S.)} \textbf{119(161)} (1982), 78--88.

\bibitem{Murao}
T. Murao,
A relationship between multiple conjugation quandle/biquandle colorings,
\textit{Kobe J. Math.} \textbf{36} (2019), no.~1--2, 57--78.

\bibitem{Polyak}
M. Polyak,
Minimal generating sets of Reidemeister moves,
\textit{Quantum Topol.} \textbf{1} (2010), 399--411.

\bibitem{TraldiColor}
L. Traldi,
Multivariate Alexander colorings,
\textit{J. Knot Theory Ramifications} \textbf{27} (2018),
no.~14, 1850076.

\bibitem{Tu}
V. Turaev,
Multi-quandles of topological pairs,
arXiv:2205.00951 [math.GT], 2022.

\end{thebibliography}
\end{document}